\documentclass{amsart}%
\usepackage{amssymb,amsmath,xcolor,graphicx,xspace}
\usepackage[raggedrightboxes]{ragged2e}
\usepackage{amsfonts}
\usepackage{amsmath}
\usepackage{amssymb}
\usepackage{graphicx}
\usepackage[unicode=true]{hyperref}
\usepackage{ulem}
\usepackage{cite}%
\providecommand{\U}[1]{\protect\rule{.1in}{.1in}}
\providecommand{\U}[1]{\protect\rule{.1in}{.1in}}
\DeclareGraphicsExtensions{.pdf,.eps,.ps,.png,.jpg,.jpeg}
\providecommand{\U}[1]{\protect\rule{.1in}{.1in}}
\newtheorem{theorem}{Theorem}
\theoremstyle{plain}
\newtheorem*{acknowledgement}{Acknowledgement}

\newtheorem{corollary}{Corollary}

\newtheorem{definition}{Definition}
\newtheorem{example}{Example}

\newtheorem{lemma}{Lemma}

\newtheorem{proposition}{Proposition}
\newtheorem{remark}{Remark}

\numberwithin{equation}{section}
\begin{document}
\title[Limiting interpolation spaces via extrapolation]{Limiting interpolation spaces via extrapolation}
\author{Sergey V. Astashkin}
\address[S.V. Astashkin]{Samara National Research University, Moskovskoye shosse 34,
443086, Sama\-ra, Russia}
\email{astash56@mail.ru}
\author{Konstantin V. Lykov}
\address[K.V. Lykov]{Image Processing Systems Institute --- Branch of the Federal
Scientific Researh Centre "Crystallography and Photonics" of Russian Academy
of Sciences, Molodogvardejskaya st. 151, 443001 Samara, Russia}
\email{alkv@list.ru}
\author{Mario Milman}
\address[M. Milman]{Instituto Argentino de Matematicas\\
Buenos Aires, Argentina}
\email{mario.milman@gmail.com}
\urladdr{https://sites.google.com/site/mariomilman/}
\thanks{The authors collaboration was partially supported by a grant from the Simons
Foundation ($\backslash$\#207929 to Mario Milman). The first author has been
supported by the Ministry of Education and Science of the Russian Federation
(project 1.470.2016/1.4) and by the RFBR grant 18-01-00414. The second author
has been supported by the RFBR grants 16-41-630676 and 17-01-00138.}
\date{February 21, 2018}

\begin{abstract}
We give a complete characterization of limiting interpolation spa\-ces for the
real method of interpolation using extrapolation theory. For this purpose the
usual tools (e.g., Boyd indices or the boundedness of Hardy type operators)
are not appropriate. Instead, our characterization hinges upon the boundedness
of some simple operators (e.g. $f\mapsto f(t^{2})/t$, or $f\mapsto f(t^{1/2}%
)$) acting on the underlying lattices that are used to control the $K$- and
$J$-functionals. Reiteration formulae, extending Holmstedt's classical
reiteration theorem to limiting spaces, are also proved and characterized in
this fashion. The resulting theory gives a unified roof to a large body of
literature that, using ad-hoc methods, had covered only special cases of the
results obtained here. Applications to Matsaev ideals, Grand Lebesgue spaces,
Bourgain-Brezis-Mironescu-Maz'ya-Shaposhnikova limits, as well as a new vector
valued extrapolation theorems, are provided.

\end{abstract}
\maketitle
\tableofcontents

\section{Foreword}

In this section we take up Halmos' famous maxim \textquotedblleft don't worry
if you do not understand the preliminaries" one step further, as we try to
briefly explain our perspective to the necessarily technical material that follows.

In this paper we give a characterization of the so-called \textquotedblleft
limiting interpolation spaces" via extrapolation theory. We also find sharp
conditions, under which the reiteration formulae associated with the above
spaces hold.

The original motivation comes from the following apparently unrelated results.
On the one hand, there is a paper by Gomez-Milman \cite{GM} where the authors
find a way to extend the classical Lions-Peetre scale of interpolation spaces
to limiting values of the parameters. In particular, in this fashion $LLogL$
and Dini type of spaces appear naturally as limiting spaces of real
interpolation scales. Moreover, the basic results of the Lions-Peetre theory
were extended to this context, including a special \textquotedblleft
reiteration formula"\footnote{In particular, they find an abstract form of an
interpolation theorem of Zygmund. In \cite{AZYG}, Zygmund had shown that if
$T$ is an operator acting on functions defined on a finite measure space, then
if $T$ is of weak types $(1,1),(p_{0},p_{0}),$ for some $p_{0}>1,$ it follows
that $T$ is bounded from $L(LogL)$ to $L^{1}.$ It is easy to give a direct
proof or one can proceed by interpolation followed by extrapolation. Indeed,
the assumptions of Zygmund's theorem imply via, the classical Marcinkiewicz
interpolation theorem, that $T$ is of strong type $(p,p)$ for $1<p<p_{0}$ with
$\left\Vert T\right\Vert _{L^{p}\rightarrow L^{p}}\preceq(p-1)^{-1}$ $\left(
(p-1)^{p-1}\right)  ^{\frac{p_{0}}{p_{0}-1}}\approx(p-1)^{-1}$ as
$p\rightarrow1$ \textbf{(}cf. \cite[Theorem 4.1, page 91]{torchinsky}%
\textbf{).} Therefore, the conclusion of Zygmund's theorem follows by Yano's
extrapolation theorem (cf. \cite{yano}).}. The results of \cite{GM} were
welcomed and sparked a literature devoted to the construction of limiting
spaces in order to complete the Lions-Peetre scale of spaces and obtain the
corresponding reiteration theorems (cf.
\cite{Cobos1,Cobos2,CF-CKU,Cobos-M,Cobos-T,Cobos-S,Cobos-S2,Cobos-K,Cobos-Segurado,Cobos-Segurado2,Cobos-D,Cobos-D2,Cobos-D3}%
, as well as the many references therein). On the other hand, not long after
that, Jawerth-Milman (cf. \cite{JM91} and the references therein) developed a
far-reaching generalization of Yano's theorem. In particular, $LLogL$ spaces,
can be naturally described by the extrapolation methods of \cite{JM91}. For
example, setting
\[
\langle A_{0},A_{1}\rangle_{0,1}^{K}:=\{f\in A_{0}+A_{1}:\,\int_{0}%
^{1}K(t,f;A_{0},A_{1})\frac{dt}{t}<\infty\},
\]
and working on spaces based on a finite measure space,
we have%
\begin{equation}
LLogL=\langle L^{1},L^{\infty}\rangle_{0,1}^{K}=\sum\limits_{p>1}\frac{L^{p}%
}{p-1}. \label{fore0}%
\end{equation}
In extrapolation theory the reiteration formula of \cite{GM}%
\begin{equation}
\langle L^{1},(L^{1},L^{\infty})_{\theta,p_{0}}^{K}\rangle_{0,1}^{K}=\langle
L^{1},L^{\infty}\rangle_{0,1}^{K} \label{fore1}%
\end{equation}
takes the following form: for all $p_{0}>1,$ with norm equivalence, we have
\[
\sum\limits_{p>1}\frac{L^{p}}{p-1}=\sum\limits_{1<p<p_{0}}\frac{L^{p}}{p-1}.
\]
A similar result holds for the so called $\Delta$-method of extrapolation.
These results are valid not only for $L^{p}$ spaces but for general
interpolation scales, and provide, in particular, an explanation to the
results of \cite{GM}. Thus, the connection between limiting spaces,
reiteration and extrapolation was known early on, at least in some special cases.

However, the vast literature on limiting spaces that has developed afterwards
seems to be largely independent from extrapolation theory\footnote{We should
also mention here the important work of Cwikel-Pustylnik \cite{CP} on sharp
forms of the Marcinkiewicz interpolation theorem. Although we shall not
discuss the connection in detail here we should mention that the limiting
spaces in \cite{CP} can be easily seen to be extrapolation spaces as well.}
(cf. the recent survey presented in \cite{Segurado}). The purpose of this
paper is to restore the connection between limiting spaces and extrapolation
theory in full generality.

The limiting interpolation spaces considered in this paper are of the form
$\langle\vec{A}\rangle_{F}^{K}$ (resp. $\langle\vec{A}\rangle_{F}^{J}),$ where
$\vec{A}$ is any ordered pair of Banach spaces, and $F$ is a fixed Banach
lattice of measurable functions on $[0,1]$, that is a parameter for the
corresponding real method used$,$ that controls the usual $K-$ (resp. $J-)$
functionals. In this context we ask: under what conditions on $F$ can we
guarantee that $\langle\vec{A}\rangle_{F}^{K}$ (resp. $\langle\vec{A}%
\rangle_{F}^{J})$ can be represented by an extrapolation functor applied to a
Lions-Peetre scale $\{\vec{A}_{\theta,q}\}$? To answer this question we use
the general class of extrapolation functors introduced in \cite{As2003} and we
impose conditions on $F.$ Surprisingly, a complete answer is provided by the
boundedness of simple operations\footnote{As it will be seen later, the
particular form of these special operators is connected with the form of the
well known Holmstedt reiteration formula (cf. \cite{BL}).} acting on the
functions of $F$. The basic operators in question are $f\mapsto f(t^{2})/t$,
or $f\mapsto f(t^{1/2}),$ and their variants/or iterations. So, for example
the general version of (\ref{fore0}) is given by Theorem \ref{StrExtrAbst(J)}:
If $f\mapsto f(t^{1/2})$ is bounded on $F$ (resp. if $f\mapsto f(t^{2})/t$ is
bounded on $F,$ cf. Theorem \ref{StrExtrAbst}) then
\[
\langle\vec{A}\rangle_{F}^{J}=Ext_{F}\{\vec{A}_{\theta,q}\},
\]
where $Ext_{F}$ is a suitable extrapolation functor generalizing the $\Sigma
$-functor (resp. when $Ext_{F}$ is a generalization of the $\Delta
$-functor)$.$ The corresponding reiteration formula (\ref{fore1}) takes the
following form. If $F$ is an interpolation Banach lattice on $[0,1]$ with
respect to the pair $({L}^{\infty},{L}^{\infty}(1/t)),$ then the boundedness
of $f\mapsto f(t^{1/2})$ on $F$ is equivalent to the following equivalent
statement: for every ordered pair $\vec{A}=(A_{0},A_{1})$ and for every
$\theta\in(0,1)$, $1\leq q\leq\infty,$ we have
\begin{equation}
\vec{A}_{F}^{K}={\langle}A_{0},\vec{A}_{\theta,q}{\rangle}_{F}^{K}.
\label{fore2}%
\end{equation}

The proof of these results requires the use of the deepest parts of real
interpolation theory. In particular, the $K-$divisibility theorem of
Brudnyi-Kruglyak \cite{BK} and some of its consequences or variants: the
strong form of the fundamental lemma (cf. \cite{CJM90} and the references
therein), the fact that suitable pairs of Banach spaces (e.g., $(L^{1}%
,L^{\infty}))$ are $Conv_{0}$-abundant, are combined with the boundedness
assumptions we place on the special operators $f\mapsto f(t^{2})/t$, or
$f\mapsto f(t^{1/2})$ that act on the underlying lattices controlling the $K$
and $J$ functionals.

In our development we uncover applications to Matsaev ideals, generalized
versions of Zygmund's theorem, and using ideas of Pisier we show vector valued
versions of Yano's theorem. We also discuss the connection of extrapolation to
limits of norm inequalities, and show how extrapolation theory leads to slight
extensions of limit theorems due to
Bourgain-Brezis-Mironescu-Maz'ya-Shaposhnikova (cf. \cite{bbm1}, \cite{maz},
\cite{Mlim}, \cite{PO}, \cite{POS}) (cf. Section \ref{sec:peelim}).

\section{Introduction}

Many of the familiar spaces we use in analysis can be described using
interpolation methods (cf. \cite{bs}, \cite{BL}, \cite{BK}, \cite{BB},
\cite{Tr}). In particular, the real and complex methods of interpolation play
a central r\^{o}le in the applications of interpolation theory to analysis.
For a given pair\footnote{In the literature they are usually refered to as a
\textquotedblleft compatible couple of Banach spaces" (cf. \cite{BL}) or
simply \textquotedblleft pairs", the latter will be our prefered
nomenclature.} of Banach spaces $\vec{A}=(A_{0},A_{1}),$ the classical
Lions-Peetre interpolation scales $\vec{A}_{\theta,q}^{K}$ and $\vec
{A}_{\theta,q}^{J}$ (cf. Section \ref{sec:back} below) are defined for
parameters $(\theta,q)\in\lbrack0,1]\times(0,\infty].$ It is customary to
restrict $\theta$ to be in the open interval $(0,1),$ and, when dealing with
Banach spaces, to let $q\in\lbrack1,\infty].$ It will be useful to remind the
reader the reasons for these restrictions. Indeed, it can be easily seen that
for $1\leq q<\infty$ and $\theta=0,$ or $\theta=1,$ we have (cf. \cite[page
168]{BB}):%
\begin{equation}
\vec{A}_{j,q}^{K}=\{0\},j=0,1;\text{ }1\leq q<\infty. \label{nueva'}%
\end{equation}
In the remaining limiting cases: $(\theta,q)\in\{0,1\}\times\{\infty\},$ the
usual definition of $\vec{A}_{\theta,q}^{K}$ still makes sense and gives (cf.
\cite[(1.34) page 300]{bs})%
\begin{equation}
\vec{A}_{0,\infty}^{K}=\tilde{A}_{0},\vec{A}_{1,\infty}^{K}=\tilde{A}_{1},
\label{K1}%
\end{equation}
where $\tilde{A}_{i}$ denotes the corresponding Gagliardo completion of
$A_{i},i=0,1,$ in $A_{0}+A_{1}$ with the norm (see \cite{Ovch})
\[
{\Vert a\Vert}_{\tilde{A}_{i}}:=\sup_{0<t<\infty}t^{-i}K(t,a;\vec{A})=%
\begin{cases}
\lim_{t\rightarrow0}t^{-1}K(t,a;\vec{A})\quad\mbox{if }i=1,\\
\lim_{t\rightarrow\infty}K(t,a;\vec{A})\quad\mbox{if }i=0.
\end{cases}
\]
Moreover, since $\vec{A}_{\theta,q}^{K}=(\tilde{A}_{0},\tilde{A}_{1}%
)_{\theta,q}^{K},$ for $0<\theta<1$, $1\leq q\leq\infty;$ it is customary to
assume that the original given pairs are Gagliardo complete, and $0<\theta<1,$
$1\leq q\leq\infty$.

A similar situation occurs when considering the spaces $\vec{A}_{\theta,q}%
^{J}$, which were initially defined for the same range of values of the
parameters $(\theta,q)$. Peetre \cite[Lemma 1.1]{pee} has shown that if
$A_{0}\cap A_{1}$ is dense in $A_{0}$ then, with norm equivalence, we have%
\begin{equation}
\vec{A}_{0,1}^{J}=A_{0}. \label{j1}%
\end{equation}
Actually, for mutually closed pairs, (\ref{j1}) is essentially an easy
consequence of the strong form of the fundamental lemma of interpolation
theory (cf. Section \ref{sec:peetre} below). Moreover, applying (\ref{j1}) to
the pair $(A_{1},A_{0}),$ it follows that if $A_{0}\cap A_{1}$ is dense in
$A_{1},$%
\begin{equation}
\vec{A}_{1,1}^{J}=A_{1}. \label{j2}%
\end{equation}
In the remaining cases, we have\footnote{In fact, there is a simple connection
between (\ref{nueva}) and (\ref{nueva'}): For example, by \cite[see proof of
Theorem 3.7.1 part 3, page 54]{BL},
\[
\left(  \vec{A}_{0,q}^{J}\right)  ^{\ast}\subset\vec{A}_{1,q^{\prime}}^{\ast
K},
\]
and thus by (\ref{nueva'}), $\left(  \vec{A}_{0,q}^{J}\right)  ^{\ast}=\{0\},$
if $1<q\leq\infty.$}%
\begin{equation}
\vec{A}_{j,q}^{J}=\{0\},j=0,1,1<q\leq\infty. \label{nueva}%
\end{equation}

On the other hand, it is possible to \textbf{modify} the Lions-Peetre
constructions in such a way that, for the limiting values of the parameter
$\theta,$ the resulting spaces are non-trivial. Indeed, such modifications
have turned out to be useful since a number of spaces of interest in analysis
can be identified in this fashion (cf. \cite{GM}, and Sections \ref{sec:back}
and \ref{sec:scha} below). Moreover, the resulting \textquotedblleft limiting
spaces" appear naturally in extrapolation theory \cite{JM91}. Indeed, one of
the main purposes of this paper is to explicitly identify and characterize
limiting spaces as extrapolation spaces, and use these representations to
prove new qualitative results for limiting spaces, including reiteration theorems.

One of the first modifications of the Lions-Peetre scale was recorded in
\cite{GM}, where the case of \textquotedblleft ordered pairs" of spaces was
discussed. Let us say that a pair $\vec{A}$ is ordered\footnote{In this paper
when considering ordered pairs we shall assume, without loss of generality,
that the embedding $A_{1}\subset A_{0}$ has norm $1.$} if $A_{1}\subset
A_{0}.$ Then, for any element $f\in A_{0}$ its $K-$functional\footnote{We
refer to Section \ref{sec:real} for the definition as well as for background
information on interpolation and extrapolation theory.}, $K(t,f;\vec{A}),$
will be constant for $t>1.$ This motivated the following definition (cf.
\cite{GM}):

\begin{definition}
\label{def:order}Let $\vec{A}$ be an ordered pair$,$ and let $0\leq\theta
\leq1,0<q\leq\infty.$ We define%
\[
\langle\vec{A}\rangle_{\theta,q}^{K}=\{f\in A_{0}:\left\Vert f\right\Vert
_{\langle\vec{A}\rangle_{\theta,q}^{K}}<\infty\},
\]
where
\begin{equation}
\left\Vert f\right\Vert _{\langle\vec{A}\rangle_{\theta,q}^{K}}=\left\{
\begin{array}
[c]{cc}%
\{\int_{0}^{1}\left[  s^{-\theta}K(s,f;\vec{A})\right]  ^{q}\frac{ds}%
{s}\}^{1/q}, & q<\infty,\\
\sup_{0<s<1}s^{-\theta}K(s,f;\vec{A}), & q=\infty.
\end{array}
\right.  \label{defX}%
\end{equation}

\end{definition}

As shown in \cite{GM}, this construction gives non-trivial spaces for the
limiting value $\theta=0$ and all $q>0.$ For example, if $\Omega$ is a
probability space then $(L^{1}(\Omega),L^{\infty}(\Omega))$ is an ordered
pair, and we have (cf. \cite{BE}, \cite{GM})%
\begin{equation}
\langle L^{1}(\Omega),L^{\infty}(\Omega)\rangle_{0,1}^{K}=L(LogL)(\Omega
)\text{.} \label{sum0}%
\end{equation}
Another example is provided by the Dini spaces defined on a smooth domain
$\Omega\subset$ $R^{n}$ (cf. \cite{GM}, \cite[Theorem 1, and the paragraph
that follows it, in page 120]{scherer}){,}%
\[
(\mathring{W}_{p}^{1}(\Omega),L^{p}(\Omega))_{0,1}^{K}=\{f:\int_{0}^{1}%
w_{p,f}(s)\frac{ds}{s}<\infty\},
\]
where {$\mathring{W}_{p}^{1}(\Omega)$ is the usual homogeneous Sobolev space
(cf. \cite{bs}) and}%
\[
w_{p,f}(s):=\sup_{\left\vert h\right\vert \leq s}\left\Vert \left(
f(\cdot+h)-f(\cdot)\right)  \chi_{\{x\in\Omega:x+h\in\Omega\}}\right\Vert
_{L^{p}}%
\]
is the $p$-modulus of continuity of $f$.

Moreover, this modification of the Lions-Peetre scale is consistent in the
sense that when the pair $\vec{A}$ is ordered, and the parameters are in the
classical range, $0<\theta<1,0<q\leq\infty,$ we have%
\[
\langle\vec{A}\rangle_{\theta,q}^{K}=\vec{A}_{\theta,q}^{K}.
\]
More precisely, for ordered pairs and $q\ge1$ we have (cf. Section
\ref{sec:normequivalence}, Proposition \ref{Aprop1} below or \cite[Lemma
3]{Mlim}){%
\begin{equation}
\left\Vert a\right\Vert _{\langle\vec{A}\rangle_{\theta,q}^{K}}\leq\left\Vert
a\right\Vert _{\vec{A}_{\theta,q}^{K}}\leq\left[  1+(1-\theta)^{1/q}{\theta
}^{-1/q}\right]  \left\Vert a\right\Vert _{\langle\vec{A}\rangle_{\theta
,q}^{K}}. \label{equivK}%
\end{equation}
}

Furthermore, let us note that when the pair is ordered, $A_{0}+A_{1}=A_{0}.$
Therefore, the $\langle\vec{A}\rangle_{\theta,q}^{K}$ spaces, $0\leq\theta
\leq1,0<q\leq\infty$, \textbf{can also be defined consistently for arbitrary
Banach pairs }simply by keeping (\ref{defX}), and letting%

\begin{equation}
\langle\vec{A}\rangle_{\theta,q}^{K}=\{f\in A_{0}+A_{1}:\left\Vert
f\right\Vert _{\langle\vec{A}\rangle_{\theta,q}^{K}}<\infty\}. \label{def2}%
\end{equation}

The corresponding $\langle\vec{A}\rangle_{\theta,q}^{J}$ spaces are defined in
an analogous fashion (cf. \cite{JM91} and Section \ref{sec:back} below), and
consist of those elements $f\in A_{0}+A_{1}$ that can be represented by
integrals of the form\footnote{Recall that in the classical Lions-Peetre
theory the elements in $J$-spaces are represented by integrals $f=\int%
_{0}^{\infty}u(s)\frac{ds}{s}.$ In the case of ordered pairs we don't lose
information if we consider only integrals {of functions defined} on $(0,1)$.
See Section \ref{sec:back} below.} $f=\int_{0}^{1}u(s)\frac{ds}{s}$ in
$A_{0}+A_{1},$ where $u:(0,1)\rightarrow A_{0}\cap A_{1}$ with $\int_{0}%
^{1}\left[  s^{-\theta}J(s,u(s);\vec{A})\right]  ^{q}\frac{ds}{s}<\infty;$
endowed with the corresponding quotient norm. It is known, and easy to see
(cf. \cite{JM91} and Section \ref{sec:back}), that for \textbf{ordered pairs},
and $0<\theta<1$, $1\leq q\leq\infty,$ these spaces coincide with the
classical $J-$spaces $\vec{A}_{\theta,q}^{J}$. In fact, we have (cf. Section
\ref{sec:normequivalence}, Proposition \ref{Aprop2} below)
\begin{equation}
\left\Vert a\right\Vert _{\vec{A}_{\theta,q}^{J}}\leq\left\Vert a\right\Vert
_{\langle\vec{A}\rangle_{\theta,q}^{J}}\leq\left(  1+\frac{4}{\log2}%
[(1-\theta)q^{\prime}]^{-1/q^{\prime}}\right)  \left\Vert a\right\Vert
_{\vec{A}_{\theta,q}^{J}}, \label{equiv}%
\end{equation}
where $1/q+1/q^{\prime}=1$.

At this point one can pursue different generalizations, e.g., through the use
of more general weights\footnote{In fact, the use of more general weights
cannot be avoided if we wish to extend the fundamental results of the
classical Lions-Peetre theory to limiting spaces (cf. (\ref{equivJ}) below).},
more general norms, etc. Indeed, the construction of limiting spaces has
recently received considerable attention (as a small sample of recent
contributions on limiting spaces we mention
\cite{Cobos1,Cobos2,CF-CKU,Cobos-M,Cobos-T,Cobos-S,Cobos-S2,Cobos-K,Cobos-Segurado,Cobos-Segurado2,Cobos-D,Cobos-D2,Cobos-D3}%
, as well as the references therein). However, it is worthwhile to emphasize
that, in all the papers listed above, only some special classes of limiting
spaces were studied, and only fragments of their corresponding theory were
considered. In contrast, in this paper we strive to give a complete
characterization of the limiting interpolation spaces associated with the real
method. A key point of our approach is the clarification of the connection
between the class of limiting spaces and extrapolation theory. In particular,
we completely describe \textquotedblleft limiting spaces" as extrapolation
spaces. For this purpose the usual tools (e.g., Boyd indices or the
boundedness of Hardy type operators) are not appropriate. Instead, our
characterization hinges upon the boundedness properties of{\ some simple
operators (e.g., $f\mapsto f(t^{2})/t,$ or $f\mapsto f(t^{1/2})$ )} acting on
the underlying lattices that are used to control the $K$- and $J$-functionals.
Reiteration formulae, extending in a meaningful way the classical results of
Holmstedt to limiting spaces, are also proved and characterized in this
fashion. In short, {we gain a new understanding on the limiting spaces for the
classical Lions-Peetre scale that goes well beyond mere considerations of
numerical parameters, and use these insights to formulate new results that
would be more difficult to guess and prove otherwise. Moreover, we believe
that the methods developed here can be used to describe and study limiting
spaces for other scales and methods of interpolation. Still another
consequence of our efforts is that we acquire extrapolation methods to prove
inequalities that involve limiting spaces.}

To motivate and illustrate our development in this paper it will be
instructive to discuss here how general ideas of extrapolation theory apply in
the special case of the $\langle\vec{A}\rangle_{0,1}^{K}$ spaces and, in the
process, allow us to re-formulate some of the fundamental theorems of the
Lions-Peetre classical theory to include limiting spaces. We believe that the
short review that follows could useful to the reader, since the relevant tools
of the extrapolation theory of \cite{JM91} do not seem to have been utilized
in the literature of limiting spaces.

It seems appropriate to start our discussion about limiting spaces with the
very issue of taking limits of norms of real interpolation spaces. Besides the
connection with extrapolation theory, that we shall develop in somewhat more
detail below (cf. Section \ref{sec:apllim}), this topic is of independent
interest in other areas of analysis (e.g., in the theory of Sobolev
inequalities cf. \cite{bbm1}, \cite{leo}, \cite[see page 244]{PO}, \cite{POS},
and the references therein). The basic results are very simple to state. For
definiteness, below we only display the limits of the Lions-Peetre norms when
$\theta\rightarrow0$ . We first normalize the real interpolation spaces
$\vec{A}_{\theta,q}^{K}$ by means of multiplying the corresponding norms by%
\begin{equation}
c_{\theta,q}=(q\theta(1-\theta))^{\frac{1}{q}}, \label{constant}%
\end{equation}
and denote the corresponding normalized spaces by $\vec{A}_{\theta
,q}^{K\blacktriangleleft}.$ In an analogous fashion we normalize the $\vec
{A}_{\theta,q}^{J\blacktriangleleft}$ spaces; here the corresponding constants
are given by $c_{\theta,q}^{J}=(q^{\prime}\theta(1-\theta))^{-1/q^{\prime}}$,
$1/q^{\prime}+1/q=1$ (cf. (\ref{cj}) in Section \ref{sec:real} below). Then,
we get\footnote{Actually the results hold for interpolation functors that are
of exact type $\theta$ (cf. \cite{Mlim}). For example, by (\ref{laj}) applied
to $\rho(t)=t^{\theta},$ the $J-$method $\vec{A}_{\theta,q}%
^{J\blacktriangleleft}$ is exact of order $\theta$, and moreover, we have that
for all $t>0,$%
\[
t^{-\theta}K(t,f;\vec{A})\leq\left\Vert f\right\Vert _{\vec{A}_{\theta
,q}^{J\blacktriangleleft}}\leq\left\Vert f\right\Vert _{A_{0}}^{1-\theta
}\left\Vert f\right\Vert _{A_{1}}^{\theta}.
\]
The corresponding limit in (\ref{lim1}) follows readily. One can also give a
similar proof of the corresponding limit for the $K-$method using \cite[Lemma
2, formula (31), page 244]{Mbmo}.} (cf. \cite{Mlim}) that for $f\in A_{0}\cap
A_{1}$ and $1\leq q<\infty$,
\begin{equation}
\lim_{\theta\rightarrow0}\left\Vert f\right\Vert _{\vec{A}_{\theta
,q}^{K\blacktriangleleft}}=\left\Vert f\right\Vert _{A_{0}}\;\text{ and
}\;\lim_{\theta\rightarrow0}\left\Vert f\right\Vert _{\vec{A}_{\theta
,q}^{J\blacktriangleleft}}=\left\Vert f\right\Vert _{A_{0}}. \label{lim1}%
\end{equation}

On the other hand, if $\vec{A}$ is an ordered pair, then for $f\in A_{1},$ we have

{%
\begin{equation}
\lim_{\theta\rightarrow0}\left\Vert f\right\Vert _{\langle\vec{A}%
\rangle_{\theta,q}^{K}}=\left\Vert f\right\Vert _{\langle\vec{A}\rangle
_{0,q}^{K}}\;\text{ and }\;\Vert a\Vert_{A_{0}}\leq\lim_{\theta\rightarrow
0}\left\Vert a\right\Vert _{\langle\vec{A}\rangle_{\theta,1}^{J}}\leq2\Vert
a\Vert_{A_{0}} \label{lim2}%
\end{equation}
(the first indicated limit follows directly from the definitions while the
inequality for the second limit is an immediate consequence of (\ref{j1}) and
(\ref{equiv}) above).} In Section~\ref{sec:apllim} (cf. Theorem \ref{teoyano})
we show how these calculations can be combined with the strong form of the
fundamental lemma of interpolation theory to prove extrapolation theorems, and
in Section \ref{sec:peelim} we show how extrapolation methods provide an
extension of (\ref{lim1}) and (\ref{lim2}), as well as add another twist to
the Bourgain-Brezis-Mironescu-Maz'ya-Shaposhnikova limit theorems for Sobolev norms.

A cornerstone result in the Lions-Peetre theory is the equivalence between the
$K$- and $J$-interpolation spaces that takes the form:%
\begin{equation}
\vec{A}_{\theta,q}^{K}=\vec{A}_{\theta,q}^{J},0<\theta<1,1\leq q\leq\infty.
\label{k/j}%
\end{equation}
As a consequence of our previous discussion we see that the equivalence
(\ref{k/j}) does not hold in the limiting cases, and we need to go beyond
power weights in order to find the correct results (cf. \cite[page 46]{JM91}
and Section \ref{sec:equivalence} below). For example, the limiting space
$\langle\vec{A}\rangle_{0,1}^{K}$ can be described as a $J-$space, but in this
case the corresponding characterization requires the use of $J-$spaces with
logarithmic weights
\begin{equation}
\langle\vec{A}\rangle_{0,1}^{K}=\{f:f=\int_{0}^{1}u(s)\frac{ds}{s}\text{ with
}\int_{0}^{1}J(s,u(s);\vec{A})\left\vert \log s\right\vert \frac{ds}{s}%
<\infty\}. \label{equivJ}%
\end{equation}
For a far-reaching generalization of this result, and its connections with the
sum extrapolation functor of Jawerth-Milman and the strong form of the
fundamental lemma of interpolation theory, we refer to
\cite[(3.9) page 25]{JM91}, and Section \ref{sec:equivalence} below. These
results are also closely connected with generalizations of the classical
interpolation inequalities of the form%
\begin{equation}
\left\Vert f\right\Vert _{\vec{A}_{\theta,q}^{K}}\leq c(\theta,q)\left\Vert
f\right\Vert _{A_{0}}^{1-\theta}\left\Vert f\right\Vert _{A_{1}}^{\theta
},0<\theta<1,1\leq q\leq\infty. \label{prova}%
\end{equation}
For example, for $\theta=0$ we must modify (\ref{prova}) as follows (cf.
\cite[page 58-59]{JM91} and Example \ref{exa:bre} below)%
\begin{equation}
\left\Vert f\right\Vert _{\langle\vec{A}\rangle_{0,1}^{K}}\leq c\left\Vert
f\right\Vert _{A_{0}}\log\left(  e+\frac{\left\Vert f\right\Vert _{A_{1}}%
}{\left\Vert f\right\Vert _{A_{0}}}\right)  . \label{prova1}%
\end{equation}
Underlying (\ref{prova1}) is the limiting equivalence of the $K$- and
$J$-methods given by (\ref{equivJ}) and the fact that, for quasi-concave
functions\footnote{Indeed, variants of (\ref{prova1}) hold for extrapolation
spaces (cf. \cite{JM91}).} $\rho,$ we have \cite[(5.5.3) page 58]{JM91}%
\begin{equation}
\inf_{t>0}\left\{  \frac{J(t,f;\vec{A})}{\rho(t)}\right\}  =\frac{\left\Vert
f\right\Vert _{A_{0}}}{\rho(\frac{\left\Vert f\right\Vert _{A_{0}}}{\left\Vert
f\right\Vert _{A_{1}}})}. \label{laj}%
\end{equation}

Another key property of the classical Lions-Peetre scale is the reiteration
theorem. In particular, let us consider the following reiteration formula (cf.
\cite{BL}),%
\begin{equation}
(A_{0},\vec{A}_{\theta_{1},q_{1}}^{K})_{\theta,q}^{K}=\vec{A}_{\theta
\theta_{1},q}^{K},\text{ }0<\theta_{1}<1,0<\theta<1. \label{intro0}%
\end{equation}
From our previous discussion it follows that (\ref{intro0}) gives a trivial
result if $\theta=0$. On the other hand, if $\vec{A}$ is an \textbf{ordered
pair}, we have (cf. \cite{GM}, \cite{JM91}) that, for all $\theta_{1}%
\in(0,1),1\leq q\leq\infty,$
\begin{equation}
\langle A_{0},\vec{A}_{\theta_{1},q}^{K}\rangle_{0,1}^{K}=\langle\vec
{A}\rangle_{0,1}^{K}. \label{intr1}%
\end{equation}
Using the modified $K-$spaces, we can formally obtain (\ref{intr1}) by letting
$\theta=0$ in (\ref{intro0}). However, the reader should also notice that
(\ref{intr1}) exhibits the somewhat surprising fact that the resulting space
on the right hand side does not depend on $\theta_{1}!$

Once again, (\ref{intr1}) can be understood using the ${\displaystyle\sum}%
-$method of extrapolation of \cite[page 48]{JM91} (cf. Section
\ref{sec:extrapol} for the definitions). To see this in more detail we recall
that the space $\langle\vec{A}\rangle_{0,1}^{K}$ can be described via the
${\displaystyle\sum}-$method as follows (cf. \cite{JM91}):
For any $q\in\lbrack1,\infty]$,%
\begin{equation}
\langle\vec{A}\rangle_{0,1}^{K}=%
{\displaystyle\sum\limits_{\theta}}
\frac{1}{\theta}\vec{A}_{\theta,q}^{K\blacktriangleleft}. \label{sum1}%
\end{equation}

From this point of view, (\ref{intr1}) simply reflects the following
statement: If the pair $\vec{A}$ is \textbf{ordered} then\footnote{Informally
the idea is that if $\theta_{0}<\theta_{1},$ say, we can split%
\[
{\displaystyle\sum\limits_{\theta<\theta_{1}}}\frac{1}{\theta}(A_{0}%
,A_{1})_{\theta,q}^{K\blacktriangleleft}={\displaystyle\sum\limits_{\theta
<\theta_{0}}}\frac{1}{\theta}(A_{0},A_{1})_{\theta,q}^{K\blacktriangleleft
}+{\displaystyle\sum\limits_{\theta_{0}<\theta<\theta_{1}}}\frac{1}{\theta
}(A_{0},A_{1})_{\theta,q}^{K\blacktriangleleft}%
\]
and then on the right hand side use the fact that we are working with an
ordered pair in order to incorporate the second term to the first term.} for
any $\theta_{i}\in(0,1),i=0,1,$ (cf. \cite{JM91})$,$%
\[%
{\displaystyle\sum\limits_{\theta<\theta_{0}}}
\frac{1}{\theta}\vec{A}_{\theta,q}^{K\blacktriangleleft}=%
{\displaystyle\sum\limits_{\theta<\theta_{1}}}
\frac{1}{\theta}\vec{A}_{\theta,q}^{K\blacktriangleleft}.
\]
Moreover, the characterization (\ref{sum1}) holds even if the pair $\vec{A}$
is \textbf{not} ordered (cf. (\ref{def2}) and \cite{JM91}):%
\begin{align}
{\displaystyle\sum\limits_{\theta}}\frac{1}{\theta}\vec{A}_{\theta
,q}^{K\blacktriangleleft}  &  =\{f:f\in A_{0}+A_{1}\text{ s.t. }%
{\displaystyle\int\limits_{0}^{\text{ }1}}K(s,f;\vec{A})\frac{ds}{s}%
<\infty\}\nonumber\\
&  =\langle\vec{A}\rangle_{0,1}^{K}. \label{sum3}%
\end{align}
Furthermore, in either case we have
\[
\left\Vert f\right\Vert _{%
{\displaystyle\sum\limits_{\theta}}
\frac{1}{\theta}\vec{A} _{\theta,q}^{K\blacktriangleleft}}\asymp
{\displaystyle\int\limits_{0}^{\text{ }1}}K(s,f;\vec{A})\frac{ds}{s}.
\]

\begin{example}
\label{sumexample}Let $(\Omega,\mu)$ be a measure space. Then,%
\[
{\displaystyle\sum\limits_{\theta}} \frac{1}{\theta}(L^{1}(\Omega),L^{\infty
}(\Omega))_{\theta,q}^{K\blacktriangleleft}=L(LogL)(\Omega)+L^{\infty}%
(\Omega).
\]

For details we refer to Section \ref{sec:vector}, Example \ref{sumsum} below.
Note that if $\mu(\Omega)<\infty,$ then $L^{\infty}(\Omega)\subset
L(LogL)(\Omega),$ and we recover (\ref{sum0}).
\end{example}

We can now explain in more detail the results of this paper. First note that
limiting spaces are themselves interpolation spaces so in this paper we shall
be more generally concerned with the problem of representing real
interpolation spaces as extrapolation spaces, as well as with the validity of
suitable versions of the reiteration formula (\ref{intr1}). Since the latter
formula requires a suitable monotonicity condition, in this paper we shall be
working mainly with ordered pairs of Banach spaces. Next, instead of
considering separately the different available methods of extrapolation we
shall formulate our results using two basic general methods of extrapolation
introduced in \cite{As2003} (cf. Section \ref{sec:extrapol} below). These
extrapolation constructions indeed contain all the known methods of
extrapolation (e.g., the $\Sigma$ and $\Delta$ methods, and their
corresponding $p-$variants). As it often happens in mathematics the added
generality clarifies the issues and leads to a streamlined presentation. In
particular, the extrapolation methods of \cite{As2003} (see also \cite{AL2009}
and \cite{AL}) utilize functional parameters and this led us to formulate our
characterizations in terms of specific properties of the functional parameters
involved. This is a particularly felicitous situation for our development
since the traditional tools (e.g., Boyd indices or the boundedness of Hardy
operators) do not seem appropriate for our goals, instead we formulate our
conditions on the behavior of some simple operators acting on the functional parameters.

Let $\vec{A}$ be an ordered pair and let us consider spaces\footnote{See
(\ref{fspace}) below for the definition.} of the form {$\langle\vec{A}%
\rangle_{F}^{K},$ }where $F$ is a suitable functional parameter for the real
method{\ on $[0,1]$}. We ask: under what conditions on $F$ can we represent
{$\langle\vec{A}\rangle_{F}^{K}$} as an extrapolation space? In Theorem
\ref{StrExtrAbst} below we show that if $F$ has the property that the
operation
\begin{equation}
f(t)\rightarrow f(t^{2})/t\text{ is bounded on }F,\label{cond}%
\end{equation}
then the interpolation space {$\langle\vec{A}\rangle_{F}^{K}$}$,$ can be
represented as a generalized extrapolation space of the form%
\begin{equation}
\Vert a\Vert_{{\langle\vec{A}\rangle_{F}^{K}}}\asymp\left\Vert t\cdot\Vert
a\Vert_{\langle\vec{A}\rangle_{\theta(t),q}^{K\blacktriangleleft}}\right\Vert
_{F},\text{ where }\theta(t)=1+\frac{1}{2\log\frac{t}{e}}.\label{repres1}%
\end{equation}
Moreover, Theorem~\ref{TH2} shows that, if $F$ is an interpolation Banach
lattice between the spaces ${L}^{\infty}$ and ${L}^{\infty}(1/t)$, then
condition (\ref{cond}) is equivalent to the following peculiar reiteration
formula%
\[
{\langle}\vec{A}_{\theta,q},A_{1}\rangle_{F}^{K}={\langle\vec{A}\rangle
_{F}^{K}}.
\]
A similar result also holds for limiting interpolation spaces which are close
to the larger "end" of an ordered Banach pair. Namely, the following extension
of (\ref{intr1}) holds
\begin{equation}
{\langle}A_{0},\vec{A}_{\theta,q}\rangle_{F}^{K}={\langle\vec{A}\rangle
_{F}^{K}}\label{condmod}%
\end{equation}
if and only if the operator $Rf(t)=f(t^{1/2})$ is bounded on $F$%
\footnote{Examples include (see Remark\ref{non-interpolation})
\[
\left\Vert f\right\Vert _{F}=\int_{0}^{1}\left\vert f(s)\right\vert \frac
{ds}{s},\left\Vert f\right\Vert _{F}=\sup_{s\in(0,1)}\left\vert
f(s)\right\vert .
\]
} (see Theorem \ref{limiting reiteration}).

If we further assume that $\vec{A}$ is Gagliardo complete then we can show
that, more generally, the representation (\ref{repres1}) holds if we replace
the real method by any family $\{I_{\theta}\}_{\theta\in(0,1)}$ of exact
interpolation functors $I_{\theta}$ with characteristic functions $t^{\theta}
$ (cf. Theorem \ref{StrExtrAbst2} below)$.$ In particular, it follows that if
$F$ satisfies (\ref{cond}) then {$\langle\vec{A}\rangle_{F}^{K}$ } can be also
represented as an extrapolation space for the scale of complex interpolation
spaces $[A_{0},A_{1}]_{\theta}$, in the sense that
\[
\Vert a\Vert_{{\langle\vec{A}\rangle_{F}^{K}}}\asymp\left\Vert t\Vert
a\Vert_{\lbrack A_{0},A_{1}]_{\theta(t)}}\right\Vert _{F}\text{,}%
\]
where $[\cdot,\cdot]_{\theta}$ denotes the complex method of interpolation
(cf. \cite{bs} for background on the complex method of interpolation).

As is well known (cf. \cite{BL}), Holmstedt's formula is an important tool to
prove reiteration theorems in the classical Lions-Peetre theory. In our
setting the relevant Holmstedt formula takes the following form: Suppose that
$F$ satisfies (\ref{cond}), then for every pair ($A_{0},A_{1})$ it holds, with
constants independent of $a\in\langle\vec{A}\rangle_{F}^{K}+\tilde{A}_{1}$,
$s>0$ and $q\in\lbrack1,\infty]$,%
\[
K(s,a;{\langle\vec{A}\rangle_{F}^{K}},\tilde{A}_{1})\asymp K(s,t\cdot\Vert
a\Vert_{\langle\vec{A}\rangle_{\theta(t),q}^{K\blacktriangleleft}}%
;F,{L}^{\infty}(1/t)),
\]
where as above $\theta(t)=1+\frac{1}{2\log(t/e)}$ (cf. Theorem
\ref{ExtrMinLat} below).

Applications to the theory of symmetrically normed operator ideals are given
in Section \ref{sec:scha}. In particular, using an extrapolation description
of the limiting Schatten-von Neumann operator spaces (cf. Theorem~\ref{Ideals}%
), we obtain a generalization of Matsaev's well-known result on the behavior
of the real and the imaginary components of a Volterra operator (cf.
Theorem~\ref{Matcaev}). In Section~\ref{sec:Grand} we show how
Theorem~\ref{StrExtrAbst} can be combined with {easy to prove rearrangement
inequalities }to give a streamlined proof of the Fiorenza-Karadzhov
description of Grand Lebesgue spaces $L^{p)}$, $p>1$, {as extrapolation
spaces} (see Theorem~\ref{Fiorenza-Karadzhov} below). In Section
\ref{sec:vector} we show how extrapolation theory combined with the formula
for the $K-$functional for vector valued $L^{p}$ spaces obtained by Pisier
\cite{PI} allows us to prove an extrapolation theorem for vector valued spaces
that extends Yano's classical extrapolation theorem. In Section \ref{sec:last}
we complete the proofs of some auxiliary results and discuss further results
and applications. We have aimed to make the paper accessible for readers who
may not be familiar with extrapolation theory. We refer to Section
\ref{sec:back} for background information about interpolation and
extrapolation theory.

\begin{acknowledgement}
We are very grateful to the referees for their constructive criticism and very
helpful suggestions to improve the presentation of the paper.
\end{acknowledgement}

\section{Background and auxiliary results\label{sec:back}}

\subsection{Real Interpolation: basic definitions\label{sec:real}}

{We assume that the reader is familiar with elementary real interpolation
theory, e.g., as presented in \cite{BL}. We shall now give a brief summary of
relevant notions in order to fix the notation we use in this paper. }

Let $\vec{A}=(A_{0},A_{1})$ be a Banach pair$.$ The Peetre $K-$functional is
defined for $x\in A_{0}+A_{1},t>0,$ by%
\[
K(t,x;\vec{A})=\inf\{\left\Vert a_{0}\right\Vert _{A_{0}}+t\left\Vert
a_{1}\right\Vert _{A_{1}}:x=a_{0}+a_{1},a_{i}\in A_{i},i=0,1\}.
\]
The corresponding dual construction is the $J-$functional defined for $x\in
A_{0}\cap A_{1},t>0,$ by%
\[
J(t,x;\vec{A})=\max\{\left\Vert x\right\Vert _{A_{0}},t\left\Vert x\right\Vert
_{A_{1}}\}.
\]
Let $0<\theta<1,$ \ and $1\leq q\leq\infty$. Let
\[
\Phi_{\theta,q}(f)=\left\{
\begin{array}
[c]{cc}%
\left\{  \int_{0}^{\infty}\left(  s^{-\theta}|f(s)|\right)  ^{q}\frac{ds}%
{s}\right\}  ^{1/q} & if\text{ }q<\infty\\
\sup_{s>0}\left\{  s^{-\theta}|f(s)|\right\}  & if\text{ }q=\infty.
\end{array}
\right.
\]
and let $\Phi_{\theta,q}$ be the function space of measurable functions on
$(0,\infty)$ such that $\Phi_{\theta,q}(f)<\infty.$ The space $\vec{A}%
_{\theta,q}^{K}$ consists of the elements $x\in A_{0}+A_{1},$ such that $\Vert
x\Vert_{\vec{A}_{\theta,q}^{K}}<\infty,$ where
\[
\Vert x\Vert_{\vec{A}_{\theta,q}^{K}}:=\Phi_{\theta,q}(K(s,x;\vec{A})).
\]
The corresponding $\vec{A}_{\theta,q}^{J}$ spaces consist of all $x\in
A_{0}+A_{1}$ that can be represented as
\begin{equation}
x=\int_{0}^{\infty}u(t)\;\frac{dt}{t}\text{ in }A_{0}+A_{1},
\label{represetation}%
\end{equation}
for some strongly measurable function $u(t)$ defined on $(0,\infty)$ with
values in $A_{0}\cap A_{1},$ and such that $\ \Phi_{\theta,q}(J(s,u(s);\vec
{A}))<\infty$. Then we let%
\[
\Vert x\Vert_{\vec{A}_{\theta,q}^{J}}=\inf\{\Phi_{\theta,q}(J(s,u(s);\vec
{A})):x=\int_{0}^{\infty}u(t)\frac{dt}{t}\}.
\]
For our development in this paper it is convenient to normalize the
interpolation norms as follows. We let%
\begin{equation}
\Vert x\Vert_{\vec{A}_{\theta,q}^{K\blacktriangleleft}}:=(q\theta
(1-\theta))^{\frac{1}{q}}\Vert x\Vert_{\vec{A}_{\theta,q}^{K}}, \label{ck}%
\end{equation}
with the convention that $(q\theta(1-\theta))^{\frac{1}{q}}=1$ when
$q=\infty.$ Then $\vec{A}_{\theta,q}^{K\blacktriangleleft}$ is equal to the
set of elements of $\vec{A}_{\theta,q}^{K}$ provided with the normalized norm
$\Vert\cdot\Vert_{\vec{A}_{\theta,q}^{K\blacktriangleleft}}.$ Likewise, the
spaces $\vec{A}_{\theta,q}^{J\blacktriangleleft}$ are equal, as sets, to the
$\vec{A}_{\theta,q}^{J}$ spaces, but are provided with the
normalized\footnote{The normalization constants given by (\ref{ck}) (resp.
\ref{cj}) are useful when comparing norms and taking limits (cf. (\ref{lim1})
above) and moreover make the corresponding interpolation functors $\vec
{A}_{\theta,q}^{K\blacktriangleleft}$ (resp. $\vec{A}_{\theta,q}%
^{J\blacktriangleleft}$) exact of exponent $\theta$ (cf. footnote (10)
above).} norms $\Vert\cdot\Vert_{_{\vec{A}_{\theta,q}^{J\blacktriangleleft}}}$
given by,
\begin{equation}
\Vert x\Vert_{_{\vec{A}_{\theta,q}^{J\blacktriangleleft}}}=(q^{\prime}%
\theta(1-\theta))^{-1/q^{\prime}}\Vert x\Vert_{_{\vec{A}_{\theta,q}^{J}}},
\label{cj}%
\end{equation}
where if $q=1$ we set $(q^{\prime}\theta(1-\theta))^{-1/q^{\prime}}=1$.

The spaces $\langle\vec{A}\rangle_{\theta,q}^{K}$, $\theta\in\lbrack
0,1],q\in\lbrack1,\infty]$, introduced in Definition \ref{def:order} above,
consist of the elements $x$ in $A_{0}+A_{1}$ with $\Vert x\Vert_{\langle
\vec{A}\rangle_{\theta,q}^{K}}<\infty,$ where
\[
\Vert x\Vert_{\langle\vec{A}\rangle_{\theta,q}^{K}}:=\Phi_{\theta
,q}(K(s,x;\vec{A})\chi_{(0,1)}(s)).
\]
The corresponding $\langle\vec{A}\rangle_{\theta,q}^{J}$ spaces consist of all
the elements $x$ of $A_{0}+A_{1}$ that can be represented by%
\[
x=\int_{0}^{1}u(t)\frac{dt}{t},
\]
for some strongly measurable function $u(t)$ defined on $(0,1),$ with values
in $A_{0}\cap A_{1},$ such that $\ \Phi_{\theta,q}(J(s,u(s);\vec{A}%
)\chi_{(0,1)}(s))<\infty.$ We let
\[
\Vert x\Vert_{\langle\vec{A}\rangle_{\theta,q}^{J}}=\inf\{\Phi_{\theta
,q}(J(s,u(s);\vec{A})\chi_{(0,1)}(s)):x=\int_{0}^{1}u(t)\frac{dt}{t}\}.
\]
Furthermore, in analogy with (\ref{ck}) and (\ref{cj}) we formally define the
spaces $\langle\vec{A}\rangle_{\theta,q}^{K\blacktriangleleft}$ (resp.
$\langle\vec{A}\rangle_{\theta,q}^{J\blacktriangleleft})$ using the
normalizations
\[
\Vert x\Vert_{\langle\vec{A}\rangle_{\theta,q}^{K\blacktriangleleft}%
}:=(q\theta(1-\theta))^{\frac{1}{q}}\Vert x\Vert_{\langle\vec{A}%
\rangle_{\theta,q}^{K}},\text{ resp. }\Vert x\Vert_{\langle\vec{A}%
\rangle_{\theta,q}^{J\blacktriangleleft}}:=(q^{\prime}\theta(1-\theta
))^{-1/q^{\prime}}\Vert x\Vert_{\langle\vec{A}\rangle_{\theta,q}^{J}}.
\]

{More generally, \textit{mutatis mutandi,} we consider the interpolation
spaces $\vec{A}_{F}^{K}$ (resp. $\vec{A}_{F}^{J}$), obtained by replacing
$\Phi_{\theta,q}$ with a Banach lattice $F$ of functions on $((0,\infty
),\frac{ds}{s})$ satisfying the continuous embeddings $L^{\infty}\cap
L^{\infty}(1/s)\subset F\subset L^{\infty}+L^{\infty}(1/s)$ (resp. $L^{1}\cap
L^{1}(1/s)\subset F\subset L^{1}+L^{1}(1/s)$). Here, and in what follows, we
use the notation $\Vert f\Vert_{L^{p}(1/s)}:=\Vert f(s)/s\Vert_{L^{p}}$,
$1\leq p\leq\infty$. Banach lattices with the above properties will be called
\textit{parameters for the $K$-method (resp. $J$-method) of interpolation on}
$(0,\infty)$. More precisely, $\vec{A}_{F}^{K}$ is defined by}
\begin{equation}
\vec{A}_{F}^{K}:=\{x\in A_{0}+A_{1}:\left\Vert x\right\Vert _{\vec{A}_{F}^{K}%
}:=\left\Vert K(s,x;\vec{A})\right\Vert _{F}<\infty\}. \label{fspace}%
\end{equation}
Likewise, we let
\[
\vec{A}_{F}^{J}:=\{x\in A_{0}+A_{1}:\left\Vert x\right\Vert _{\vec{A}_{F}^{J}%
}<\infty\},
\]
where
\begin{align*}
\Vert x\Vert_{\vec{A}_{F}^{J}}  &  :=\inf\{\left\Vert J(s,u(s);\vec
{A})\right\Vert _{F}:x=\int_{0}^{\infty}u(s)\frac{ds}{s};\\
&  {\,}\text{where }u:(0,\infty)\rightarrow A_{0}\cap A_{1}\text{ is strongly
measurable}\}.
\end{align*}

For our purposes in this paper it is important also to consider the
modification of the above spaces that one obtains when the parameter space
$F$\textbf{\ is a Banach lattice of functions on } $(\mathbf{(0,1),\frac
{ds}{s}})$ {satisfying the continuous embeddings $L^{\infty}(1/s)(0,1)\subset
F\subset L^{\infty}(0,1)$ (resp. $L^{1}(1/s)(0,1)\subset F\subset L^{1}%
(0,1)$). We will refer to such Banach lattices as \textit{parameters for the
$K$-method (resp. $J$-method) of interpolation on} $(0,1)$.} The corresponding
definitions are formally the same:%
\[
\langle\vec{A}\rangle_{F}^{K}:=\{x\in A_{0}+A_{1}:\left\Vert x\right\Vert
_{\langle\vec{A}\rangle_{F}^{K}}:=\left\Vert K(s,x;\vec{A})\right\Vert
_{F}<\infty\},
\]
and likewise%
\[
\langle\vec{A}\rangle_{F}^{J}:=\{x\in A_{0}+A_{1}:\left\Vert x\right\Vert
_{\langle\vec{A}\rangle_{F}^{J}}<\infty\},
\]
where%
\begin{align*}
\left\Vert x\right\Vert _{\langle\vec{A}\rangle_{F}^{J}}:=  & \\
\inf_{x=\int_{0}^{1}u(t)\frac{dt}{t}}  &  \{\left\Vert J(s,u(s);\vec
{A})\right\Vert _{F}:\;\;u:(0,1)\rightarrow A_{0}\cap A_{1}\text{ is strongly
measurable}\}.
\end{align*}

We shall say that $\vec{A}$ is an \textit{ordered} Banach pair (or simply an
\textquotedblleft ordered pair") if $A_{1}\subset A_{0}.$ Moreover, to
simplify the discussion we shall always assume that the norm of the embedding
$A_{1}\subset A_{0}$ is less than or equal to one.{\ An increasing
non-negative function $f$ on $[0,1]$ is called \textit{quasi-concave,} if
$f(t)/t$ decreases.} In general, given two Banach spaces $A,B$, we shall write
$A\overset{C}{\subset}B$ to indicate that the norm of the embedding is bounded
by above by $C.$

\subsection{On the equivalence of interpolation norms on ordered
pairs\label{sec:normequivalence}}

In this section we explicitly compare the interpolation constructions that we
introduced in the previous section, on the class of ordered pairs of Banach
spaces. Since the results are important for our purposes in this paper, and do
not seem to be readily available in the literature, we provide full details,
including explicit computation of the constants involved in the norm inequalities.

\begin{proposition}
\label{Aprop1} Let $\vec{A}=(A_{0},A_{1})$ be an ordered pair. Then

(i) for every parameter $F$ for the $K$-method on $(0,\infty)$ we have, with
equivalence of norms,
\[
{\vec{A}_{F}^{K}}=\langle\vec{A}\rangle_{\tilde{F}}^{K},
\]
where $\tilde{F}$ is the sublattice of $F$ consisting of all functions $f$
such that $supp\,f\subset\lbrack0,1]$;

(ii) for all $0<\theta<1$, $1\leq q\leq\infty,$ we have
\[
\left\Vert a\right\Vert _{\langle\vec{A}\rangle_{\theta,q}^{K}}\leq\left\Vert
a\right\Vert _{\vec{A}_{\theta,q}^{K}}\leq\left[  1+(1-\theta)^{1/q}{\theta
}^{-1/q}\right]  \left\Vert a\right\Vert _{\langle\vec{A}\rangle_{\theta
,q}^{K}},
\]
where we let $(1-\theta)^{1/q}{\theta}^{-1/q}=1$ when $q=\infty.$
\end{proposition}

\begin{proof}
(i) It is plain that%
\[
\left\Vert a\right\Vert _{\langle\vec{A}\rangle_{\tilde{F}}^{K}}\leq\left\Vert
a\right\Vert _{{\vec{A}_{F}^{K}}}.
\]
Suppose now that $a\in\langle\vec{A}\rangle_{\tilde{F}}^{K}.$ Since the pair
$\vec{A}$ is ordered, $K(s,a;\vec{A})=\left\Vert a\right\Vert _{A_{0}}$ for
$s\geq1.$ Then we can write%
\begin{align*}
\left\Vert a\right\Vert _{{\langle\vec{A}\rangle_{F}^{K}}}  &  \leq\Vert
K(s,a;\vec{A})\chi_{\lbrack0,1]}(s)\Vert_{\tilde{F}}+\Vert K(s,a;\vec{A}%
)\chi_{\lbrack1,\infty]}(s)\Vert_{{F}}\\
&  =\left\Vert a\right\Vert _{\langle\vec{A}\rangle_{\tilde{F}}^{K}}+\Vert
a\Vert_{A_{0}}\Vert\chi_{\lbrack1,\infty]}\Vert_{{F}}.
\end{align*}

Moreover, since the function $K(s,a;\vec{A})/{s}$ decreases, we have%
\begin{align*}
\left\Vert a\right\Vert _{A_{0}}  &  =\frac{K(1,a;\vec{A})}{1}\frac{\Vert
s\chi_{\lbrack0,1]}(s)\Vert_{{F}}}{\Vert s\chi_{\lbrack0,1]}(s)\Vert_{{F}}%
}\leq\Vert s\chi_{\lbrack0,1]}(s)\Vert_{{F}}^{-1}\left\Vert K(s,a;\vec
{A})\right\Vert _{\tilde{F}}\\
&  =\Vert s\chi_{\lbrack0,1]}(s)\Vert_{{F}}^{-1}\left\Vert a\right\Vert
_{\langle\vec{A}\rangle_{\tilde{F}}^{K}}.
\end{align*}
Collecting estimates we get
\begin{equation}
\left\Vert a\right\Vert _{{\langle\vec{A}\rangle_{F}^{K}}}\leq\left(
1+\Vert\chi_{\lbrack1,\infty]}\Vert_{{F}}\cdot\Vert s\chi_{\lbrack
0,1]}(s)\Vert_{{F}}^{-1}\right)  \left\Vert a\right\Vert _{\langle\vec
{A}\rangle_{\tilde{F}}^{K}}, \label{Sum est}%
\end{equation}
and the desired result follows.

(ii) We apply \eqref{Sum est} with $F=\Phi_{\theta,q}.$ The desired result now
follows computing the $\Phi_{\theta,q}$-norms of the functions $\chi
_{\lbrack1,\infty]}(s)$ and $s\chi_{\lbrack0,1]}(s)$ and inserting the
corresponding results in~\eqref{Sum est}. See also \cite[Lemma 3]{Mlim}.
\end{proof}

Likewise, for the $J$-method we have,

\begin{proposition}
\label{Aprop2} Let $\vec{A}=(A_{0},A_{1})$ be an ordered pair. Then

(i) for every parameter $G$ for the $J$-method on $(0,\infty)$ we have, with
equivalence of norms,%
\[
\vec{A}_{G}^{J}=\langle\vec{A}\rangle_{\tilde{G}}^{J},
\]
where $\tilde{G}$ is the sublattice of $G$ consisting of all functions $f$
such that $supp\,f\subset\lbrack0,1]$;

(ii) for all $0<\theta<1$, $1\leq q\leq\infty,$ we have
\[
\left\Vert a\right\Vert _{\vec{A}_{\theta,q}^{J}}\leq\left\Vert a\right\Vert
_{\langle\vec{A}\rangle_{\theta,q}^{J}}\leq\left(  1+\frac{4}{\log2}%
[(1-\theta)q^{\prime}]^{-1/q^{\prime}}\right)  \left\Vert a\right\Vert
_{\vec{A}_{\theta,q}^{J}},
\]
where $[(1-\theta)q^{\prime}]^{-1/q^{\prime}}=1$ when $q=1$.
\end{proposition}

\begin{proof}
(i) It is obvious that $\langle\vec{A}\rangle_{\tilde{G}}^{J}\subset\vec
{A}_{G}^{J}$ and, furthermore, the norm of the embedding is $1.$ We now prove
the opposite inclusion.

Let $a\in(A_{0},A_{1})_{G}^{J}.$ For any $\varepsilon>0,$ we can select a
representation $a=\int_{0}^{\infty}u(s)\frac{ds}{s}$ such that
\[
\Vert J(s,u(s);\vec{A})\Vert_{G}\leq\Vert a\Vert_{(A_{0},A_{1})_{G}^{J}%
}+\varepsilon.
\]
Let us define
\[
\tilde{u}(s)=\left\{
\begin{array}
[c]{cc}%
u(s) & 0<s<1/2\\
\frac{1}{\log2}\int_{1/2}^{\infty}u(s)\frac{ds}{s} & 1/2\leq s<1
\end{array}
\right.  .
\]
It is easy to see that
\[
a=\int_{0}^{1}\tilde{u}(s)\frac{ds}{s}.
\]
Moreover, $\tilde{u}(s)$ is an admissible function. Indeed, since the pair
$\vec{A}$ is ordered, we have $A_{0}\cap A_{1}=A_{1},$ and by H\"{o}lder's
inequality,
\begin{align}
\left\Vert \int_{1/2}^{\infty}u(s)\frac{ds}{s}\right\Vert _{A_{1}}  &
\leq\int_{1/2}^{\infty}\left\Vert u(s)\right\Vert _{A_{1}}\frac{ds}%
{s}\nonumber\\
&  \leq\int_{1/2}^{\infty}J(s,u(s);\vec{A})\,\frac{ds}{s^{2}}\nonumber\\
&  \leq\Vert J(s,u(s);\vec{A})\Vert_{G}\Vert s^{-1}\chi_{\lbrack1/2,\infty
)}\Vert_{G^{\prime}}\nonumber\\
&  \leq(\left\Vert a\right\Vert _{\vec{A}_{G}^{J}}+\varepsilon)\Vert
s^{-1}\chi_{\lbrack1/2,\infty)}\Vert_{G^{\prime}}, \label{admissibility}%
\end{align}
where $G^{\prime}$ is the K\"{o}the dual to $G$ with respect to the bilinear
form
\[
(f,g)=\int_{0}^{\infty}f(s)g(s)\,\frac{ds}{s}.
\]
Since $G\subset{L}^{1}+{L}^{1}(1/s)$, it follows that $G^{\prime}\supset
{L}^{\infty}\cap{L}^{\infty}(s)$. Therefore, \newline$\Vert s^{-1}%
\chi_{\lbrack1/2,\infty)}\Vert_{G^{\prime}}<\infty$. In particular, from
\eqref{admissibility} it follows that $\tilde{u}(s):(0,1]\rightarrow A_{1},$
as we wished to show.

Furthermore, we have
\[
\Vert J(s,u(s);\vec{A})\chi_{\lbrack0,1]}\Vert_{\tilde{G}}\leq\left(
\left\Vert a\right\Vert _{\vec{A}_{G}^{J}}+\varepsilon\right)  +\frac{1}%
{\log2}\left\Vert J\left(  s,\int_{1/2}^{\infty}u(r)\,\frac{dr}{r};\vec
{A}\right)  \chi_{\lbrack1/2,1]}(s)\right\Vert _{G}.
\]
Since $J(s,\cdot;\vec{A})$ is a norm for each $s>0,$ and the pair $\vec{A}$ is
ordered, from \eqref{admissibility} we get%

\begin{align*}
\left\Vert J\left(  s,\int_{1/2}^{\infty}u(r)\,\frac{dr}{r};\vec{A}\right)
\chi_{\lbrack1/2,1]}\right\Vert _{G}  &  \leq\left\Vert \int_{1/2}^{\infty
}J(s,u(r);\vec{A})\,\frac{dr}{r}\cdot\chi_{\lbrack1/2,1]}(s)\right\Vert _{G}\\
&  \leq\left\Vert \int_{1/2}^{\infty}\Vert u(r)\Vert_{A_{1}}\,\frac{dr}%
{r}\cdot\chi_{\lbrack1/2,1]}(s)\right\Vert _{G}\\
&  \leq\int_{1/2}^{\infty}\Vert u(r)\Vert_{A_{1}}\,\frac{dr}{r}\cdot\Vert
\chi_{\lbrack1/2,1]}\Vert_{G}\\
&  \leq(\left\Vert a\right\Vert _{\vec{A}_{G}^{J}}+\varepsilon)\cdot\Vert
s^{-1}\chi_{\lbrack1/2,\infty)}\Vert_{G^{\prime}}\cdot\Vert\chi_{\lbrack
1/2,1]}\Vert_{G}.
\end{align*}
Combining inequalities and letting $\varepsilon\rightarrow0,$ we obtain
\begin{equation}
\left\Vert a\right\Vert _{\langle\vec{A}\rangle_{\tilde{G}}^{J}}\leq\left(
1+\frac{1}{\log2}\cdot\Vert s^{-1}\chi_{\lbrack1/2,\infty)}\Vert_{G^{\prime}%
}\cdot\Vert\chi_{\lbrack1/2,1]}\Vert_{G}\right)  \cdot\left\Vert a\right\Vert
_{\vec{A}_{G}^{J}}. \label{general est}%
\end{equation}
This concludes the proof of (i) since we obviously have $\Vert\chi
_{\lbrack1/2,1]}\Vert_{G}<\infty$.

(ii) We apply \eqref{general est} with $G=\Phi_{\theta,q}$. We need to
estimate the norms $\Vert\chi_{\lbrack1/2,1]}\Vert_{\Phi_{\theta,q}}$ and
$\Vert s^{-1}\chi_{\lbrack1/2,\infty)}\Vert_{(\Phi_{\theta,q})^{\prime}}$. For
this purpose note that,
\[
\Vert g\Vert_{(\Phi_{\theta,q})^{\prime}}=\Big(\int_{0}^{\infty}(s^{\theta
}|g(s)|)^{q^{\prime}}\,\frac{ds}{s}\Big)^{1/q^{\prime}},
\]
where $1/q+1/q^{\prime}=1$ (with the natural modification if $q^{\prime
}=\infty$). Consequently,
\[
\Vert s^{-1}\chi_{\lbrack1/2,\infty)}\Vert_{(\Phi_{\theta,q})^{\prime}%
}=\Big(\int_{1/2}^{\infty}s^{(\theta-1)q^{\prime}}\,\frac{ds}{s}%
\Big)^{1/q^{\prime}}\leq2\left(  (1-\theta)q^{\prime}\right)  ^{-1/q^{\prime}%
}.
\]
Furthermore, it can be easily verified that
\[
\Vert\chi_{\lbrack1/2,1]}\Vert_{\Phi_{\theta,q}}=\Big(\int_{1/2}^{1}s^{-\theta
q}\,\frac{ds}{s}\Big)^{1/q}\leq\left(  \frac{2^{\theta q}-1}{\theta q}\right)
^{1/q}\leq2.
\]
The desired result follows upon inserting this information in \eqref{general
est}.
\end{proof}

\begin{remark}
\label{Rem1} We generally use Banach lattices on $[0,1]$ as parameters for the
$K$- and $J$-methods of interpolation. However, we should like to observe that
from Proposition~\ref{Aprop1} it follows that, for every ordered pair $\vec
{A}$ and all $\theta\in\lbrack1/2,1),$ we have%
\[
\left\Vert a\right\Vert _{\langle\vec{A}\rangle_{\theta,q}^{K}}\leq\left\Vert
a\right\Vert _{\vec{A}_{\theta,q}^{K}}\leq2\left\Vert a\right\Vert
_{\langle\vec{A}\rangle_{\theta,q}^{K}}.
\]
Similarly, by Proposition~\ref{Aprop2},
\[
\left\Vert a\right\Vert _{\vec{A}_{\theta,q}^{J}}\leq\left\Vert a\right\Vert
_{\langle\vec{A}\rangle_{\theta,q}^{J}}\leq\left(  1+\frac{8}{\log2}\right)
\left\Vert a\right\Vert _{\vec{A}_{\theta,q}^{J}},
\]
for all $\theta\in(0,1/2]$. Therefore, for ordered pairs $\vec{A}=(A_{0}%
,A_{1})$, $A_{1}\subset A_{0}$, the extrapolation descriptions of limiting
interpolation spaces obtained in Theorems \ref{StrExtrAbst}, \ref{ExtrMinLat}
and \ref{StrExtrAbst(J)} below will still hold if we replace the scale
$\left\{  {\langle\vec{A}\rangle_{\theta,q}^{K\blacktriangleleft}}\right\}
_{\theta\in\lbrack1/2,1)}$ (resp. $\left\{  {\langle\vec{A}\rangle_{\theta
,q}^{J\blacktriangleleft}}\right\}  _{\theta\in(0,1/2]})$ with the classical
scale $\left\{  {\vec{A}_{\theta,q}^{K\blacktriangleleft}}\right\}
_{\theta\in\lbrack1/2,1)}$ (resp. $\left\{  {\vec{A}_{\theta,q}%
^{J\blacktriangleleft}}\right\}  _{\theta\in(0,1/2]})$.
\end{remark}

\subsection{Strong Fundamental Lemma and Peetre's limit
theorem\label{sec:peetre}}

Although we will formally recall {the strong form of the fundamental lemma
later (cf. (\ref{j3}) below), it }will be instructive to present now {a proof
of Peetre's limit formula (cf. (\ref{j1}) above), both because Peetre's result
is useful for our purposes in this paper and furthermore, because the method
of proof illustrates, in a simple context, an argument that appears several
times in our development in this paper (cf. also Sections
\ref{sec:equivalence} and \ref{sec:apllim} below). More precisely, we now
provide a simple approach (compare with \cite[Lemma 1.1]{pee}) to obtain the
non-trivial part of (\ref{j1}) for Gagliardo complete pairs.}

Let $\vec{A}$ be a Gagliardo complete pair with $A_{0}\cap A_{1}$ dense in
$A_{0}$. Let $f\in A_{0}.$ By the strong form of the fundamental lemma we can
find a decomposition such that
\[
f=\int_{0}^{\infty}u(s)\frac{ds}{s},
\]
such that for every $t>0,$ it holds
\[
\int_{0}^{t}J(s,u(s),\vec{A})\frac{ds}{s}\leq\gamma K(t,f,\vec{A}),
\]
where $\gamma$ is a universal constant. But then, taking supremum over all
$t>0,$ it follows that%
\[
\left\Vert f\right\Vert _{\vec{A}_{0,1}^{J}}\leq\gamma\left\Vert f\right\Vert
_{A_{0}},
\]
as we wished to show.

\subsection{Extrapolation methods\label{sec:extrapol}}

In this section we continue the exposition started in the Introduction of the
paper and provide a brief background survey on the extrapolation methods we
shall use in this paper. For more information and results we refer to the
recent survey \cite{AL2017} and the references therein.

It will be useful to start by pointing out the common root between
interpolation and extrapolation. In interpolation we work with functors $F,$
that assign to each pair $\vec{A}$ an interpolation space $F(\vec{A}$)$.$ On
the other hand, extrapolation functors $\mathfrak{E}$ are defined {on some set
of families} $\{A_{\theta}\}_{\theta\in(0,1)}$ of compatible Banach spaces (in
the sense that there exist two Banach spaces $A_{0}$ and $A_{1}$, such that
for each $\theta\in(0,1)$, we have with continuous inclusions $A_{1}\subset
A_{\theta}\subset A_{0})$ and assign to each such family an extrapolation
space $\mathfrak{E}(\{A_{\theta}\}_{\theta\in(0,1)})$ with the following
interpolation property. If $T$ is an operator such that $T:A_{\theta
}\overset{1}{\rightarrow}B_{\theta},$ for each $\theta\in(0,1),$ then $T$ can
be extended to $T:\mathfrak{E}(\{A_{\theta}\}_{\theta\in(0,1)})\rightarrow
\mathfrak{E}(\{B_{\theta}\}_{\theta\in(0,1)})$. The simplest, and at the same
time more important extrapolation functors, are $\Sigma$ and $\Delta$ methods
which we now describe (cf. \cite{JM91}).

Suppose that the norms $M_{0}(\theta)$ of the inclusions$\,\ A_{\theta}\subset
A_{0}$ satisfy the condition: $\sup_{\theta\in(0,1)}M_{0}(\theta)<\infty$.
Then we can form the space $\Sigma\{A_{\theta}\}_{\theta\in(0,1)}$ of all the
elements $x\in A_{0}$ that can be represented \ by $x=\Sigma a_{\theta}$,
$a_{\theta}\in A_{\theta}$, with $\Sigma\left\Vert a_{\theta}\right\Vert
_{A_{\theta}}<\infty$. We endow $\Sigma\{A_{\theta}\}_{\theta\in(0,1)}$ with
the corresponding quotient norm. Likewise, if the norms $M_{1}(\theta)$ of the
embeddings $A_{1}\subset A_{\theta}$ are uniformly bounded, we form the space
$\Delta\{A_{\theta}\}_{\theta\in(0,1)}$ of all elements $x\in\cap_{\theta
\in(0,1)}A_{\theta},$ such that $\left\Vert x\right\Vert _{\Delta\{A_{\theta
}\}_{\theta\in(0,1)}}:=\sup_{\theta\in(0,1)}\left\Vert x\right\Vert
_{A_{\theta}}<\infty.$

The theory of extrapolation originated from the classical results of Yano
\cite{yano}. One of the objectives of modern extrapolation theory is to
characterize family of inequalities and, in particular, reverse the
interpolation process and find the best possible end point inequalities. In
this sense modern extrapolation is a qualitative evolution from the initial
extrapolation theorems of Yano. Indeed, in the classical extrapolation
theorems of Yano type, e.g., from the assumption that $T:L^{p}(0,1)\rightarrow
L^{p}(0,1),$ with $\left\Vert T\right\Vert _{L^{p}(0,1)\rightarrow L^{p}%
(0,1)}\leq\frac{c}{p-1},$ for all $p>1,$ we can conclude that
$T:LLogL(0,1)\rightarrow L^{1}(0,1)$ but we have the drawback that, in
general, $T:LLogL(0,1)\rightarrow L^{1}(0,1)\nRightarrow$ $T:L^{p}%
(0,1)\rightarrow L^{p}(0,1)$ (unless we have more structural assumptions on
the underlying measure space and the operator $T$ (cf. \cite{Tao})). On the
other hand, by \cite{JM91}, the assumptions of Yano's theorem above are
equivalent to{\ the inequality
\[
K(t,Tf;L^{1},L^{\infty})\leq c\int_{0}^{t}K(s,f;L^{1},L^{\infty})\frac{ds}%
{s},
\]
} or informally, denoting $x^{\ast\ast}(t):=\frac{1}{t}\int_{0}^{t}x^{\ast
}(s)\,ds$, we have
\[
\left\Vert T\right\Vert _{L^{p}(0,1)\rightarrow L^{p}(0,1)}\leq\frac{c}%
{p-1},\text{for all }p>1\Leftrightarrow\left(  Tf\right)  ^{\ast\ast}%
(t)\leq\frac{c}{t}\int_{0}^{t}f^{\ast\ast}(s)ds.
\]

The $\Sigma$- and $\Delta$-methods are the natural prototypes of two general
families of extrapolation functors that were{\ introduced in \cite{As2003} and
then were studied in \cite{As05,AL2009,AL,AL2017}. Let $F$ be a Banach
function lattice on the interval $[0,1]$ (with respect to the usual Lebesgue
measure). A given family $\{A_{\theta}\}_{\theta\in(0,1)}$ of compatible
Banach spaces, we define the Banach space $\mathbf{F}(\{A_{\theta}%
\}_{\theta\in(0,1)})$, consisting of all $a\in$}$\cap_{\theta\in
(0,1)}A_{\theta}${\ such that the function $\theta\in(0,1)\mapsto{\Vert
a\Vert}_{A_{\theta}}$ belongs to $F$, endowed with the norm $\Vert
a\Vert:={\left\Vert \,{\Vert a\Vert}_{A_{\theta}}\,\right\Vert }_{F}$. In
particular, if $F=L^{\infty}[0,1]$, we arrive at the definition of the
$\Delta$-functor. In analogous way one can define a family of extrapolation
functors generalizing the $\Sigma$-functor (see the definition of the $\vec
{A}_{\xi,q,G}^{J}$ spaces in Section \ref{sec:main}, (\ref{represetation1})
below).}

\section{A new characterization of limiting interpolation
spaces\label{sec:main}}

{We begin with the following key result of this paper.}

\begin{theorem}
\label{StrExtrAbst} Let $\vec{A}=(A_{0},A_{1})$ be a Banach pair, and let $F$
be a parameter for the $K$-method on $[0,1]$.
Suppose that the operator $Tf(t):=f(t^{2})/t$ is bounded on $F$. Then there
exist absolute constants such that for all $a\in${$\langle\vec{A}\rangle
_{F}^{K}$}, $q\in\lbrack1,\infty]$, it holds%
\begin{equation}
\Vert a\Vert_{{\langle\vec{A}\rangle_{F}^{K}}}\asymp\left\Vert t\cdot\Vert
a\Vert_{{\langle}\vec{A}\rangle_{\theta(t),q}^{K\blacktriangleleft}%
}\right\Vert _{F}\text{,} \label{BasEq}%
\end{equation}
where $\theta(t)=1+\frac{1}{2\log(t/e)}$.
\end{theorem}

\begin{proof}
From the inequality
\[
\min(1,t/s)K(s,a;\vec{A})\leqslant K(t,a;\vec{A})
\]
we get
\begin{align*}
\Vert a\Vert_{{\langle}\vec{A}\rangle_{\theta,q}^{K\blacktriangleleft}}  &
\geq(\theta(1-\theta)q)^{\frac{1}{q}}K(s,a;\vec{A})\left(  \int\limits_{0}%
^{1}(t^{-\theta}\min(1,t/s))^{q}\,\frac{dt}{t}\right)  ^{\frac{1}{q}}\\
&  \geq(\theta(1-\theta)q)^{\frac{1}{q}}K(s,a;\vec{A})\left(  \frac{s^{-\theta
q}}{(1-\theta)q}\right)  ^{\frac{1}{q}}\\
&  =\theta^{\frac{1}{q}}s^{-\theta}K(s,a;\vec{A}).
\end{align*}
Hence, if $1\leqslant q\leqslant r<\infty$,
\begin{align*}
\Vert a\Vert_{{\langle}\vec{A}\rangle_{\theta,r}^{K\blacktriangleleft}}  &
=(\theta(1-\theta)r)^{\frac{1}{r}}\left(  \int\limits_{0}^{1}(t^{-\theta
}K(t,a;\vec{A}))^{q}(t^{-\theta}K(t,a;\vec{A}))^{r-q}\,\frac{dt}{t}\right)
^{\frac{1}{r}}\\
&  \leq({r}/{q})^{\frac{1}{r}}\Vert a\Vert_{{\langle}\vec{A}\rangle_{\theta
,q}^{K\blacktriangleleft}}^{\frac{q}{r}}\cdot\theta^{\frac{1}{r}-\frac{1}{q}%
}\Vert a\Vert_{{\langle}\vec{A}\rangle_{\theta,q}^{K\blacktriangleleft}%
}^{1-\frac{q}{r}}=({r}/{q})^{\frac{1}{r}}\theta^{\frac{1}{r}-\frac{1}{q}}\Vert
a\Vert_{{\langle}\vec{A}\rangle_{\theta,q}^{K\blacktriangleleft}}.
\end{align*}
On the other hand, from the definition of $\theta(t)$ it follows readily that
for $0<t\leq1$, we have $1/2\leq\theta(t)<1$. Combining this observation with
the preceding inequality, we obtain that for all $0<t\leq1$,%
\begin{equation}
\frac{1}{2}\Vert a\Vert_{\langle\vec{A}\rangle_{\theta(t),\infty
}^{K\blacktriangleleft}}\leq\Vert a\Vert_{\langle\vec{A}\rangle_{\theta
(t),r}^{K\blacktriangleleft}}\leq4\Vert a\Vert_{\langle\vec{A}\rangle
_{\theta(t),q}^{K\blacktriangleleft}},\text{\quad}1\leq q\leq r\leq
\infty\text{.} \label{embeddings}%
\end{equation}
Moreover, on account of the fact that
\[
\log t^{1-\theta(t)}=-\frac{\log t}{2\log(t/e)}\geq-\frac{1}{2},\;\;0<t\leq1,
\]
it follows that
\begin{equation}
t^{1-\theta(t)}\geq e^{-1/2},\;\;0<t\leq1. \label{aux1}%
\end{equation}
Therefore, for all $0<t\leq1,$
\begin{align}
t\Vert a\Vert_{\langle\vec{A}\rangle_{\theta(t),q}^{K\blacktriangleleft}}  &
\geq\frac{t}{2}\Vert a\Vert_{\langle\vec{A}\rangle_{\theta(t),\infty
}^{K\blacktriangleleft}}=\frac{t}{2}\underset{0<s\leq1}{\sup}\left(
s^{-\theta(t)}K(s,a;\vec{A})\right) \nonumber\\
&  \geq\frac{1}{2}t^{1-\theta(t)}K(t,a;\vec{A})=\frac{1}{2\sqrt{e}}%
K(t,a;\vec{A}). \label{Basaux}%
\end{align}
Consequently,
\[
\left\Vert t\cdot\Vert a\Vert_{\langle\vec{A}\rangle_{\theta(t),q}%
^{K\blacktriangleleft}}\right\Vert _{F}\geq\frac{1}{2\sqrt{e}}\Vert
a\Vert_{{\langle\vec{A}\rangle_{F}^{K}}}\text{.}%
\]

To prove the converse inequality let us write,%
\begin{equation}
\left\Vert t\cdot\Vert a\Vert_{\langle\vec{A}\rangle_{\theta(t),q}%
^{K\blacktriangleleft}}\right\Vert _{F}\leq\left\Vert t\cdot\Vert
a\Vert_{\langle\vec{A}\rangle_{\theta(t),q}^{K\blacktriangleleft}}%
\chi_{(0,1/e)}\right\Vert _{F}+\left\Vert t\cdot\Vert a\Vert_{\langle\vec
{A}\rangle_{\theta(t),q}^{K\blacktriangleleft}}\chi_{(1/e,1)}\right\Vert _{F}.
\label{absorb}%
\end{equation}
We will now show that the second term on the right hand side of (\ref{absorb})
can be absorbed into the first one. Indeed, the definition of $\theta(t)$
implies that for $1/e<t\leq1,$ we have $\theta(t)\leq3/4$; consequently%
\begin{align*}
\left\Vert t\cdot\Vert a\Vert_{\langle\vec{A}\rangle_{\theta(t),q}%
^{K\blacktriangleleft}}\chi_{(1/e,1)}\right\Vert _{F}  &  \leq C_{1}\left\Vert
t\cdot\Vert a\Vert_{\langle\vec{A}\rangle_{\theta(t),q}^{K\blacktriangleleft}%
}\chi_{(1/e,1)}\right\Vert _{{L}^{\infty}(1/t)}\\
&  \leq\frac{4}{3}C_{1}\Vert a\Vert_{\langle\vec{A}\rangle_{3/4,q}%
^{K\blacktriangleleft}}\\
&  \leq\frac{4}{3}C_{1}e\left\Vert t\cdot\Vert a\Vert_{\langle\vec{A}%
\rangle_{\theta(t),q}^{K\blacktriangleleft}}\chi_{(1/e^{2},1/e)}\right\Vert
_{{L}^{\infty}}\\
&  \leq\frac{4}{3}C_{1}C_{2}e\left\Vert t\cdot\Vert a\Vert_{\langle\vec
{A}\rangle_{\theta(t),q}^{K\blacktriangleleft}}\chi_{(0,1/e)}\right\Vert _{F},
\end{align*}
where $C_{1}$ and $C_{2}$ are the constants of the embeddings ${L}^{\infty
}(1/t)\subset F$ and $F\subset{L}^{\infty}$, respectively. Inserting this
estimate in (\ref{absorb}) we find
\begin{align}
\left\Vert t\cdot\Vert a\Vert_{\langle\vec{A}\rangle_{\theta(t),q}%
^{K\blacktriangleleft}}\right\Vert _{F}  &  \leq\left\Vert t\cdot\Vert
a\Vert_{\langle\vec{A}\rangle_{\theta(t),q}^{K\blacktriangleleft}}%
\chi_{(0,1/e)}\right\Vert _{F}+\left\Vert t\cdot\Vert a\Vert_{\langle\vec
{A}\rangle_{\theta(t),q}^{K\blacktriangleleft}}\chi_{(1/e,1)}\right\Vert
_{F}\nonumber\\
&  \leq\Big(1+\frac{4}{3}C_{1}C_{2}e\Big)\left\Vert t\cdot\Vert a\Vert
_{\langle\vec{A}\rangle_{\theta(t),q}^{K\blacktriangleleft}}\chi
_{(0,1/e)}\right\Vert _{F}. \label{EQ1100}%
\end{align}
We now estimate the right-hand side of \eqref{EQ1100}. From \eqref{embeddings}
it follows that for every $1/2\leq\theta<1,$ and $1\leq q\leq\infty$,
\begin{equation}
\Vert a\Vert_{\langle\vec{A}\rangle_{\theta,q}^{K\blacktriangleleft}}%
\leq4\Vert a\Vert_{\langle\vec{A}\rangle_{\theta,1}^{K\blacktriangleleft}}%
\leq4(1-\theta)\int\limits_{0}^{1}s^{-\theta}K(s,a;\vec{A})\,\frac{ds}{s}.
\label{aux2}%
\end{equation}
Furthermore, for all $0<t\leq1/e,$ we have
\[
\log t^{1-\theta(t)}=-\frac{\log t}{2\log(t/e)}\leq-\frac{1}{4}.
\]
Hence,
\[
t^{1-\theta(t)}\leq e^{-1/4}.
\]
The last inequality, combined with the fact that $K(s,a;\vec{A})/s$ is a
decreasing function, yields that for all $0<t\leq1/e$,
\begin{align*}
t(1-\theta(t))\int\limits_{0}^{1}s^{-\theta(t)}K(s,a;\vec{A})\frac{ds}{s}  &
=t(1-\theta(t))\int\limits_{t}^{1}s^{-\theta(t)}K(s,a;\vec{A})\frac{ds}{s}\\
&  +t(1-\theta(t))\sum_{k=1}^{\infty}\int\limits_{t^{2^{k}}}^{t^{2^{k-1}}%
}s^{-\theta(t)}K(s,a;\vec{A})\frac{ds}{s}\\
&  \leq K(t,a;\vec{A})+\sum_{k=1}^{\infty}\frac{K(t^{2^{k}},a;\vec{A}%
)}{t^{2^{k}}}t\cdot t^{(1-\theta(t))2^{k-1}}\\
&  \leq K(t,a;\vec{A})+\sum_{k=1}^{\infty}e^{-2^{k-3}}\frac{K(t^{2^{k}}%
,a;\vec{A})}{t^{2^{k}}}t\\
&  \leq\sum_{k=0}^{\infty}e^{1-2^{k-3}}\frac{K(t^{2^{k}},a;\vec{A})}{t^{2^{k}%
}}t\\
&  =\sum_{k=0}^{\infty}e^{1-2^{k-3}}T^{k}\left(  K(t,a;\vec{A})\right)  .
\end{align*}
From \eqref{aux2}, and the fact that $T$ is bounded on $F$, we obtain
\begin{align*}
\left\Vert t\cdot\Vert a\Vert_{\langle\vec{A}\rangle_{\theta(t),q}%
^{K\blacktriangleleft}}\chi_{(0,1/e)}\right\Vert _{F}  &  \leq4\left\Vert
\sum_{k=0}^{\infty}e^{1-2^{k-3}}T^{k}\left(  K(t,a;\vec{A})\right)
\right\Vert _{F}\\
&  \leq4\sum_{k=0}^{\infty}e^{1-2^{k-3}}\left\Vert T^{k}\left(  K(t,a;\vec
{A})\right)  \right\Vert _{F}\\
&  \leq4\sum_{k=0}^{\infty}e^{1-2^{k-3}}\Vert T\Vert_{F\rightarrow F}%
^{k}\left\Vert K(t,a;\vec{A})\right\Vert _{F}\\
&  \leq C\Vert a\Vert_{{\langle\vec{A}\rangle_{F}^{K}}}\text{.}%
\end{align*}
Finally, combining the last inequality with \eqref{EQ1100} we obtain the
desired result.
\end{proof}

\begin{remark}
\label{Zam1} An inspection of the proof of Theorem~\ref{StrExtrAbst} shows
that the constants of equivalence \eqref{BasEq} depend on the norm of the
operator $T$ acting on $F$, and on the norm of the embeddings ${L}^{\infty
}(1/t)\subset F\subset{L}^{\infty}$. The latter dependence can be eliminated
if $F$ is an interpolation space with respect to the pair $({L}^{\infty}%
,{L}^{\infty}(1/t))$. Indeed, arguing in the same way as in the proof of
Theorem~\ref{StrExtrAbst}, we obtain
\[
\left\Vert t\cdot\Vert a\Vert_{\langle\vec{A}\rangle_{\theta_{1}%
(t),q}^{K\blacktriangleleft}}\chi_{(0,1/e)}\right\Vert _{F}\leq C\Vert
a\Vert_{{\langle\vec{A}\rangle_{F}^{K}}},
\]
where $\theta_{1}(t)=1+\frac{1}{2\log t}$, $0<t<1/e$. Therefore, if
$C^{\prime}$ denotes the norm of the dilation operator $\sigma_{e}%
f(t):=f(t/e)$ on $F$, we see that
\begin{align*}
\left\Vert t\cdot\Vert a\Vert_{\langle\vec{A}\rangle_{\theta(t),q}%
^{K\blacktriangleleft}}\right\Vert _{F}  &  =e\left\Vert \sigma_{e}%
(t\cdot\Vert a\Vert_{\langle\vec{A}\rangle_{\theta_{1}(t),q}%
^{K\blacktriangleleft}}\chi_{(0,1/e)})\right\Vert _{F}\\
&  \leq eC^{\prime}\left\Vert t\cdot\Vert a\Vert_{\langle\vec{A}%
\rangle_{\theta_{1}(t),q}^{K\blacktriangleleft}}\chi_{(0,1/e)}\right\Vert
_{F}\\
&  \leq eCC^{\prime}\left\Vert a\right\Vert _{\langle\vec{A}\rangle_{F}^{K}},
\end{align*}
and our claim follows.

In particular, we note that the constants of equivalence in \eqref{BasEq} are
independent of $q$. It is now easy to see that
\begin{equation}
\left\Vert a\right\Vert _{{\langle\vec{A}\rangle_{F}^{K}}}\asymp\left\Vert
t\cdot\Vert a\Vert_{\langle\vec{A}\rangle_{\theta(t),q(t)}%
^{K\blacktriangleleft}}\right\Vert _{F} \label{ExtRel}%
\end{equation}
for each continuous function $q(t):\,(0,1]\rightarrow\lbrack1,\infty]$.
\end{remark}

\begin{remark}
\label{Zam3} We remind the reader that, according to Remark \ref{Rem1}, when
dealing with ordered pairs $\vec{A}=(A_{0},A_{1}),$ we can replace the scale
$\{\langle\vec{A}\rangle_{\theta,q}^{K\blacktriangleleft}\}$ in Theorem
\ref{StrExtrAbst} with the scale of the normalized Lions-Peetre spaces
$\{\vec{A}_{\theta,q}^{K\blacktriangleleft}\}$. Similarly, in \eqref{ExtRel}
we can replace the scale $\{\langle\vec{A}\rangle_{\theta(t),q(t)}%
^{K\blacktriangleleft}\}_{t\in(0,1]}$ with the scale $\{\vec{A}_{\theta
(t),q(t)}^{K\blacktriangleleft}\}_{t\in(0,1]}$. As an example, consider the
ordered pair $\vec{A}=(L^{1}[0,1],L^{\infty}[0,1]).$ Observe that if
$\theta\geq1/2$, then setting $q=1/(1-\theta)$ we have $q\geq2,$ and one can
easily verify that
\[
\frac{1}{\sqrt{2}}\leq(q\theta(1-\theta))^{1/q}\leq1.
\]
Therefore, combining
\[
\Vert a\Vert_{\langle\vec{A}\rangle_{\theta,q}^{K}}=\Big(\int_{0}%
^{1}\Big(\frac{1}{t}\int_{0}^{t}a^{\ast}(s)\,ds\Big)^{q}\,dt\Big)^{1/q},
\]
(see the formula for the $K$-functional of the couple $(L^{1},L^{\infty})$ in
\cite[Theorem~5.2.1]{BL}) and the fact that the norm of the Hardy operator
$a(s)\mapsto1/t\int_{0}^{t}a^{\ast}(s)\,ds$ on $L_{q}[0,1]$ is equal to
$q/(q-1)$, we obtain
\[
\frac{1}{\sqrt{2}}\Vert a\Vert_{L_{q}[0,1]}\leq\Vert a\Vert_{\vec{A}%
_{\theta,q}^{K\blacktriangleleft}}\leq2\Vert a\Vert_{L_{q}[0,1]}.
\]
Consequently, for a Banach lattice $F$ satisfying the conditions of Theorem
\ref{StrExtrAbst} we obtain
\[
\Vert a\Vert_{{\langle L^{1}[0,1],L^{\infty}[0,1]\rangle_{F}^{K}}}%
\asymp\left\Vert t\cdot\Vert a\Vert_{L^{q(t)}[0,1]}\right\Vert _{F}%
,\quad{where}\;q(t)=2\log(e/t).
\]

On the other hand, for non-ordered pairs $\vec{A}$ it follows readily from the
definitions that for $\theta\geq1/2$ we can write
\[
\langle\vec{A}\rangle_{\theta,q}^{K\blacktriangleleft}=(A_{0}+A_{1}%
,A_{1})_{\theta,q}^{K\blacktriangleleft},
\]
with the norm equivalence independent of $\theta$. Indeed, it is easy to see
that
\[
K(t,a;A_{0}+A_{1},A_{1})=%
\begin{cases}
K(t,a;A_{0},A_{1})\quad\mbox{for}\quad0<t\leq1,\\
K(1,a;A_{0},A_{1})\quad\mbox{for}\quad t\geq1.
\end{cases}
\]
Hence, for $\theta\geq1/2$ we have$,$
\begin{align*}
\Vert a\Vert_{\langle\vec{A}\rangle_{\theta,q}^{K\blacktriangleleft}}  &
\leq\Vert a\Vert_{(A_{0}+A_{1},A_{1})_{\theta,q}^{K\blacktriangleleft}}%
\leq\Vert a\Vert_{\langle\vec{A}\rangle_{\theta,q}^{K\blacktriangleleft}%
}+(1-\theta)^{\frac{1}{q}}K(1,a;A_{0},A_{1})\\
&  \leq(1+({1-\theta})^{\frac{1}{q}}\theta^{-\frac{1}{q}})\Vert a\Vert
_{\langle\vec{A}\rangle_{\theta,q}^{K\blacktriangleleft}}\leq2\Vert
a\Vert_{\langle\vec{A}\rangle_{\theta,q}^{K\blacktriangleleft}}.
\end{align*}
In particular, let us consider the (unordered) pair $\vec{A}=(L^{1}%
(0,\infty),L^{\infty}(0,\infty)).$ Then, for every Banach lattice $F$ that
satisfies the conditions of Theorem \ref{StrExtrAbst}, we have the
equivalence
\[
\Vert a\Vert_{{\langle\vec{A}\rangle_{F}^{K}}}=\Vert a\Vert_{{\langle
L^{1},L^{\infty}\rangle_{F}^{K}}}\asymp\left\Vert t\cdot{\Vert a\Vert
}_{L^{q(t)}+L^{\infty}}\right\Vert _{F},\quad{where}\;q(t)=2\log
(e/t),t\in\lbrack0,1].
\]
For example, if $F=L^{\infty}(1/t)(0,1)$ we get the trivial equality
\[
L^{\infty}(0,\infty)=\Delta_{q\geq2}(L^{q}(0,\infty)+L^{\infty}(0,\infty)).
\]

\end{remark}

\begin{remark}
\label{Zam1a} It will be useful to present a detailed reformulation of
Theorem~\ref{StrExtrAbst} for interpolation spaces constructed using the
$K$-method and parameter spaces defined on $[1,\infty)$. Suppose that $G$ is a
Banach lattice of functions on $[1,\infty)$ such that $L^{\infty}%
[1,\infty)\subset G\subset L^{\infty}(1/t)[1,\infty)$, and let $\Vert
b\Vert_{\langle\langle{\vec{B}}\rangle\rangle_{G}^{K}}:=\Vert K(t,b;\vec
{B})\chi_{(1,\infty)}\Vert_{G}.$ Then, if the operator $Sf(t):=f(t^{2})$ is
bounded on $G$, it follows that for each $q\in\lbrack1,\infty]$,
\begin{equation}
\Vert b\Vert_{\langle\langle{B_{0},B_{1}}\rangle\rangle_{G}^{K}}%
\asymp\left\Vert \Vert b\Vert_{\langle\langle B_{0},B_{1}\rangle\rangle
_{\eta(t),q}^{K\blacktriangleleft}}\right\Vert _{G}\text{,} \label{BasEqa}%
\end{equation}
where $\eta(t)=\frac{1}{2\log(et)}$, $t\geq1$, and
\[
\Vert b\Vert_{\langle\langle B_{0},B_{1}\rangle\rangle_{\eta(t),q}%
^{K\blacktriangleleft}}:=[q\eta(t)(1-\eta(t))]^{1/q}\left(  \int%
\limits_{1}^{\infty}\left(  s^{-\eta(t)}K(s,b;\vec{B})\right)  ^{q}\,\frac
{ds}{s}\right)  ^{\frac{1}{q}},
\]
with the natural modification if $q=\infty$.

Indeed, since $K(t,b;\vec{B})=tK(1/t,b;B_{1},B_{0})$, $t>0$, we readily see
that
\[
\langle\langle B_{0},B_{1}\rangle\rangle_{G}^{K}=\langle B_{1},B_{0}%
\rangle_{F}^{K},
\]
where $F$ is the Banach lattice on $[0,1]$ normed by $\Vert f\Vert_{F}:=\Vert
tf(1/t)\Vert_{G}$. It is plain that the operator $S$ is bounded on $G$ if and
only if the operator $Tf(t)=f(t^{2})/t$ is bounded on $F$. Therefore,
combining with Theorem~\ref{StrExtrAbst}, we have
\begin{equation}
\Vert b\Vert_{\langle\langle{B_{0},B_{1}}\rangle\rangle_{G}^{K}}=\Vert
b\Vert_{\langle{B_{1},B_{0}}\rangle_{F}^{K}}\asymp\left\Vert t\cdot\Vert
b\Vert_{\langle B_{1},B_{0}\rangle_{\theta(t),q}^{K\blacktriangleleft}%
}\right\Vert _{F}=\left\Vert \Vert b\Vert_{\langle B_{1},B_{0}\rangle
_{\theta(1/t),q}^{K\blacktriangleleft}}\right\Vert _{G}\text{,} \label{pre}%
\end{equation}
where $\theta(t)=1+\frac{1}{2\log(t/e)}$. On the other hand, it is easy to see
that (see also \cite[Theorem~3.4.1(a)]{BL}) $\langle B_{1},B_{0}%
\rangle_{\theta,q}^{K\blacktriangleleft}=\langle\langle B_{0},B_{1}%
\rangle\rangle_{1-\theta,q}^{K\blacktriangleleft}$ for all $0<\theta<1$ and
$1\leq q\leq\infty,$ isometrically. Therefore, if we insert this information
back in (\ref{pre}) and observe that $1-\theta(1/t)=\eta(t),$ we obtain%
\[
\Vert b\Vert_{\langle\langle{B_{0},B_{1}}\rangle\rangle_{G}^{K}}%
\asymp\left\Vert \Vert b\Vert_{\langle\langle B_{0},B_{1}\rangle
\rangle_{1-\theta(1/t),q}^{K\blacktriangleleft}}\right\Vert _{G}=\left\Vert
\Vert b\Vert_{\langle\langle B_{0},B_{1}\rangle\rangle_{\eta(t),q}%
^{K\blacktriangleleft}}\right\Vert _{G},
\]
as desired. Moreover, proceeding as in Remark \ref{Zam3}, we see that, if
$B_{0}\subset B_{1}$, we can replace the scale $\{\langle\langle\vec{B}%
\rangle\rangle_{\eta(t),q}^{K\blacktriangleleft}\}$ in \eqref{BasEqa} by the
scale $\{\vec{B}_{\eta(t),q}^{K\blacktriangleleft}\}$; the same considerations
apply when dealing with scales where $q(t):\,[1,\infty)\rightarrow
\lbrack1,\infty]$ is an arbitrary continuous function.
\end{remark}

\begin{remark}
\label{remDiscrete} The proof of Theorem \ref{StrExtrAbst} is based on
two-sided pointwise estimates, as a consequence, in the context of Remark
\ref{Zam1a}, we can replace the lattice $G$ of functions defined on
$[1,\infty)$ by the canonical Banach lattice of sequences $G_{d}$ modelled on
$G$. Since this remark will be useful later in Section \ref{sec:scha}, we now
develop this point in detail. We let $G_{d}$ be the sequence space defined by%
\[
\left\Vert \left\{  \xi_{n}\right\}  _{n=1}^{\infty}\right\Vert _{G_{d}%
}=\left\Vert \sum\limits_{n=1}^{\infty}\xi_{n}\chi_{\lbrack n,n+1)}%
(t)\right\Vert _{G}.
\]
$G_{d}$ is a Banach sequence lattice, and a number of properties of $G_{d}$
can be derived from corresponding properties of $G.$ In particular, we see
that $\ell^{\infty}(\mathbb{N})\subset G_{d}\subset\ell^{\infty}%
(1/n)(\mathbb{N}).$ Moreover, in this context, the discrete version of the
operator $S$ can be defined by $S_{d}f(n):=f(n^{2}),$ and we readily see that
if $S$ is bounded on $G$ we have that $S_{d}$ is bounded on $G_{d}$.
Consequently, for every Banach pair $\vec{B}$ we have
\[
{\Vert\{K(n,f;\vec{B})\}_{n=1}^{\infty}\Vert}_{G_{d}}\asymp\left\Vert
\bigl\{\Vert f\Vert_{\langle\langle B_{0},B_{1}\rangle\rangle_{\eta
(n),q(n)}^{K\blacktriangleleft}}\bigr\}_{n=1}^{\infty}\right\Vert _{G_{d}},
\]
where $\eta(n)=\frac{1}{2\log(en)}$ and $q(n):\,\mathbb{N}\rightarrow
\lbrack1,\infty]$ is an arbitrary function.
\end{remark}

Our next result shows that, if we impose more conditions on the parameter
space $F$, we can obtain a converse to Theorem~\ref{StrExtrAbst}.

We shall say that a Banach pair $\vec{A}$ is \textit{$Conv_{0}$-abundant on
$[0,1],$} if there is a constant $C>0$ such that for every concave increasing
function $f$ on $[0,1]$ such that $\underset{t\rightarrow0}{\lim}f(t)=0,$ one
can find $a\in A_{0}+A_{1}$ that satisfies the inequality
\[
C^{-1}f(t)\leq K(t,a;\vec{A})\leq Cf(t),\;\;0\leq t\leq1.
\]

\begin{example}
\label{exampleabundant}T\textit{he pair }$(L^{1}[0,1],L^{\infty}%
[0,1])$\textit{\ is $Conv_{0}$-abundant on $[0,1]$ (cf. \cite{BK}).}
\end{example}

\begin{theorem}
\label{ExtrMinLat} Let $F$ be an interpolation Banach lattice on $[0,1]$ with
respect to the pair $({L}^{\infty},{L}^{\infty}(1/t))$. Then, the following
conditions are equivalent:

(a) the operator $Tf(t)=f(t^{2})/t$ is bounded on $F$;

(b) for every Banach pair $\vec{A}=(A_{0},A_{1})$ and for each $q\in
\lbrack1,\infty],$ the equivalence \eqref{BasEq} holds for the space
{$\langle\vec{A}\rangle_{F}^{K}$};

(c) for every Banach pair $\vec{A}=(A_{0},A_{1})$, we have, with constants
independent of $a\in\langle\vec{A}\rangle_{F}^{K}+\tilde{A}_{1}$, $s>0$ and
$q\in\lbrack1,\infty]$,%
\[
K(s,a;{\langle\vec{A}\rangle_{F}^{K}},\tilde{A}_{1})\asymp K(s,t\cdot\Vert
a\Vert_{\langle\vec{A}\rangle_{\theta(t),q}^{K\blacktriangleleft}}%
;F,{L}^{\infty}(1/t)),
\]
where $\theta(t)=1+\frac{1}{2\log(t/e)},t\in(0,1)$;

(d) there exist a Banach pair ${\vec{A}}$ that is $Conv_{0}$-abundant on
$[0,1]$ and $q\in\lbrack1,\infty]$ such that equivalence \eqref{BasEq} holds
for the space {$\langle\vec{A}\rangle_{F}^{K}$}.
\end{theorem}

\begin{proof}
Since the implication $(a)\Rightarrow(b)$ was proved in
Theorem~\ref{StrExtrAbst}, and the implications $(c)\Rightarrow(d)$ and
$(b)\Rightarrow(d)$ are trivial (for example, to prove $(c)\Rightarrow(d)$
simply apply $(c)$ to the pair $\vec{A}=(L^{1}(0,1),L^{\infty}(0,1))$ and
$s=1$), it therefore only remains to prove that $(a)\Rightarrow(c)$ and
$(d)\Rightarrow(a)$.

$(a)\Rightarrow(c)$. It is easy to verify that the norm of the operator $T$ on
the space ${L}^{\infty}(1/t)$ equals one. Moreover, since $T$ is bounded on
$F,$ we see that for each fixed $s>0$, $K(s,Tf;F,{L}^{\infty}(1/t))\leq
\max\{1,\Vert T\Vert_{F\rightarrow F}\}K(s,f;F,{L}^{\infty}(1/t))$. In other
words, if for each $s>0,$ we denote by $\Sigma_{s}$ the space ${L}^{\infty
}(1/t)+F$ endowed with the norm $K(s,\cdot;F,{L}^{\infty}(1/t)),$ then $T:$
$\Sigma_{s}\rightarrow\Sigma_{s}$ is bounded, and the norm of $T$ on
$\Sigma_{s}$ does not exceed $\max\{1,\Vert T\Vert_{F\rightarrow F}\}.$
Consequently, if we apply Theorem \ref{StrExtrAbst} using the functional
parameter $\Sigma_{s},$ then, for any pair $\vec{A},$ we have (with absolute
constants independent of $a$ and $s$; cf. Remark \ref{Zam1})
\begin{align*}
K(s,K(\cdot,a;\vec{A});F,{L}^{\infty}(1/t))  &  =\Vert a\Vert_{{\langle\vec
{A}\rangle_{\Sigma_{s}}^{K}}}\asymp\left\Vert t\cdot\Vert a\Vert_{\langle
\vec{A}\rangle_{\theta(t),q}^{K\blacktriangleleft}}\right\Vert _{\Sigma_{s}}\\
&  \asymp K(s,t\cdot\Vert a\Vert_{\langle\vec{A}\rangle_{\theta(t),q}%
^{K\blacktriangleleft}};F,{L}^{\infty}(1/t)).
\end{align*}
Therefore, it suffices to check that for all $s>0$ we have,
\begin{equation}
K(s,K(\cdot,a;\vec{A});F,{L}^{\infty}(1/t))\asymp K(s,a;{\langle\vec{A}%
\rangle_{F}^{K},}\tilde{A}_{1}). \label{EQ101a}%
\end{equation}

Let us consider an arbitrary representation $a=a_{0}+a_{1}$, $a_{0}\in
{\langle\vec{A}\rangle_{F}^{K}}$, $a_{1}\in\tilde{A}_{1}$. Then,
\begin{align*}
K(s,K(\cdot,a;\vec{A});F,{L}^{\infty}(1/t))  &  \leq K(s,\left(  K(\cdot
,a_{0};\vec{A})+K(\cdot,a_{1};\vec{A})\right)  ;F,{L}^{\infty}(1/t))\\
&  \leq\Vert K(\cdot,a_{0};\vec{A})\Vert_{F}+s\Vert K(\cdot,a_{1};\vec
{A})\Vert_{{L}^{\infty}(1/t)}\\
&  \leq\Vert K(\cdot,a_{0};\vec{A})\Vert_{F}+s\cdot\sup_{0<t\leq1}%
\frac{K(t,a_{1};\vec{A})}{t}\\
&  =\Vert a_{0}\Vert_{{\langle\vec{A}\rangle_{F}^{K}}}+s\Vert a_{1}%
\Vert_{\tilde{A}_{1}}.
\end{align*}
Consequently, taking the infimum over all admissible representations, we
obtain{%
\[
K(s,K(\cdot,a;\vec{A});F,{L}^{\infty}(1/t))\leq K(s,a;{\langle\vec{A}%
\rangle_{F}^{K},}\tilde{A}_{1}).
\]
}

We now prove the reverse inequality. By the definition of the $K$-functional,
we can select a decomposition of $K(t,a;\vec{A})$ such that
\begin{equation}
K(t,a;\vec{A})=f_{0}(t)+f_{1}(t),\text{ with }f_{0}\in F\text{ and }f_{1}%
\in{L}^{\infty}(1/t) \label{durban-}%
\end{equation}
and%
\begin{equation}
K(s,K(\cdot,a;\vec{A});F,{L}^{\infty}(1/t))\geq\frac{1}{2}\left(  \Vert
f_{0}\Vert_{F}+s\Vert f_{1}\Vert_{{L}^{\infty}(1/t)}\right)  . \label{durban}%
\end{equation}
Consider the mapping $f\longmapsto\tilde{f}$, where
\begin{equation}
\tilde{f}(t):=\underset{0<s\leq1}{\sup}\min\{1,t/s\}|f(s)|,\text{\quad}%
0<t\leq1\text{.} \label{maping}%
\end{equation}
It can be readily verified that $\tilde{f}$ is quasi-concave, and $\tilde
{f}(t)\geq|f(t)|$, $0<t\leq1.$ Moreover, since the mapping $f\longmapsto
\tilde{f}$ is bounded on both, ${L}^{\infty}$ and ${L}^{\infty}(1/t)$, \ we
find, by interpolation, that $f\longmapsto\tilde{f}$ is bounded on $F$. Hence,
it follows from (\ref{durban}) that%
\begin{equation}
K(s,K(\cdot,a;\vec{A});F,{L}^{\infty}(1/t))\geq c\left(  \Vert\tilde{f_{0}%
}\Vert_{F}+s\Vert\tilde{f_{1}}\Vert_{{L}^{\infty}(1/t)}\right)  .
\label{delet}%
\end{equation}
On the other hand, by (\ref{durban-}), we have
\[
K(t,a;\vec{A})\leq\tilde{f}_{0}(t)+\tilde{f}_{1}(t),\text{\ \ }0<t\leq
1\text{.}%
\]
Hence, on account of the $K$-divisibility property (we extend the functions
$\tilde{f}_{i}$, $i=0,1$, to the half-line $[0,\infty)$ setting $\tilde{f}%
_{i}(t)=\frac{\tilde{f}_{i}(1)K(t,a;\vec{A})}{K(1,a;\vec{A})}$ for $t>1$), we
can find $a_{0}\in A_{0}$ and $a_{1}\in A_{1}$ such that $a=a_{0}+a_{1}$ and
\[
K(t,a_{i};\vec{A})\leq D\tilde{f}_{i}(t)\;\;\mbox{for all}\;\;0<t\leq
1\;\mbox{and}\;i=0,1,
\]
where $D$ is a universal constant. Since $F$ and ${L}^{\infty}(1/t)$ are
lattices$,$ we get
\[
\left\Vert K(\cdot,a_{0};\vec{A})\right\Vert _{F}\leq D\left\Vert \tilde
{f}_{0}\right\Vert _{F}\;\text{ and }\;\left\Vert K(\cdot,a_{1};\vec
{A})\right\Vert _{{L}^{\infty}(1/t)}\leq D\left\Vert \tilde{f}_{1}\right\Vert
_{{L}^{\infty}(1/t)}.
\]
Combining with (\ref{delet}), {we obtain
\begin{align*}
K(s,K(\cdot,a;\vec{A});F,{L}^{\infty}(1/t))  &  \geq c\left(  \Vert
K(\cdot,a_{0};\vec{A})\Vert_{F}+s\Vert K(\cdot,a_{1};\vec{A})\Vert
_{{L}^{\infty}(1/t)}\right) \\
&  =c\left(  \Vert a_{0}\Vert_{\vec{A}_{F}^{K}}+s\Vert a_{1}\Vert_{\tilde
{A}_{1}}\right) \\
&  \geq cK(s,a;\vec{A}_{F}^{K},\tilde{A}_{1}),
\end{align*}
} concluding the proof of \eqref{EQ101a}.

$(d)\Rightarrow(a)$. Let $\vec{A}$ be a $Conv_{0}$-abundant on $[0,1]$ Banach
pair, and let $1\leq q\leq\infty$ be such that the space $\langle\vec
{A}\rangle_{F}^{K}$ satisfies \eqref{BasEq}. The argument in (\ref{Basaux})
above shows that
\[
t\cdot\Vert a\Vert_{\langle\vec{A}\rangle_{\theta(t),q}^{K\blacktriangleleft}%
}\geq\frac{1}{2{e}t}K(t^{2},a;\vec{A}).
\]
Applying the $F$-norm, and using the hypothesis, we obtain
\[
\Vert t^{-1}K(t^{2},a;\vec{A})\Vert_{F}\leq2e\left\Vert t\cdot\Vert
a\Vert_{\langle\vec{A}\rangle_{\theta(t),q}^{K\blacktriangleleft}}\right\Vert
_{F}\leq C\Vert K(t,a;\vec{A})\Vert_{F}\text{.}%
\]
Consequently, since $\vec{A}$ is a $Conv_{0}$-abundant pair on $[0,1],$ and
$T$ is a monotone linear operator, it follows that for every concave
increasing function $f\in{F}$ such that $\underset{t\rightarrow0}{\lim
}f(t)=0,$ we have
\begin{equation}
\Vert{T}f\Vert_{{F}}\leq C\Vert f\Vert_{{F}}. \label{EQ112ab}%
\end{equation}
We now rule out the possibility that there could exist a concave increasing
function $f_{0},$ such that $\underset{t\rightarrow0}{\lim}f_{0}(t)>0$ and
$f_{0}$ $\in{F.}$ For, if this were the case, then it would follow that
${F}={L}^{\infty}$. But if ${F}={L}^{\infty}$ we can show that the equivalence
\eqref{BasEq} fails for every $Conv_{0}$-abundant Gagliardo complete pair
$\vec{A}.$ Indeed, we first observe that
\[
\left\Vert a\right\Vert _{\langle\vec{A}\rangle_{{L}^{\infty}}^{K}}%
=\sup_{0<t\leq1}K(t,a;\vec{A})=\left\Vert a\right\Vert _{A_{0}+A_{1}},
\]
whence $\langle\vec{A}\rangle_{{L}^{\infty}}^{K}=A_{0}+A_{1}$. Suppose now
that, to the contrary, \eqref{BasEq} holds. Then applying successively
\eqref{embeddings}, \eqref{equivK}, and the fact that $\theta(t)\geq1/2$ for
all $0<t\leq1$, yields
\begin{align*}
\left\Vert a\right\Vert _{A_{0}+A_{1}}  &  \asymp\left\Vert t\cdot\Vert
a\Vert_{\langle\vec{A}\rangle_{\theta(t),q}^{K\blacktriangleleft}}\right\Vert
_{{L}^{\infty}}\\
&  \geq\frac{1}{2}\left\Vert t\cdot\Vert a\Vert_{\langle\vec{A}\rangle
_{\theta(t),\infty}^{K\blacktriangleleft}}\right\Vert _{{L}^{\infty}}\\
&  \geq\frac{1}{2}\liminf_{t\rightarrow1}\Vert a\Vert_{\vec{A}_{\theta
(t),\infty}^{K\blacktriangleleft}}\\
&  \geq\frac{1}{2}\Vert a\Vert_{\vec{A}_{1/2,\infty}^{K\blacktriangleleft}}.
\end{align*}
This implies that $\vec{A}_{1/2,\infty}^{K\blacktriangleleft}=A_{0}+A_{1}$
and, consequently, for all $a\in A_{0}+A_{1}$ there exists $C_{a}$ such that
$K(t,a;\vec{A})\leq C_{a}\sqrt{t}$. This obviously contradicts the assumption
that the pair $\vec{A}$ is $Conv_{0}$-abundant. As a result it follows that
$F\neq{L}^{\infty}$ and hence \eqref{EQ112ab} is fulfilled for all concave
increasing functions $f$ $\in{F}$. Moreover, since every quasi-concave
function $f$ is equivalent to its least concave majorant\footnote{in fact we
have $f\leq\mathbf{f}\leq2f$ (see e.g., \cite[Ch.~II, \S \,1, Corollary after
Theorem~1.1]{KPS}).} $\mathbf{f}$, the inequality \eqref{EQ112ab} can be
easily extended to the set of all quasi-concave functions.

Finally, recall that the mapping $f\longmapsto\tilde{f}$ (see \eqref{maping})
is bounded on ${F}$, $\tilde{f}(t)\geq|f(t)|$ if $0<t\leq1,$ and $\tilde{f}$
is a quasi-concave function for each $f$ (cf. \cite{BK}). Since the operator
$T$ is monotone, the properties of $\tilde{f}$ permit us to extend the
inequality \eqref{EQ112ab} to the whole lattice ${F}$, and therefore complete
the proof.
\end{proof}

Theorem~\ref{StrExtrAbst} can be strengthened when a given \textit{ordered}
pair $(A_{0},A_{1})$ is Gagliardo complete, i.e., when the smaller space
$A_{1}$ is Gagliardo complete with respect to the larger space $A_{0}$.

\begin{theorem}
\label{StrExtrAbst2} Let $\vec{A}=(A_{0},A_{1})$ be a Gagliardo complete
ordered pair. Let $F$ be a parameter for the $K$-method on $[0,1]$, and,
moreover, suppose that the operator $Tf(t)=f(t^{2})/t$ is bounded on $F$.
Then, for every family $\{I_{\theta}\}_{\theta\in(0,1)}$ of exact
interpolation functors of exponent\footnote{i.e. for each $\theta,$ the
characteristic function of $I_{\theta}$ is $t^{\theta}.$} $\theta,$ we have,%
\[
\Vert a\Vert_{{\langle\vec{A}\rangle_{F}^{K}}}\asymp\left\Vert t\cdot\Vert
a\Vert_{I_{\theta(t)}(\vec{A})}\right\Vert _{F}\text{,}%
\]
where $\theta(t)=1+\frac{1}{2\log(t/e)}$, $0<t\leq1,$ and the constants of
this equivalence are universal.
\end{theorem}

\begin{proof}
It is well-known that (cf. \cite[p.~11, (2.10)]{JM91}),
\begin{equation}
\vec{A}_{\theta,1}^{J}\overset{1}{\subset}I_{\theta}(\vec{A}%
)\overset{1}{\subset}\vec{A}_{\theta,\infty}^{K}\text{.} \label{exact functor}%
\end{equation}

Moreover, we also have $\vec{A}_{\theta,\infty}^{K}=\langle\vec{A}%
\rangle_{\theta,\infty}^{K\blacktriangleleft}$ and $\vec{A}_{\theta,1}%
^{J}=\vec{A}_{\theta,1}^{J\blacktriangleleft}$ isometrically and, using
\eqref{equivK}, we see that for $1/2\leq\theta<1$ {the embedding $\langle
\vec{A}\rangle_{\theta,1}^{K\blacktriangleleft}\overset{2}{\subset}\vec
{A}_{\theta,1}^{K\blacktriangleleft}$ holds}. Therefore, the desired result
would follow from Theorem~\ref{StrExtrAbst} (see also Remark~\ref{Zam1}) and
embeddings~\eqref{exact functor} if we could show that, for some constant
$C>0,$ independent of $\theta\in(0,1),$ it holds that%
\[
\vec{A}_{\theta,1}^{K\blacktriangleleft}\overset{C}{\subset}\vec{A}_{\theta
,1}^{J}\text{,}%
\]
or equivalently,
\begin{equation}
\theta(1-\theta)\vec{A}_{\theta,1}^{K}\overset{C}{\subset}\vec{A}_{\theta
,1}^{J}\text{.} \label{fromprevious}%
\end{equation}
The embedding (\ref{fromprevious}) is known (cf. \cite[page 34 line 5]{JM91},
a more recent proof is given in \cite[Theorem 1]{CFKK}). For the sake of
completeness we prove (\ref{fromprevious}) using the argument implicit in
\cite{JM91}. By hypothesis, $\vec{A}$ is a mutually closed pair and, moreover,
$A_{1}$ is dense in $\vec{A}_{\theta,1}^{J}$. Therefore, by the strong form of
the fundamental lemma (cf. \cite{CJM90}), any element $a\in\vec{A}_{\theta
,1}^{J}$ can be represented by $a=\int_{0}^{\infty}u(s)\frac{ds}{s},$ with
\begin{equation}
\int\limits_{0}^{\infty}\min(1,t/s)J(s,u(s);\vec{A})\;\frac{ds}{s}\leq\gamma
K(t,a;\vec{A}),\text{\ \ }t>0\text{,} \label{j3}%
\end{equation}
where $\gamma$ is a universal constant independent of $a$. Consequently,%
\begin{align*}
\theta(1-\theta)\Vert a\Vert_{\vec{A}_{\theta,1}^{K}}  &  =\theta
(1-\theta)\int\limits_{0}^{\infty}t^{-\theta}K(t,a;\vec{A})\;\frac{dt}{t}\\
&  \geq\theta(1-\theta)\gamma^{-1}\int\limits_{0}^{\infty}t^{-\theta}%
\int\limits_{0}^{\infty}\min(1,t/s)J(s,u(s);\vec{A})\frac{ds}{s}\frac{dt}{t}\\
&  =\theta(1-\theta)\gamma^{-1}\int\limits_{0}^{\infty}\left(  \int%
\limits_{0}^{\infty}t^{-\theta}\min(1,t/s)\,\frac{dt}{t}\right)
J(s,u(s);\vec{A})\frac{ds}{s}\\
&  =\theta(1-\theta)\gamma^{-1}\int\limits_{0}^{\infty}\left(  \int%
\limits_{0}^{s}\frac{t^{1-\theta}}{s}\frac{dt}{t}+\int\limits_{s}^{\infty
}t^{-\theta}\frac{dt}{t}\right)  J(s,u(s);\vec{A})\frac{ds}{s}\\
&  =\theta(1-\theta)\gamma^{-1}\int\limits_{0}^{\infty}\frac{s^{-\theta}%
}{\theta(1-\theta)}J(s,u(s);\vec{A})\frac{ds}{s}\\
&  =\gamma^{-1}\int\limits_{0}^{\infty}s^{-\theta}J(s,u(s);\vec{A})\frac
{ds}{s}\geq\gamma^{-1}\Vert a\Vert_{\vec{A}_{\theta,1}^{J}}\text{.}%
\end{align*}
Therefore we have shown that (\ref{fromprevious}) holds, and the proof is complete.
\end{proof}

We conclude this section with an extrapolation description of the limiting
spaces associated with the $J$-method.

Let $\xi(t):=\frac{1}{2\log(e/t)}$, $0<t\leq1$, and $1\leq q\leq\infty$.
Furthermore, suppose that $\vec{A}$ is an ordered pair and $G$ is a parameter
for the $J$-method on $(0,1]$, i.e., a Banach lattice on $((0,1],\frac{dt}%
{t})$ such that $L^{1}(1/t)\subset G\subset L^{1}$. Denote by $\vec{A}%
_{\xi,q,G}^{J}$ the space of all $a\in A_{0}$ that admit a representation
\begin{equation}
a=\int_{0}^{1}u(t)\frac{dt}{t}\text{,} \label{represetation1}%
\end{equation}
where $u(t):(0,1]\rightarrow A_{1}$ is a strongly measurable function such
that $\Vert u(t)\Vert_{\langle\vec{A}\rangle_{\xi(t),q}^{J\blacktriangleleft}%
}\in{G}$. We provide $\vec{A}_{\xi,q,G}^{J}$\ with the quotient norm
\[
\Vert a\Vert_{\vec{A}_{\xi,q,G}^{J}}:=\inf\left\Vert \Vert u(t)\Vert
_{\langle\vec{A}\rangle_{\xi(t),q}^{J\blacktriangleleft}}\right\Vert
_{G}\text{,}%
\]
where the infimum is taken over all $u(t)$ satisfying \eqref{represetation1}.

Let $G^{\prime}$ be the K\"{o}the dual lattice to $G$ with respect to the
bilinear form
\[
(f,g):=\int_{0}^{1}f(t)g(t)\,\frac{dt}{t^{2}}.
\]

\begin{theorem}
\label{StrExtrAbst(J)} Let $G$ be a separable parameter for the $J$-method on
$[0,1]$, and let the operator $Rf(t):=f(\sqrt{t})$ be bounded on $G$. Suppose
that $\vec{A}=(A_{0},A_{1})$ is an ordered pair such that $A_{1}$ is dense in
$A_{0}$ and $A_{0}^{\ast}$ is dense in the space $\langle A_{1}^{\ast}%
,A_{0}^{\ast}\rangle_{G^{\prime}}^{K}$. Then, for each $q\in\lbrack1,\infty],$
we have, with equivalence of norms,%
\[
\langle\vec{A}\rangle_{G}^{J}=\vec{A}_{\xi,q,G}^{J}.
\]

\end{theorem}

In the proof of this theorem we shall make use of the following auxiliary
result. Let $\eta>0$ and let $H_{\eta,\infty}$ be the Banach lattice of
measurable functions on $[0,1]$ such that%
\[
{\Vert f\Vert}_{H_{\eta,\infty}}:=\sup_{0<s\leq1}(s^{-\eta}|f(s)|)<\infty.
\]

\begin{lemma}
\label{lemmaH} Suppose that $G$ is a parameter for the $J$-method on $[0,1]$
such that the operator $Rf(t)=f(\sqrt{t})$ is bounded on $G$. Then, for each
$\eta>0$ the embedding $H_{\eta,\infty}\subset G$ holds.
\end{lemma}

\begin{proof}
First, observe that for each $d>1$ we have $L_{\infty}(t^{-d})\subset
L_{1}(1/t)(\frac{dt}{t})$.

Let $\eta>0$, choose a positive integer $n$ such that $2^{n}>\eta^{-1}$, and
set $d:=\eta2^{n}$. Then, $d>1$ and
\[
L_{\infty}(t^{-d})\subset L_{1}(1/t)(\frac{dt}{t})\subset G.
\]
Consequently, there exists $C>0$ such that
\[
\Vert f(s^{d/\eta})\Vert_{G}\leq C\sup_{0<s\leq1}\left(  s^{-d}|f(s^{d/\eta
})|\right)  =C\sup_{0<s\leq1}\left(  s^{-\eta}|f(s)|\right)  =C\Vert
f\Vert_{H_{\eta,\infty}}.
\]
On the other hand, by iteration and the definition of $d,$
\[
R^{n}\left(  f(s^{d/\eta})\right)  =f\left(  s^{d/(\eta2^{n})}\right)  =f(s).
\]
Therefore, since $R$ is bounded on $G,$ we obtain
\[
\Vert f\Vert_{G}=\left\Vert R^{n}\left(  f(s^{d/\eta})\right)  \right\Vert
_{G}\leq\Vert R\Vert^{n}\Vert f(s^{d/\eta})\Vert_{G}\leq C\Vert R\Vert
^{n}\Vert f\Vert_{H_{\eta,\infty}},
\]
as we wished to show.
\end{proof}

\begin{proof}
[Proof of Theorem~\ref{StrExtrAbst(J)}]

As a first step we shall prove that
\begin{equation}
\langle\vec{A}\rangle_{G}^{J}\subset\vec{A}_{\xi,q,G}^{J}. \label{1 emb}%
\end{equation}

Let $a\in\langle\vec{A}\rangle_{G}^{J}$. Pick a strongly measurable function
$u(t)$ supported on $(0,1],$ with values in the space $A_{1},$ satisfying
\eqref{represetation1}, and such that
\begin{equation}
\Vert J(s,u(s);\vec{A})\Vert_{G}\leq2\Vert a\Vert_{\langle\vec{A}\rangle
_{G}^{J}}\text{.} \label{represetation2}%
\end{equation}
Recall that $\vec{A}\mapsto\vec{A}_{\theta,q}^{J\blacktriangleleft}$ is an
exact interpolation functor with characteristic function $t^{\theta}$, and,
moreover, $0<\xi(t)\leq1/2$ for all $0<t\leq1$. Therefore, applying
successively \eqref{equiv}, \eqref{exact functor}, \cite[Theorem~3.2.2]{BL}
and the inequality $t^{-\xi(t)}\leq\sqrt{e}$ (see \eqref{aux1}), we obtain
that, for each $q\in\lbrack1,\infty]$ and $0<t\leq1$,%
\begin{align*}
\Vert u(t)\Vert_{\langle\vec{A}\rangle_{\xi(t),q}^{J\blacktriangleleft}}  &
\leq C\Vert u(t)\Vert_{\vec{A}_{\xi(t),q}^{{J\blacktriangleleft}}}\leq C\Vert
u(t)\Vert_{\vec{A}_{\xi(t),1}^{J}}\\
&  \leq C_{1}t^{-\xi(t)}J(t,u(t);\vec{A})\leq C_{1}\sqrt{e}J(t,u(t);\vec
{A}),\text{{}}%
\end{align*}
where $C_{1}>0$ does not depend on $t$ and $q$. Combining with
\eqref{represetation2}, yields
\[
\Vert a\Vert_{\vec{A}_{\xi,q,G}^{J}}\leq\left\Vert \Vert u(t)\Vert
_{\langle\vec{A}\rangle_{\xi(t),q}^{J\blacktriangleleft}}\right\Vert _{G}\leq
C_{1}\sqrt{e}\Vert J(t,u(t);\vec{A})\Vert_{G}\leq2C_{1}\sqrt{e}\Vert
a\Vert_{\langle\vec{A}\rangle_{G}^{J}}\text{,}%
\]
concluding the proof of \eqref{1 emb}.

Next, we prove that $\langle\vec{A}\rangle_{G}^{J}$ is dense in $\vec{A}%
_{\xi,q,G}^{J}$. Let $a\in\vec{A}_{\xi,q,G}^{J}$. Then we can find a
representation \eqref{represetation1}, where $u(t):(0,1]\rightarrow A_{1}$ is
a strongly measurable function such that $\Vert u(t)\Vert_{\langle\vec
{A}\rangle_{\xi(t),q}^{J\blacktriangleleft}}\in{G}$, and
\begin{equation}
\Vert a\Vert_{\vec{A}_{\xi,q,G}^{J}}>\frac{1}{2}\left\Vert \Vert
u(t)\Vert_{\langle\vec{A}\rangle_{\xi(t),q}^{J\blacktriangleleft}}\right\Vert
_{G}. \label{repNew1}%
\end{equation}
We shall show in a moment that for each $\delta>0$
\begin{equation}
a_{\delta}:=\int_{\delta}^{1}u(t)\,\frac{dt}{t}\in\langle\vec{A}\rangle
_{G}^{J}. \label{repNew2}%
\end{equation}
This given, writing
\[
a-a_{\delta}=\int_{0}^{1}u(t)\cdot\chi_{(0,\delta)}(t)\,\frac{dt}{t},
\]
and using the fact that the lattice $G$ is separable, we see that for every
$\varepsilon>0$ we can find $\delta>0$ such that
\[
{\Vert a-a_{\delta}\Vert}_{\vec{A}_{\xi,q,G}^{J}}\leq\left\Vert \Vert
u(t)\Vert_{\langle\vec{A}\rangle_{\xi(t),q}^{J\blacktriangleleft}}\cdot
\chi_{(0,\delta)}(t)\right\Vert _{G}<\varepsilon.
\]
We now return to the proof of \eqref{repNew2}. First of all, observe that if
$\delta\leq t\leq1$, then $1/2\geq\xi(t)\geq\eta:=(2\log(e/\delta))^{-1}>0$.
Therefore, for each $a_{1}\in A_{1}$ and $t\in\lbrack\delta,1]$ we have
\[
\Vert a_{1}\Vert_{\langle\vec{A}\rangle_{\xi(t),q}^{J\blacktriangleleft}}%
\geq(q^{\prime}\xi(t)(1-\xi(t)))^{-1/q^{\prime}}\Vert a_{1}\Vert_{\langle
\vec{A}\rangle_{\eta,q}^{J}}\geq e^{-1/e}\Vert a_{1}\Vert_{\langle\vec
{A}\rangle_{\eta,q}^{J}}.
\]
Moreover, since $A_{1}\subset A_{0}$, then by Proposition \ref{Aprop2} and
\cite[Theorem 3.4.1 (b)]{BL} it follows that%

\[
\langle\vec{A}\rangle_{\eta,q}^{J}=\vec{A}_{\eta,q}^{J}\subset\vec{A}%
_{\eta,\infty}^{J}=\langle\vec{A}\rangle_{\eta,\infty}^{J},
\]
with constants that depend only on $\delta$ and $q$.
Hence,
\[
\Vert a_{1}\Vert_{\langle\vec{A}\rangle_{\xi(t),q}^{J\blacktriangleleft}}\geq
c_{\delta,q}\Vert a_{1}\Vert_{\langle\vec{A}\rangle_{\eta,\infty}^{J}}.
\]
Combining this with \eqref{repNew1}, yields
\[
\left\Vert \Vert u(t)\Vert_{\langle\vec{A}\rangle_{\eta,\infty}^{J}}\cdot
\chi_{(\delta,1)}(t)\right\Vert _{G}<2c_{\delta,q}^{-1}\Vert a\Vert_{\vec
{A}_{\xi,q,G}^{J}}.
\]
It follows that, for every $\delta\leq t\leq1$ there exists a strongly
measurable function $v(\cdot,t):(0,1]\rightarrow A_{1}$ such that
\[
u(t)=\int_{0}^{1}v(s,t)\,\frac{ds}{s},
\]
and
\begin{equation}
\left\Vert \sup_{0<s\leq1}(s^{-\eta}J(s,v(s,t;\vec{A})))\cdot\chi_{(\delta
,1)}(t)\right\Vert _{G}<2c_{\delta,q}^{-1}\Vert a\Vert_{\vec{A}_{\xi,q,G}^{J}%
}. \label{repNew3}%
\end{equation}
We shall now verify that the integral
\[
\int_{0}^{1}\int_{\delta}^{1}v(s,t)\,\frac{dt}{t}\,\frac{ds}{s}%
\]
is absolutely convergent in $A_{0}$. Indeed, by \eqref{repNew3} and the fact
that $G\subset L_{1}[0,1](\frac{dt}{t})$, it follows that
\begin{align*}
\int_{0}^{1}\int_{\delta}^{1}{\Vert v(s,t)\Vert}_{A_{0}}\,\frac{dt}{t}%
\,\frac{ds}{s}  &  \leq\int_{0}^{1}s^{\eta}\int_{\delta}^{1}\sup_{0<s\leq
1}(s^{-\eta}J(s,v(s,t;\vec{A})))\,\frac{dt}{t}\,\frac{ds}{s}\\
&  \leq C{\left\Vert \sup_{0<s\leq1}(s^{-\eta}J(s,v(s,t;\vec{A})))\cdot
\chi_{(\delta,1)}(t)\right\Vert }_{G}\int_{0}^{1}\frac{ds}{s^{1-\eta}}\\
&  \leq2C(\eta c_{\delta,q})^{-1}\Vert a\Vert_{\vec{A}_{\xi,q,G}^{J}}<\infty.
\end{align*}
Therefore
\[
a_{\delta}=\int_{\delta}^{1}\int_{0}^{1}v(s,t)\,\frac{ds}{s}\,\frac{dt}%
{t}=\int_{0}^{1}\int_{\delta}^{1}v(s,t)\,\frac{ds}{s}\,\frac{dt}{t},
\]
and we can write
\begin{equation}
a_{\delta}=\int_{0}^{1}w(s)\,\frac{ds}{s},\quad\mbox{where}\;\;w(s):=\int%
_{\delta}^{1}v(s,t)\,\frac{dt}{t}. \label{repNew4}%
\end{equation}
It is easy to see that $w(s):(0,1]\rightarrow A_{1}$. Indeed, using
successively Minkowski's inequality and the embedding $G\subset L_{1}%
[0,1](\frac{dt}{t}),$ for each $s\in(0,1]$ we have
\begin{align*}
{\Vert w(s)\Vert}_{A_{1}}  &  \leq\int_{\delta}^{1}{\Vert v(s,t)\Vert}_{A_{1}%
}\,\frac{dt}{t}\leq\frac{1}{s}\int_{\delta}^{1}J(s,v(s,t);\vec{A})\,\frac
{dt}{t}\\
&  \leq\frac{C}{s}{\left\Vert \sup_{0<s\leq1}(s^{-\eta}J(s,v(s,t;\vec
{A})))\cdot\chi_{(\delta,1)}(t)\right\Vert }_{G}\\
&  \leq2C(sc_{\delta,q})^{-1}\Vert a\Vert_{\vec{A}_{\xi,q,G}^{J}}<\infty.
\end{align*}

Furthermore, by Minkowski's inequality, Lemma \ref{lemmaH}, and the embedding
$G\subset L_{1}[0,1](\frac{dt}{t}),$ we obtain
\begin{align*}
{\Vert J(s,w(s);\vec{A})\Vert}_{G}  &  \leq\int_{\delta}^{1}{\Vert
J(s,v(s,t);\vec{A})\Vert}_{G}\,\frac{dt}{t}\\
&  \leq C^{\prime}\int_{\delta}^{1}\sup_{0<s\leq1}(s^{-\eta}J(s,v(s,t);\vec
{A})\,\frac{dt}{t}\\
&  \leq CC^{\prime}{\left\Vert \sup_{0<s\leq1}(s^{-\eta}J(s,v(s,t);\vec
{A})\cdot\chi_{(\delta,1)}(t)\right\Vert }_{G}\\
&  \leq2CC^{\prime}c_{\delta,q}^{-1}\Vert a\Vert_{\vec{A}_{\xi,q,G}^{J}%
}<\infty.
\end{align*}
Thus, taking into account \eqref{repNew4}, we see that $a_{\delta}\in
\langle\vec{A}\rangle_{G}^{J}$, and \eqref{repNew2} is proved. This concludes
the proof that $\langle\vec{A}\rangle_{G}^{J}$ is dense in $\vec{A}_{\xi
,q,G}^{J}$, and therefore from \eqref{1 emb} it follows that
\begin{equation}
(\vec{A}_{\xi,q,G}^{J})^{\ast}\subset(\langle\vec{A}\rangle_{G}^{J})^{\ast}.
\label{repNewNew1}%
\end{equation}
Now, our next aim will be to prove the converse embedding.

On account of the fact that $A_{1}$ is dense in $A_{0}$, it follows that
$(A_{1}^{\ast},A_{0}^{\ast})$ is an ordered pair. We shall show that
\begin{equation}
\langle A_{1}^{\ast},A_{0}^{\ast}\rangle_{G^{\prime}}^{K}\subset(\vec{A}%
_{\xi,q,G}^{J})^{\ast}. \label{2 emb}%
\end{equation}

Let $0<\theta<1$ and $1\leq q\leq\infty,$ be arbitrary (but fixed). Let
$a\in\langle\vec{A}\rangle_{\theta,q}^{J}$. Then we can find a strongly
measurable function $u(t):(0,1]\rightarrow$ $A_{1}$ such that $a=\int_{0}%
^{1}u(s)\frac{ds}{s},$ and, moreover,
\begin{equation}
\Big(\int_{0}^{1}\left(  s^{-\theta}J(s,u(s);\vec{A})\right)  ^{q}\,\frac
{ds}{s}\Big)^{1/q}\leq2\Vert a\Vert_{\langle\vec{A}\rangle_{\theta,q}^{J}%
}\text{.} \label{represetation3}%
\end{equation}
Let $b\in A_{0}^{\ast},$ then, by duality and H\"{o}lder's inequality,%
\begin{align*}
\left\vert \langle b,a\rangle\right\vert  &  =\left\vert \langle b,\int%
_{0}^{1}u(s)\frac{ds}{s}\rangle\right\vert \leq\int_{0}^{1}\left\vert \langle
b,u(s)\rangle\right\vert \frac{ds}{s}\\
&  \leq\int_{0}^{1}K(1/s,b;A_{0}^{\ast},A_{1}^{\ast})\cdot J(s,u(s);\vec
{A})\,\frac{ds}{s}\\
&  \leq\Big(\int_{0}^{1}\left(  s^{\theta}K(1/s,b;A_{0}^{\ast},A_{1}^{\ast
})\right)  ^{q^{\prime}}\,\frac{ds}{s}\Big)^{1/q^{\prime}}\Big(\int_{0}%
^{1}\left(  s^{-\theta}J(s,u(s);\vec{A})\right)  ^{q}\,\frac{ds}{s}%
\Big)^{1/q}\\
&  \leq\Big(\int_{0}^{1}\left(  s^{\theta-1}K(s,b;A_{1}^{\ast},A_{0}^{\ast
})\right)  ^{q^{\prime}}\,\frac{ds}{s}\Big)^{1/q^{\prime}}\Big(\int_{0}%
^{1}\left(  s^{-\theta}J(s,u(s);\vec{A})\right)  ^{q}\,\frac{ds}{s}%
\Big)^{1/q}.
\end{align*}
Combining with \eqref{represetation3} yields
\begin{equation}
\left\vert \langle b,a\rangle\right\vert \leq2\Vert b\Vert_{\langle
A_{1}^{\ast},A_{0}^{\ast}\rangle_{1-\theta,q^{\prime}}^{K\blacktriangleleft}%
}\Vert a\Vert_{\langle\vec{A}\rangle_{\theta,q}^{J\blacktriangleleft}},
\label{EQ1}%
\end{equation}
for $0<\theta<1$, $1\leq q\leq\infty$, and for all $a\in\vec{A}_{\theta,q}%
^{J}$ and $b\in A_{0}^{\ast}$.

Now, let $a\in\vec{A}_{\xi,q,G}^{J}$, pick a strongly measurable function
$u(t):(0,1]\rightarrow$ $A_{1},$ such that \eqref{represetation1} holds, and,
moreover,
\begin{equation}
\left\Vert \Vert u(t)\Vert_{\langle\vec{A}\rangle_{\xi(t),q}%
^{J\blacktriangleleft}}\right\Vert _{G}\leq2\Vert a\Vert_{\vec{A}_{\xi
,q,G}^{J}}. \label{EQ2}%
\end{equation}
Observe that $1-\xi(t)=\theta(t),$ $0<t\leq1$. Therefore, by \eqref{EQ1} and
\eqref{EQ2}, we see that for each $b\in A_{0}^{\ast}$
\begin{align}
\left\vert \langle b,a\rangle\right\vert  &  \leq\int_{0}^{1}|\langle
b,u(t)\rangle|\,\frac{dt}{t}\leq2\int_{0}^{1}\Vert b\Vert_{\langle A_{1}%
^{\ast},A_{0}^{\ast}\rangle_{\theta(t),q^{\prime}}^{K\blacktriangleleft}}\Vert
u(t)\Vert_{\langle\vec{A}\rangle_{\xi(t),q}^{J\blacktriangleleft}}\,\frac
{dt}{t}\nonumber\\
&  =2\int_{0}^{1}t\Vert b\Vert_{\langle A_{1}^{\ast},A_{0}^{\ast}%
\rangle_{\theta(t),q^{\prime}}^{K\blacktriangleleft}}\Vert u(t)\Vert
_{\langle\vec{A}\rangle_{\xi(t),q}^{J\blacktriangleleft}}\,\frac{dt}{t^{2}%
}\nonumber\\
&  \leq2\left\Vert t\cdot\Vert b\Vert_{\langle A_{1}^{\ast},A_{0}^{\ast
}\rangle_{\theta(t),q^{\prime}}^{K\blacktriangleleft}}\right\Vert _{G^{\prime
}}\cdot\left\Vert \Vert u(t)\Vert_{\langle\vec{A}\rangle_{\xi(t),q}%
^{J\blacktriangleleft}}\right\Vert _{G}\nonumber\\
&  \leq4\left\Vert t\cdot\Vert b\Vert_{\langle A_{1}^{\ast},A_{0}^{\ast
}\rangle_{\theta(t),q^{\prime}}^{K\blacktriangleleft}}\right\Vert _{G^{\prime
}}\cdot\Vert a\Vert_{\vec{A}_{\xi,q,G}^{J}}. \label{nec}%
\end{align}
One can easily verify that the operator $Rf(t)=f(\sqrt{t})$ is bounded on $G$
if and only if the operator $Tf=f(t^{2})/t$ is bounded on $G^{\prime}$. Thus,
from Theorem~\ref{StrExtrAbst} applied to the ordered pair $(A_{1}^{\ast
},A_{0}^{\ast})$ we obtain,%
\[
\left\Vert t\cdot\Vert b\Vert_{\langle A_{1}^{\ast},A_{0}^{\ast}%
\rangle_{\theta(t),q^{\prime}}^{K\blacktriangleleft}}\right\Vert _{G^{\prime}%
}\leq C\Vert b\Vert_{\langle A_{1}^{\ast},A_{0}^{\ast}\rangle_{G^{\prime}}%
^{K}}.
\]
The preceding inequality combined with (\ref{nec}) yields that for all $b\in
A_{0}^{\ast}$, $a\in\vec{A}_{\xi,q,G}^{J},$
\[
\left\vert \langle b,a\rangle\right\vert \leq C\Vert b\Vert_{\langle
A_{1}^{\ast},A_{0}^{\ast}\rangle_{G^{\prime}}^{K}}\Vert a\Vert_{\vec{A}%
_{\xi,q,G}^{J}}.
\]
This estimate extends to all $b\in\langle A_{1}^{\ast},A_{0}^{\ast}%
\rangle_{G^{\prime}}^{K}$ on account of the fact that, by assumption,
$A_{0}^{\ast}$ is dense in $\langle A_{1}^{\ast},A_{0}^{\ast}\rangle
_{G^{\prime}}^{K}.$ Therefore, the proof of \eqref{2 emb} is complete.

Furthermore, since $G$ is separable it follows that $A_{1}$ is dense in
$\langle\vec{A}\rangle_{G}^{J}$ and consequently $(\langle\vec{A}\rangle
_{G}^{J})^{\ast}=\langle A_{1}^{\ast},A_{0}^{\ast}\rangle_{G^{\prime}}^{K}$
(cf. \cite{BK}). Thus, from embeddings \eqref{repNewNew1} and \eqref{2 emb} it
follows that $(\langle\vec{A}\rangle_{G}^{J})^{\ast}=(\vec{A}_{\xi,q,G}%
^{J})^{\ast}$. Since $\langle\vec{A}\rangle_{G}^{J}$ is embedded into $\vec
{A}_{\xi,q,G}^{J}$ as a dense subset, we can invoke the Hahn-Banach Theorem to
conclude that, as we wished to show, $\langle\vec{A}\rangle_{G}^{J}=\vec
{A}_{\xi,q,G}^{J}$.
\end{proof}

\begin{remark}
\label{K/J formulas}Let $1\leq q<\infty$. As it is shown in \cite[Theorem~7.6]%
{CF-CKU} (generalizing formula \eqref{equivJ}), for every ordered pair
$\vec{A}$ the modified $K$-interpolation space $\langle\vec{A}\rangle
_{0,q}^{K}$ (see Definition \ref{def:order}) can be obtained also by using the
$J$-method. Namely, the latter space coincides (with equivalence of norms)
with the space $\langle\vec{A}\rangle_{0,q,\log}^{J}$, generated by the space
$L_{\log}^{q}$ on $((0,1),ds/s)$, normed by%
\[
\left\Vert f\right\Vert _{L_{\log}^{q}}=\Big(\int_{0}^{1}(|f(s)|\log
(e/s))^{q}\,\frac{ds}{s}\Big)^{1/q}.
\]
It is easy to see that the operator $Rf(t)=f(\sqrt{t})$ is bounded on the
lattice $L_{\log}^{q}$. Therefore, applying Theorem \ref{StrExtrAbst(J)} to
$G=L_{\log}^{q}$ with $1<q<\infty$, we get an extrapolation description of the
spaces $\langle\vec{A}\rangle_{0,q}^{K}$.
\end{remark}

\section{Reiteration properties of limiting interpolation
spaces\label{sec:reiteration}}

Throughout this section ${L}^{p}:={L}^{p}[0,1]$, $1\leq p\leq\infty$.

\begin{theorem}
\label{TH2} Let $F$ be an interpolation Banach lattice between ${L}^{\infty}$
and ${L}^{\infty}(1/t)$ on $[0,1]$. The following conditions are equivalent:

(i) the operator $Tf(t)=f(t^{2})/t$ is bounded on $F$;

(ii) for all ordered pairs $\vec{A}=(A_{0},A_{1})$, and for all $0<\theta<1$
and $1\leq q\leq\infty$ we have
\[
{\langle}\vec{A}_{\theta,q},A_{1}{\rangle}_{F}^{K}={\langle\vec{A}\rangle
_{F}^{K}}%
\]
(with constants depending on $\theta$ and $q$);

(iii) for every ordered pair $\vec{A}=(A_{0},A_{1})$ there exist $\theta
_{1},\theta_{2}\in(0,1)$, $\theta_{1}\neq\theta_{2}$, and $1\leq q_{1}%
,q_{2}\leq\infty$ such that
\[
{\langle}\vec{A}_{\theta_{1},q_{1}},A_{1}{\rangle}_{F}^{K}={\langle}\vec
{A}_{\theta_{2},q_{2}},A_{1}{\rangle}_{F}^{K};
\]

(iv) there exist $\theta_{1},\theta_{2}\in(0,1)$, $\theta_{1}\neq\theta_{2}$,
and $1\leq q_{1},q_{2}\leq\infty$ such that
\[
{\langle}(L^{1},{L}^{\infty})_{\theta_{1},q_{1}},{L}^{\infty}{\rangle}_{F}%
^{K}=\langle(L^{1},{L}^{\infty})_{\theta_{2},q_{2}},{L}^{\infty}\rangle
_{F}^{K}\text{.}%
\]

\end{theorem}

\begin{proof}
$(i)\Rightarrow(ii)$. Since $A_{1}\subset A_{0}$, it follows that $\vec
{A}_{\theta,q}\subset A_{0}$ for all $0<\theta<1,1\leq q\leq\infty.$
Consequently,
\[
{\langle}\vec{A}_{\theta,q},A_{1}{\rangle}_{F}^{K}\subset{\langle\vec
{A}\rangle_{F}^{K}.}%
\]
Since $\vec{A}_{\theta,1}\subset$ $\vec{A}_{\theta,q},1\leq q\leq\infty,$ to
prove the converse embedding it will suffice to show that ${\langle\vec
{A}\rangle_{F}^{K}}\subset{\langle}\vec{A}_{\theta,1},A_{1}{\rangle}_{F}^{K}.$

Given $\theta\in(0,1)$ pick $m\in\mathbb{N}$ such that $2^{-m}\leq
1-\theta<2^{1-m}\text{.}$ Then, for all $0<t\leq1$ we have
\[
t^{2^{m}}\leq t^{1/(1-\theta)}<t^{2^{m-1}}\text{.}%
\]
Using Holmstedt's formula (see \cite{Holm} or \cite[Corollary~3.6.2(b)]{BL})
we obtain
\begin{align}
K(t,a;\vec{A}_{\theta,1}^{K},A_{1})  &  \asymp\int_{0}^{t^{1/(1-\theta)}%
}s^{-\theta}K(s,a;\vec{A})\,\frac{ds}{s}\nonumber\\
&  =\int_{t^{2^{m}}}^{t^{1/(1-\theta)}}s^{-\theta}K(s,a;\vec{A})\,\frac{ds}%
{s}\nonumber\\
&  +\sum_{n=m}^{\infty}\int_{t^{2^{n+1}}}^{t^{2^{n}}}s^{-\theta}K(s,a;\vec
{A})\,\frac{ds}{s}. \label{EQ100}%
\end{align}
We estimate separately each of the terms from the right-hand side of
\eqref{EQ100}, assuming that $0<t\leq1/2$.

For the first integral we use the concavity of $K(s,a;\vec{A})$ with respect
to $s$ \cite[Lemma~3.1.1]{BL} and the definition of $T,$ to obtain
\[
\int_{t^{2^{m}}}^{t^{1/(1-\theta)}}s^{-\theta}K(s,a;\vec{A})\,\frac{ds}{s}%
\leq\frac{1}{1-\theta}\frac{K(t^{2^{m}},a;\vec{A})}{t^{2^{m}}}\cdot t=\frac
{1}{1-\theta}T^{m}\left(  K(t,a;\vec{A})\right)  \text{.}%
\]
Similarly, since $(1-\theta)2^{n}\geq1$ for all $n\in\mathbb{N}$ such that
$n\geq m$, we have
\begin{align*}
\int_{t^{2^{n+1}}}^{t^{2^{n}}}s^{-\theta}K(s,a;\vec{A})\,\frac{ds}{s}  &
\leq\frac{1}{1-\theta}\frac{K(t^{2^{n+1}},a;\vec{A})}{t^{2^{n+1}}}\cdot t\cdot
t^{(1-\theta)2^{n}-1}\\
&  \leq\frac{1}{1-\theta}T^{n+1}\left(  K(t,a;\vec{A})\right)  2^{1-(1-\theta
)2^{n}}.
\end{align*}
Inserting these estimates in \eqref{EQ100} we find that there exists a
constant $C,$ that depends only on $\theta,$ such that for all $0<t\leq1/2$,
\begin{equation}
K(t,a;\vec{A}_{\theta,1}^{K},A_{1})\leq C\left(  T^{m}\left(  K(t,a;\vec
{A})\right)  +\sum_{n=m}^{\infty}T^{n+1}\left(  K(t,a;\vec{A})\right)
2^{1-(1-\theta)2^{n}}\right)  . \label{dob}%
\end{equation}
On account of the fact that the $K-$functional is a concave function we have
that, for all $t\in(0,1),$
\[
K(t,a;\vec{A}_{\theta,1}^{K},A_{1})\leq2K(t/2,a;\vec{A}_{\theta,1}^{K}%
,A_{1}).
\]
Consequently, by (\ref{dob})
\begin{align*}
\left\Vert K(t,a;\vec{A}_{\theta,1}^{K},A_{1})\right\Vert _{F}  &
\leq2\left\Vert K(t/2,a;\vec{A}_{\theta,1}^{K},A_{1})\chi_{\lbrack
0,1]}(t)\right\Vert _{F}\\
&  \leq2C\left(  \Vert T\Vert^{m}+\sum_{n=m+1}^{\infty}\Vert T\Vert
^{n}2^{1-(1-\theta)2^{n-1}}\right)  \Vert K(t,a;\vec{A})\Vert_{F}.
\end{align*}
Thus, with constants depending on $\theta$ and $q,$ we have
\[
{\langle\vec{A}\rangle_{F}^{K}}\subset{\langle}\vec{A}_{\theta,1}^{K}%
,A_{1}{\rangle}_{F}^{K}\subset{\langle}\vec{A}_{\theta,q}^{K},A_{1}{\rangle
}_{F}^{K},
\]
as we wished to show.

Since the implications $(ii)\Rightarrow(iii)$ and $(iii)\Rightarrow(iv)$ are
immediate, it is only left to prove

$(iv)\Rightarrow(i)$. Without loss of generality we can assume that
{$\theta_{1}<\theta_{2}$. Choose $\tilde{\theta}_{1}$ and $\tilde{\theta}_{2}$
so that $\theta_{1}<\tilde{\theta}_{1}<\tilde{\theta}_{2}<\theta_{2}$. Then,
since $L^{\infty}\subset L^{1}$, by \cite[Theorem 3.4.1 (c), (d)]{BL} we
obtain
\[
(L^{1},L^{\infty})_{\theta_{2},q_{2}}\subset(L^{1},L^{\infty})_{\tilde{\theta
}_{2},1/(1-\tilde{\theta}_{2})}\subset(L^{1},L^{\infty})_{\tilde{\theta}%
_{1},1/(1-\tilde{\theta}_{1})}\subset(L^{1},L^{\infty})_{\theta_{1},q_{1}}.
\]
Therefore, setting $q:=1/(1-\tilde{\theta}_{1})$ and $p:=1/(1-\tilde{\theta
}_{2})$, we have $1<q<p<\infty$ and, by $(iv)$,
\begin{equation}
{\langle L}^{p},{L}^{\infty}{\rangle}_{{F}}^{K}={\langle L}^{q},{L}^{\infty
}{\rangle}_{{F}}^{K}\text{{}}. \label{EQ110a}%
\end{equation}
}

We now show that the operator $S_{r}f(t):=f(t^{r})$ is bounded from ${L}^{p}$
into ${L}^{q},$ if $1<r<p/q$. Indeed, by H\"{o}lder's inequality, we have%
\begin{align*}
\Vert S_{r}f\Vert_{q}^{q}  &  =\int_{0}^{1}|f(s^{r})|^{q}\;ds=\frac{1}{r}%
\int_{0}^{1}u^{1/r-1}|f(u)|^{q}\;du\\
&  \leq\frac{1}{r}\left(  \int_{0}^{1}u^{\frac{(1-r)p}{r(p-q)}}\;du\right)
^{\frac{p-q}{p}}\left(  \int_{0}^{1}|f(u)|^{p}\;du\right)  ^{q/p},
\end{align*}
where $\int_{0}^{1}u^{\frac{(1-r)p}{r(p-q)}}\,du<\infty$ on account of the
fact that $1<r<p/q$. Moreover, it is plain that for all $r>1$ the operator
$S_{r}$ is bounded on ${L}^{\infty}$. Hence, by interpolation and
\eqref{EQ110a},%
\[
S_{r}:{\langle L}^{p},{L}^{\infty}{\rangle}_{{F}}^{K}\rightarrow{\langle
L}^{q},{L}^{\infty}{\rangle}_{{F}}^{K}={\langle L}^{p},{L}^{\infty}{\rangle
}_{{F}}^{K},
\]
whenever $1<r<p/q$. Consequently, if we let $X:={\langle L}^{p},{L}^{\infty
}{\rangle}_{{F}}^{K},$ then $S_{r}$ is bounded on $X$ for $1<r<p/q$. We now
show that this implies that $S_{r}$ is bounded on $X$ for all $r>1.$ Indeed,
if $r\geq p/q>1,$ we can select $n\in\mathbb{N}$ large enough so that
$r^{1/n}\in(1,p/q)$. Then $S_{r^{1/n}}:X\rightarrow X$ and we conclude since
$S_{r}=(S_{r^{1/n}})^{n}.$

Next, we exploit the properties of the lattice $F$ to prove that the
definition of $X$ self improves to
\begin{equation}
X={\langle}L^{1},{L}^{\infty}{\rangle}_{{F}}^{K}\text{.} \label{EQ111a}%
\end{equation}
It is immediate that $X={\langle L}^{p},{L}^{\infty}{\rangle}_{{F}}^{K}%
\subset{\langle}L^{1},{L}^{\infty}{\rangle}_{{F}}^{K}.$ To prove the opposite
embedding we use an equivalent expression for $K(t,g;{L}^{p},{L}^{\infty})$
\cite[Theorem~5.2.1]{BL} as follows
\begin{align*}
K(t,f(s^{1/p});{L}^{p},{L}^{\infty})  &  \asymp\left(  \int_{0}^{t^{p}}\left[
f^{\ast}(s^{1/p})\right]  ^{p}\,ds\right)  ^{1/p}=p^{1/p}\left(  \int_{0}%
^{t}\left[  f^{\ast}(u)\right]  ^{p}u^{p-1}\,du\right)  ^{1/p}\\
&  \leq p^{1/p}\left(  \int_{0}^{t}f^{\ast}(u)\,du\right)  ^{1/p}\left(
\underset{0<u\leq t}{\sup}uf^{\ast}(u)\right)  ^{(p-1)/p}\\
&  \leq p^{1/p}\left(  \int_{0}^{t}f^{\ast}(u)\,du\right)  ^{1/p}\left(
\int_{0}^{t}f^{\ast}(u)du\right)  ^{(p-1)/p}\\
&  \leq p^{1/p}\int_{0}^{t}f^{\ast}(u)\,du\\
&  =p^{1/p}K(t,f;L^{1},{L}^{\infty})\text{.}%
\end{align*}
Consequently, since $S_{p}$ is bounded on $X$, we obtain
\begin{align*}
\Vert f\Vert_{X}  &  =\Vert S_{p}S_{1/p}(f)\Vert_{X}\leq\Vert S_{p}\Vert\Vert
K(t,f(s^{1/p});{L}^{p},{L}^{\infty})\Vert_{{F}}\\
&  \leq p^{1/p}\Vert S_{p}\Vert\Vert K(t,f;L^{1},{L}^{\infty})\Vert_{{F}%
}\text{,}%
\end{align*}
whence
\[
{\langle}L^{1},{L}^{\infty}{\rangle}_{{F}}^{K}\subset X\text{,}%
\]
and the proof of \eqref{EQ111a} is completed.

Further, from \eqref{EQ111a} and the boundedness of $S_{2}$ on $X$ we obtain
\[
\left\Vert \int_{0}^{t}S_{2}f^{\ast}(s)\,ds\right\Vert _{{F}}\leq\Vert
S_{2}f\Vert_{X}\leq\Vert S_{2}\Vert\Vert f\Vert_{X}\leq C\left\Vert \int%
_{0}^{t}f^{\ast}(s)\,ds\right\Vert _{{F}}\text{.}%
\]
On the other hand,
\[
\int_{0}^{t}S_{2}f^{\ast}(s)\,ds=\frac{1}{2}\int_{0}^{t^{2}}u^{-1/2}f^{\ast
}(u)\,du\geq\frac{1}{2t}\int_{0}^{t^{2}}f^{\ast}(u)\,du\text{,}%
\]
and therefore
\begin{equation}
\left\Vert \frac{1}{t}\int_{0}^{t^{2}}f^{\ast}(s)\,ds\right\Vert _{{F}}%
\leq2C\left\Vert \int_{0}^{t}f^{\ast}(s)\,ds\right\Vert _{{F}}\text{.}
\label{concava}%
\end{equation}

We now re-interpret (\ref{concava}) as the boundedness of the operator $T$ on
a class of concave functions$.$ Indeed, since $K(t,f;L^{1},{L}^{\infty}%
)=\int_{0}^{t}f^{\ast}(s)\,ds,$ we have
\[
\frac{1}{t}\int_{0}^{t^{2}}f^{\ast}(s)\,ds=T(K(\cdot,f;L^{1},{L}^{\infty
}))(t).
\]
Furthermore, since the pair $(L^{1},{L}^{\infty})$ is $Conv_{0}$-abundant on
$[0,1]$ (cf. Example \ref{exampleabundant} above), we can rewrite
\ (\ref{concava}) as follows: for every concave increasing function $f\in{F}$
such that $\underset{t\rightarrow0}{\lim}f(t)=0$
\begin{equation}
\Vert Tf\Vert_{{F}}\leq2C\Vert f\Vert_{{F}}. \label{EQ112a}%
\end{equation}
We will show in a moment that
(\ref{EQ112a}) holds for all concave increasing functions from $F$. Then, the
interpolation property of $F$ guarantees (by the same argument we gave in the
course of the proof of Theorem~\ref{ExtrMinLat} above) that $T$ is bounded on
all of $F$, concluding the proof. To prove the above claim we argue by
contradiction. Suppose that there exists a concave increasing function
$f_{0}\in{F}$ such that $\underset{t\rightarrow0}{\lim}f_{0}(t)>0$. Then
clearly ${F}={L}^{\infty}$. But by Proposition \ref{Aprop1} and the first
equation from (\ref{K1}), we have
\begin{align*}
\langle(L^{1},{L}^{\infty})_{\theta,q},{L}^{\infty}\rangle_{{L}^{\infty}}^{K}
&  =((L^{1},{L}^{\infty})_{\theta,q},{L}^{\infty})_{{0,\infty}}^{K}\\
&  =(L^{1},{L}^{\infty})_{\theta,q}^{K},
\end{align*}
for every $0<\theta<1$ and $1\leq q\leq\infty$. Therefore $(iv)$ fails in this
case, which gives a contradiction. So $F\neq{L}^{\infty}$ and therefore
\eqref{EQ112a} holds for each concave increasing function $f\in{F,}$ as we
wished to show.
\end{proof}

Since the pair $(L^{1},{L}^{\infty})$ is $Conv_{0}$-abundant (cf. Example
\ref{exampleabundant} above), from Theorems~\ref{ExtrMinLat} and~\ref{TH2} we obtain

\begin{corollary}
\label{cor2} Suppose $F$ is an interpolation Banach lattice of functions with
respect to the pair $({L}^{\infty},{L}^{\infty}(1/t))$ on $[0,1]$. Then, the
following conditions are equivalent:

(i) the operator $T f (t) =f (t^{2})/t$ is bounded on $F$;

(ii) for every Banach pair $\vec{A}=(A_{0},A_{1})$ and every $1\leq
q\leq\infty$ the equivalence \eqref{BasEq} holds for the space {$\langle
\vec{A}\rangle_{F}^{K}$};

(iii) for every Banach pair $\vec{A}=(A_{0},A_{1})$, we have, with constants
independent of $a\in\langle\vec{A}\rangle_{F}^{K}+\tilde{A}_{1}$, $s>0$ and
$1\leq q\leq\infty$,
\[
K(s,a;{\langle\vec{A}\rangle_{F}^{K}},\tilde{A}_{1})\asymp K(s,\Vert
a\Vert_{\langle\vec{A}\rangle_{\theta(t),q}^{K\blacktriangleleft}}%
;F,{L}^{\infty}(1/t)),
\]
where $\theta(t)=1+\frac{1}{2\log(t/e)}$;

(iv) for every ordered pair $\vec{A}=(A_{0},A_{1})$, and for all $\theta
\in(0,1)$ and $1\leq q\leq\infty$ we have
\[
{\langle\vec{A}\rangle_{F}^{K}}={\langle}\vec{A}_{\theta,q},A_{1}{\rangle}%
_{F}^{K};
\]

(v) for every ordered pair $\vec{A}=(A_{0},A_{1})$, there exist $\theta
_{1},\theta_{2}\in(0,1)$, $\theta_{1}\neq\theta_{2}$, and $1\leq q_{1}%
,q_{2}\leq\infty$ such that
\[
{\langle}\vec{A}_{\theta_{1},q_{1}},A_{1}{\rangle}_{F}^{K}={\langle}\vec
{A}_{\theta_{2},q_{2}},A_{1}{\rangle}_{F}^{K};
\]

(vi) there exist $\theta_{1},\theta_{2}\in(0,1)$, $\theta_{1}\neq\theta_{2}$,
and $1\leq q_{1},q_{2}\leq\infty$ such that
\[
{\langle}(L^{1},{L}^{\infty})_{\theta_{1},q_{1}},{L}^{\infty}{\rangle}_{F}%
^{K}={\langle}(L^{1},{L}^{\infty})_{\theta_{2},q_{2}},{L}^{\infty}{\rangle
}_{F}^{K}\text{.}%
\]

\end{corollary}

\begin{remark}
Applying functors ${\langle}\cdot,\cdot{\rangle}_{F}^{K}$, with $F$ satisfying
the conditions of Corollary~\ref{cor2}, to the pair $(L^{1}[0,1],L^{\infty
}[0,1])$, we obtain the so-called \textit{strong extrapolation} rearrangement
invariant spaces introduced and studied in \cite{AL2009,AL,AL2017}. The latter
ones can be characterized via the boundedness of the operator $S_{2}%
f(t):=f(t^{2})$ (see the proof of Theorem~\ref{TH2}). In particular, in this
connection it is instructive to compare Corollary~\ref{cor2} with Theorem~4.3
from \cite{AL2017}.
\end{remark}

Next, we investigate the reiteration properties of limiting interpolation
spaces which are close to the larger "end point" of an ordered pair.{\ To
avoid imposing extra conditions needed to use duality (see
Theorem~\ref{StrExtrAbst(J)}), we prefer to handle this case using limiting
interpolation $K$-functors. As was observed in the Introduction, the simplest
functor of such a form, $\vec{A}\mapsto\langle\vec{A}\rangle_{0,1}^{K}$, was
introduced and studied in the paper \cite{GM} (see also \cite{CF-CKU}). The
following theorem is a far-reaching generalization and a refinement of these
results (cf. Remark~\ref{non-interpolation} below). }

\begin{theorem}
\label{limiting reiteration} Let $G$ be an interpolation Banach lattice on
$[0,1]$ with respect to the pair $({L}^{\infty},{L}^{\infty}(1/t))$. The
following conditions are equivalent:

(a) the operator $Rf(t):=f(t^{1/2})$ is bounded on $G$;

(b) for every ordered pair $\vec{A}=(A_{0},A_{1})$ and for every $\theta
\in(0,1)$, $1\leq q\leq\infty,$ we have
\[
{\langle}\vec{A}{\rangle}_{G}^{K}={\langle}A_{0},\vec{A}_{\theta,q}{\rangle
}_{G}^{K};
\]

(c) for every ordered pair $\vec{A}=(A_{0},A_{1})$ there exist $\theta
_{1},\theta_{2}\in(0,1)$, $\theta_{1}\neq\theta_{2}$, $1\leq q_{1},q_{2}%
\leq\infty,$ such that
\[
{\langle}A_{0},\vec{A}_{\theta_{1},q_{1}}{\rangle}_{G}^{K}={\langle}A_{0}%
,\vec{A}_{\theta_{2},q_{2}}{\rangle}_{G}^{K};
\]

(d) there exist $\theta_{1},\theta_{2}\in(0,1)$, $\theta_{1}\neq\theta_{2}$,
$1\leq q_{1},q_{2}\leq\infty,$ such that
\[
{\langle}L^{1},(L^{1},{L}^{\infty})_{\theta_{1},q_{1}}{\rangle}_{G}%
^{K}={\langle}L^{1},(L^{1},{L}^{\infty})_{\theta_{2},q_{2}}{\rangle}_{G}%
^{K}\text{.}%
\]

\end{theorem}

\begin{proof}
The implications $(b)\Rightarrow(c)$ and $(c)\Rightarrow(d)$ are
straightforward. Therefore we only need to prove the implications
$(a)\Rightarrow(b)$ and $(d)\Rightarrow(a)$.

$(a)\Rightarrow(b)$. First, observe that for arbitrary $p>1$, the restriction
of the operator $R_{p}f(t):=f(t^{1/p})$ to the set of all increasing functions
is bounded on $G$. Indeed, by iteration, we see that together with $R$ the
operator $R_{2^{m}}=R^{m}$ is bounded on $G$ for each $m\in\mathbb{N}$. Given
$p>1$ we pick $m$ such that $2^{m}\geq p$. Let $g$ be an increasing function
$g\in G$, then on account of the fact that $t^{1/p}\leq t^{2^{-m}}$ for $0\leq
t\leq1$, we obtain
\[
\Vert R_{p}g\Vert_{G}\leq\Vert R_{2^{m}}g\Vert_{G}\leq\Vert R_{2^{m}}%
\Vert\Vert g\Vert_{G}\text{.}%
\]
Let $f\in A_{0}+A_{1},\theta\in(0,1),$ and set $g=K(t,f;\vec{A}).$ The
previous inequality implies that
\[
\Vert K(t,f;\vec{A})\Vert_{G}=\Vert R_{1/\theta}(K(t^{1/\theta},f;\vec
{A}))\Vert_{G}\leq\Vert R_{1/\theta}\Vert\Vert K(t^{1/\theta},f;\vec{A}%
)\Vert_{G}\text{.}%
\]
On the other hand, by Holmstedt's formula, we have
\[
K(t,f;A_{0},\vec{A}_{\theta,\infty}^{K})\asymp t\underset{t^{1/\theta}\leq
s\leq1}{\sup}s^{-\theta}K(s,f;\vec{A})\geq K(t^{1/\theta},f;\vec{A})\text{,
for all }0<t\leq1.
\]
It follows that%
\[
\Vert K(t,f;A_{0},\vec{A}_{\theta,q}^{K})\Vert_{G}\geq\Vert K(t,f;A_{0}%
,\vec{A}_{\theta,\infty}^{K})\Vert_{G}\geq c\Vert K(t^{1/\theta},f;\vec
{A})\Vert_{G}\text{, for all }1\leq q\leq\infty.
\]
Thus,
\[
{\langle}A_{0},\vec{A}_{\theta,q}^{K}{\rangle}_{G}^{K}\subset\vec{A}_{G}%
^{K}\text{.}%
\]
This gives the desired result since the converse embedding is obvious.

$(d)\Rightarrow(a)$. Proceeding exactly as in the proof of Theorem~\ref{TH2},
we see that $(d)$ guarantees the existence of $1<q<p<\infty$ such that
\begin{equation}
{\langle L}^{1},{L}^{p}{\rangle}_{G}^{K}={\langle L}^{1},{L}^{q}{\rangle}%
_{G}^{K}\text{.} \label{5.7}%
\end{equation}
It will be useful now to consider a family of auxiliary operators defined as
follows. For each $r>1,$ we let $Q_{r}$ be the operator defined by
$Q_{r}f(s):=s^{1/r-1}f(s^{1/r})$, $0<s\leq1$, $f\in$ ${L}^{1}={L}^{1}[0,1].$
By an easy change of variables we see that the operators $Q_{r}$ are bounded
on ${L}^{1},$ for every $r>1$ . We will now show that there exists $r_{0}>1$
such that, for all $1<r<r_{0}$, $Q_{r}$ is a bounded operator, $Q_{r}:$
${L}^{p}\rightarrow{L}^{q}$. In fact, we can let $r_{0}:=\frac{q(p-1)}%
{p(q-1)}=\frac{q^{\prime}}{p^{\prime}}>1.$ Indeed, since for $1<r<r_{0}%
:=\frac{q(p-1)}{p(q-1)},$ we have $\int_{0}^{1}u^{\frac{(r-1)(1-q)p}{p-q}%
}\;du<\infty,$ H\"{o}lder's inequality yields
\begin{align*}
\Vert Q_{r}f\Vert_{q}^{q}  &  =\int_{0}^{1}(s^{1/r-1})^{q}|f(s^{1/r}%
)|^{q}\;ds=r\int_{0}^{1}u^{(r-1)(1-q)}|f(u)|^{q}\;du\\
&  \leq r\left(  \int_{0}^{1}u^{\frac{(r-1)(1-q)p}{p-q}}\;du\right)
^{\frac{p-q}{p}}\left(  \int_{0}^{1}|f(u)|^{p}\;du\right)  ^{q/p}\text{.}%
\end{align*}
Thus, $\Vert Q_{r}f\Vert_{q}\leq C\Vert f\Vert_{p}$, with a constant $C$ that
depends only on $p,q,r$. Moreover, since $Q_{r}$ is also bounded on $L^{1},$
it follows by interpolation and (\ref{5.7}) that $Q_{r}$ is bounded on the
space $X:={\langle L}^{1},{L}^{p}{\rangle}_{G}^{K}$, for all $1<r<r_{0}.$
Observe that for each $r>1$ and $k\in\mathbb{N}$, we have $Q_{r}^{k}=Q_{r^{k}%
},$ whence $Q_{r}$ is actually bounded on $X$ for all $r>1$. Furthermore, it
is plain that $X$ is also an interpolation space with respect to the pair
$({L}^{1},{L}^{\infty})$, whence $X$ is a rearrangement invariant space. As a
matter of fact we now show that, more precisely, $X$ can be described by
\begin{equation}
X={\langle L}^{1},{L}^{\infty}{\rangle}_{G}^{K}\text{.} \label{EQ111}%
\end{equation}
The preceding discussion shows that
comparing norms of the spaces from (\ref{EQ111}) we may assume without loss of
generality that $f=f^{\ast}.$ Let $p^{\prime}$ be defined as usual by
$1/p+1/{p^{\prime}}=1.$ Then, by Holmstedt's formula for the pair $({L}%
^{1},{L}^{p})$ (cf. \cite[Ch. 5. Section 7, Problem 2, page 124]{BL}), we
have
\begin{align*}
K(t,Q_{p^{\prime}}f^{\ast};{L}^{1},{L}^{p})  &  \succeq\int_{0}^{t^{p^{\prime
}}}(Q_{p^{\prime}}f^{\ast})^{\ast}(s)\;ds=\int_{0}^{t^{p^{\prime}}%
}s^{1/p^{\prime}-1}f^{\ast}(s^{1/p^{\prime}})\;ds\\
&  =p^{\prime}\int_{0}^{t}f^{\ast}(u)\;du\\
&  =p^{\prime}K(t,f;{L}^{1},{L}^{\infty})\text{.}%
\end{align*}
Thus, for some $c>0$
\[
\Vert Q_{p^{\prime}}f\Vert_{X}=\Vert K(t,Q_{p^{\prime}}f;{L}^{1},{L}^{p}%
)\Vert_{G}\geq cp^{\prime}\Vert K(t,f;{L}^{1},{L}^{\infty})\Vert_{G}\text{.}%
\]
On the other hand, on account of the fact that $Q_{p^{\prime}}$ is bounded on
$X$, we obtain
\[
\Vert Q_{p^{\prime}}f\Vert_{X}\leq\Vert Q_{p^{\prime}}\Vert\Vert f\Vert
_{X}=\Vert Q_{p^{\prime}}\Vert\Vert K(t,f;{L}^{1},{L}^{p})\Vert_{G}\leq\Vert
Q_{p^{\prime}}\Vert\Vert K(t,f;{L}^{1},{L}^{\infty})\Vert_{G}\text{,}%
\]
and (\ref{EQ111}) follows.


From \eqref{EQ111}, \cite[Theorem~5.2.1]{BL} and the boundedness of $Q_{2}$ on
$X$ we have
\[
\left\Vert \int_{0}^{t}(Q_{2}f^{\ast})^{\ast}(s)\;ds\right\Vert _{G}\leq
C^{\prime}\Vert Q_{2}f^{\ast}\Vert_{X}\leq C^{\prime}\Vert Q_{2}\Vert\Vert
f\Vert_{X}\leq C\left\Vert \int_{0}^{t}f^{\ast}(s)\;ds\right\Vert _{G}\text{.}%
\]
Combining this estimate with the equation
\[
\int_{0}^{t}(Q_{2}f^{\ast})^{\ast}(s)\;ds=\int_{0}^{t}s^{1/2}f^{\ast}%
(s^{1/2})\;\frac{ds}{s}=2\int_{0}^{t^{1/2}}f^{\ast}(s)\;ds\text{,}%
\]
we arrive at the inequality
\[
\left\Vert \int_{0}^{t^{1/2}}f^{\ast}(s)\;ds\right\Vert _{G}\leq{C}\left\Vert
\int_{0}^{t}f^{\ast}(s)\,ds\right\Vert _{G}\text{.}%
\]
Equivalently, setting $g(t):=\int_{0}^{t}f^{\ast}(s)\,ds$, we have
\[
\Vert Rg\Vert_{G}\leq C\Vert g\Vert_{G}.
\]
Hence, the latter inequality holds for every concave increasing function $g\in
G$ such that $\underset{t\rightarrow0}{\lim}g(t)=0$. Observe that $(a)$ holds
if $G={L}^{\infty}$. Therefore, we can assume that $R$, restricted to the set
of all concave increasing functions from $G$, is bounded on $G$. Thus,
proceeding in the same way as in the proof of Theorem~\ref{ExtrMinLat}, we can
show that the operator $R$ is bounded on all of $G$.
\end{proof}

\begin{remark}
\label{non-interpolation}A straightforward inspection shows that the
implication $(i)\Rightarrow(ii)$ (resp. $(a)\Rightarrow(b)$) of Theorem
\ref{TH2} (resp. Theorem \ref{limiting reiteration}) holds under the weaker
assumptions that $G$ is a Banach lattice on $[0,1]$ such that $G\supset
{L}^{\infty}\cap{L}^{\infty}(1/t)$ and the operator $Rf(t)=f(t^{1/2})$ is
bounded on ${G}$. Clearly, the Banach lattice $L^{1}([0,1],\frac{dt}{t})$
satisfies the above conditions. Therefore, since ${\langle}\vec{A}{\rangle
}_{0,1}^{K}={\langle}\vec{A}{\rangle}_{L^{1}([0,1],\frac{dt}{t})}^{K}$, then
from the implication $(a)\Rightarrow(b)$ of Theorem \ref{limiting reiteration}
it follows, in particular, the reiteration formula \eqref{intr1}.
\end{remark}

Observe that one of the main motivations for the reiteration theorem obtained
in \cite{GM} was the following interpolation theorem of Zygmund (cf.
\cite{AZYG}) originally stated for the periodic Hilbert transform.

\begin{theorem}
\label{zyg} Let $T$ be a{\ quasilinear} operator defined on the space
$L^{1}:=L^{1}(0,1)$ such that $T$ is of weak type $(1,1)$ and strong type
$(p,p),$ for some $p>1.$ Then $T:LLogL\rightarrow L^{1}$.
\end{theorem}

Applying Theorem~\ref{limiting reiteration}, we can extend the latter result
to a rather wide class of limiting interpolation spaces close to $L^{1}$.
Recall that $f^{\ast\ast}(t):=\frac{1}{t}\int_{0}^{t}f^{\ast}(s)\,ds$, and the
space $L^{p,\infty}$, $1\leq p<\infty$, consists of all measurable functions
on $[0,1]$ such that
\[
\Vert f\Vert_{L^{p,\infty}}:=\sup_{0<t\leq1}t^{1/p}f^{\ast}(t)<\infty.
\]

\begin{theorem}
Under the assumptions of Theorem \ref{zyg}, suppose that $G$ is a Banach
lattice on $[0,1]$ such that $G\supset{L}^{\infty}\cap{L}^{\infty}(1/t)$ and
the operation $f(t)\mapsto f(\sqrt{t})$ is bounded on $G$. Then we have
\[
T:\,\langle L^{1},L^{\infty}\rangle_{G}^{K}\rightarrow X_{G},
\]
where%
\[
\langle L^{1},L^{\infty}\rangle_{G}^{K}=\{f:\left\Vert tf^{\ast\ast
}(t)\right\Vert _{G}<\infty\},\text{ and }X_{G}:=\{f:\left\Vert tf^{\ast
}(t)\right\Vert _{G}<\infty\}.
\]

\end{theorem}

\begin{remark}
Note that if $G=L^{1}([0,1],\frac{dt}{t})$, then $X_{G}=L^{1}$ and $\langle
L^{1},L^{\infty}\rangle_{G}^{K}=\langle L^{1},L^{\infty}\rangle_{0,1}%
^{K}=LLogL$.
\end{remark}

\begin{proof}
By interpolation,%
\[
T:\,\langle L^{1},L^{p}\rangle_{G}^{K}\rightarrow\langle L^{1,\infty
},L^{p,\infty}\rangle_{G}^{K}.
\]
Since (cf. \cite[Theorem~5.2.1]{BL})
\[
L^{p}=(L^{1},L^{\infty})_{1/p^{\prime},p}^{K},
\]
where $p^{\prime}=p/(p-1)$, it follows from Theorem \ref{limiting reiteration}
that%
\begin{align*}
\langle L^{1},L^{p}\rangle_{G}^{K}  &  =\langle L^{1},(L^{1},L^{\infty
})_{1/p^{\prime},p}^{K}\rangle_{G}^{K}\\
&  =\langle L^{1},L^{\infty}\rangle_{G}^{K}\\
&  =\{f:\left\Vert tf^{\ast\ast}(t)\right\Vert _{G}<\infty\}.
\end{align*}
Next we deal with the range space $\langle L^{1,\infty},L^{p,\infty}%
\rangle_{G}^{K}$. Since $L^{1,\infty}$ is not a Banach space, we cannot apply
Theorem \ref{limiting reiteration} to the pair $(L^{1,\infty},L^{p,\infty}).$
Instead, we show directly that%
\[
\langle L^{1,\infty},L^{p,\infty}\rangle_{G}^{K}\subset\{f:\left\Vert
tf^{\ast}(t)\right\Vert _{G}<\infty\},
\]
by means of proving that%
\begin{equation}
\left\Vert f\right\Vert _{\langle L^{1,\infty},L^{p,\infty}\rangle_{G}^{K}%
}\geq c\left\Vert tf^{\ast}(t)\right\Vert _{G}, \label{verify1}%
\end{equation}
using the $G-$boundedness of the operator $R_{1/p^{\prime}}%
f(t)=f(t^{1/p^{\prime}})$ restricted to the set of increasing functions.
Indeed, writing $L^{p,\infty}=(L^{1,\infty},L^{\infty})_{1/p^{\prime},\infty}$
(cf. \cite[Theorem~5.3.1]{BL}) and applying Holmstedt's formula (cf.
\cite[Corollary~3.6.2(a)]{BL}), for each $0<t\leq1$ we have
\begin{align*}
K(t,f;L^{1,\infty},L^{p,\infty})  &  \asymp K(t,f;L^{1,\infty},(L^{1,\infty
},L^{\infty})_{1/p^{\prime},\infty})\\
&  \asymp t\underset{t^{p^{\prime}}\leq s\leq1}{\sup}s^{-1/p^{\prime}%
}K(s,f;L^{1,\infty},L^{\infty})\\
&  \geq K(t^{p^{\prime}},f;L^{1,\infty},L^{\infty})\text{.}%
\end{align*}
By the proof of [$(a)\Rightarrow(b)]$ in Theorem \ref{limiting reiteration},
we know that the operator $R_{1/p^{\prime}}f(t)=f(t^{1/p^{\prime}})$
restricted to the set of increasing functions is $G-$bounded. Consequently,
since $K(t^{p^{\prime}},f;L^{1,\infty},L^{\infty})$ is increasing, there
exists an absolute constant $c>0$ such that%
\[
\left\Vert f\right\Vert _{\langle L^{1,\infty},L^{p,\infty}\rangle_{G}^{K}%
}=\left\Vert K(t,f;L^{1,\infty},L^{p,\infty})\right\Vert _{G}\geq c\left\Vert
K(t,f;L^{1,\infty},L^{\infty})\right\Vert _{G}.
\]
But it is well known (cf. \cite{tor} and the references therein) that%
\begin{align*}
K(t,f;L^{1,\infty},L^{\infty})  &  \asymp\sup_{s<t}\{sf^{\ast}(s)\}\\
&  \geq tf^{\ast}(t).
\end{align*}
Hence, (\ref{verify1}) follows concluding the proof.
\end{proof}

\begin{remark}
\label{parameters} As was observed in Introduction, the theory developed in
this paper provides a unified roof to a large body of literature devoted to
the study of particular examples of limiting interpolation spaces. We now
briefly discuss the connections with our work. Let $\vec{A}=(A_{0},A_{1})$ be
a Banach pair such that $A_{1}\subset A_{0}$. Given $b>0$ and $1\leq
q\leq\infty,$ we define the Banach lattice $F_{b,q}$ normed by
\[
\Vert f\Vert_{F_{b,q}}:=\left(  \int_{0}^{1}\left(  \frac{|f(s)|}{s(1-\log
s)^{b}}\right)  ^{q}\frac{ds}{s}\right)  ^{1/q}\;\;if\;\;1\leq q<\infty,
\]
and
\[
\Vert f\Vert_{F_{b,\infty}}:=\sup_{0<s\leq1}\frac{|f(s)|}{s(1-\log s)^{b}}.
\]
The spaces $\langle\vec{A}\rangle_{F_{b,\infty}}^{K}$ have been introduced,
using a different notation, by Cobos, Fern\'{a}ndez-Cabrera, Manzano, and
Mart\'{\i}nez in \cite[see Theorem~2.6]{Cobos1}). Later, Cobos, Fern{\'{a}}%
ndez-Cabrera, K\"{u}hn, and Ullrich \cite{CF-CKU}, considered various
properties of the family of the spaces $F_{1,q}$, $1<q\leq\infty.$ In the
paper \cite{Cobos-S2} one can find a more general construction of limiting
interpolation spaces using powers of iterated logarithms. Let $L_{1}(s)=\log
s$ and $L_{j}(s)=\log(L_{j-1}(s))$, $j>1$. For any $\bar{\alpha}=(\alpha
_{1},\dots,\alpha_{m})\in\mathbb{R}^{m}$ we set $\phi_{\bar{\alpha}}%
(s):=\prod_{j=1}^{m}L_{j}(s)^{\alpha_{j}}$. Then, under suitable conditions on
$\bar{\alpha}$, we can define the Banach lattice $F_{\bar{\alpha},q}^{\prime}$
consisting of all measurable functions $f$ on $[0,c_{m}]$ ($c_{m}\in(0,1)$
depends only on $m$) such that
\[
\Vert f\Vert_{F_{\bar{\alpha},q}^{\prime}}:=\left(  \int_{0}^{c_{m}}\left(
\frac{|f(s)|}{s\phi_{\bar{\alpha}}(s)}\right)  ^{q}\frac{ds}{s}\right)
^{1/q}\;\;if\;\;1\leq q<\infty
\]
and
\[
\Vert f\Vert_{F_{\bar{\alpha},\infty}^{\prime}}:=\sup_{0<s\leq c_{m}}%
\frac{|f(s)|}{s\phi_{\bar{\alpha}}(s)}.
\]
In \cite{Cobos-K}, Cobos and K\"{u}hn introduced the corresponding spaces
$\langle\vec{A}\rangle_{F_{\bar{\alpha},q}^{\prime}}^{K}$ and showed their
description by means of the $J$-functional.

In the papers \cite{Cobos-D, Cobos-D2, Cobos-D3}, the authors have considered
constructions, which give limiting interpolation spaces that are close to the
larger end point space of an ordered pair, i.e., to the space $A_{0}$. For any
$b\in\mathbb{R}$, and $1\leq q\leq\infty,$ we define the Banach lattice
$G_{b,q}$ normed by
\[
\Vert f\Vert_{G_{b,q}}:=\left(  \int_{0}^{1}\left(  \frac{|f(s)|}{(1-\log
s)^{b}}\right)  ^{q}\frac{ds}{s}\right)  ^{1/q}\;\;if\;\;1\leq q<\infty,
\]
and
\[
\Vert f\Vert_{G_{b,\infty}}:=\sup_{0<s\leq1}\frac{|f(s)|}{(1-\log s)^{b}}.
\]
It is easy to see that the operator $Tf(t)=f(t^{2})/t$ (resp. $Rf(t)=f(\sqrt
{t})$) is bounded on the lattices $F_{b,q}$ and $F_{\bar{\alpha},q}^{\prime}$
(resp. $R$ is bounded on $G_{b,q}$). To illustrate the issues that are
involved we shall now verify that $T$ is bounded on $F_{b,q}$. Indeed, after a
change of variables we have
\begin{align*}
\Vert Tf\Vert_{F_{b,q}}  &  =\left(  \int_{0}^{1}\left(  \frac{|f(s^{2}%
)|}{s^{2}(1-\log s)^{b}}\right)  ^{q}\frac{ds}{s}\right)  ^{1/q}%
=2^{-1/q}\left(  \int_{0}^{1}\left(  \frac{|f(t)|}{t(1-\log(\sqrt{t}))^{b}%
}\right)  ^{q}\frac{dt}{t}\right)  ^{1/q}\\
&  \leq2^{(b-1)/q}\left(  \int_{0}^{1}\left(  \frac{|f(t)|}{t(1-\log t)^{b}%
}\right)  ^{q}\frac{dt}{t}\right)  ^{1/q}=2^{(b-1)/q}\Vert f\Vert_{F_{b,q}}.
\end{align*}
Let us finally observe that some classes of limiting interpolation spaces were
also introduced and studied in the case of unordered pairs (see, for instance,
\cite{Cobos2}, \cite{Cobos-Segurado}, and \cite{Cobos-Segurado2}).
Nevertheless, usually these spaces can be represented as intersections of the
spaces obtained using the constructions described above.
\end{remark}

\section{Extrapolation description of limiting Schatten-von Neumann operator
classes and a generalization of Matsaev's theorem\label{sec:scha}}

We apply some of the results obtained in this paper to the Banach pair
$(\ell^{1},\ell^{\infty}):=$ $(\ell^{1}(\mathbb{N)},\ell^{\infty}(\mathbb{N)})
$ {\ (For a different approach, based on the direct estimation of the
$\ell^{p}$-norms involved, we refer to \cite{Lykov2018}).}

Recall that by Calder\'{o}n's theorem (cf. \cite{bs}) the pair $(\ell^{1}%
,\ell^{\infty})$ is $K-$monotone; consequently, if $X:=X(\mathbb{N)}$ is an
interpolation space with respect to $(\ell^{1},\ell^{\infty})$ there exists a
Banach sequence lattice $F:=F(\mathbb{N)}$ such that, uniformly on
$a=\{a_{k}\}\in X,$
\[
\Vert a\Vert_{X}\asymp\left\Vert {\{K(n,\{a_{k}\};\ell^{1},\ell^{\infty}%
)\}}_{n}\right\Vert _{F}.
\]
The corresponding $K-$functional is given by
\[
K(n,\{a_{k}\};\ell^{1},\ell^{\infty})=\sum_{j=1}^{n}a_{j}^{\ast}%
,\text{\ \ }n\in\mathbb{N}\text{, }\{a_{k}\}\in\ell^{\infty},
\]
where $\{a_{j}^{\ast}\}$ is the non-increasing rearrangement of the sequence
$\{|a_{k}|\}$. In other words, $X$ is an interpolation symmetric sequence
space with respect to $(\ell^{1},\ell^{\infty})$ if and only if there exists
Banach sequence lattice $F$ such that for all $a=\{a_{k}\}\in X$
\begin{equation}
\Vert a\Vert_{X}\asymp\Biggl\Vert\Bigl\{\sum_{j=1}^{n}a_{j}^{\ast}%
\Bigr\}_{n}\Biggr\Vert_{F}\text{.} \label{interp for seq}%
\end{equation}
Moreover, by Hardy, and reverse Hardy inequalities (cf. \cite[Lemma~2.1]{M-97}
or \cite[Example~7]{M-03}) it follows that for all $1<p<\infty$
\[
\Vert a\Vert_{\ell^{p}}\leq\Vert a\Vert_{(\ell^{1},\ell^{\infty}%
)_{1-1/p,p}^{K\blacktriangleleft}}=\Vert a\Vert_{(\ell^{\infty},\ell
^{1})_{1/p,p}^{K\blacktriangleleft}}\leq e\Vert a\Vert_{\ell^{p}},
\]
and, according to Remarks \ref{Rem1}\footnote{Since $\ell^{1}\subset
\ell^{\infty}$.} and \ref{Zam1a}, for all $1<p\leq2$
\[
\frac{1}{2}\Vert a\Vert_{(\ell^{\infty},\ell^{1})_{1/p,p}^{K\blacktriangleleft
}}\leq\Vert a\Vert_{\langle\langle\ell^{1},\ell^{\infty}\rangle\rangle
_{1-1/p,p}^{K\blacktriangleleft}}=\Vert a\Vert_{\langle\ell^{\infty},\ell
^{1}\rangle_{1/p,p}^{K\blacktriangleleft}}\leq\Vert a\Vert_{(\ell^{\infty
},\ell^{1})_{1/p,p}^{K\blacktriangleleft}}.
\]
Therefore, applying Theorem~\ref{StrExtrAbst} (see Remark \ref{remDiscrete})
gives the following result.


\begin{theorem}
\label{theorPosl} Let $X$ be a symmetric sequence space satisfying
\eqref{interp for seq} and let the operator
$S(\{a_{n}\}):=\{a_{n^{2}}\}$ be bounded on $F$. Then
\[
\Vert a\Vert_{X}\asymp\left\Vert {\{\Vert a\Vert_{\ell^{p(n)}}\}}%
_{n}\right\Vert _{F}\text{,}%
\]
where $p(n)=\frac{2\log(en)}{2\log(en)-1}$, $n\in\mathbb{N}$.
\end{theorem}

Now, suppose that $\mathcal{H}$ is a separable complex Hilbert space. Recall
that the \textit{Schatten-von Neumann class} ${{\mathfrak{S}}}^{p}$ consists
of all compact operators $T:\;\mathcal{H}\rightarrow\mathcal{H}$ such that the
norm
\[
\Vert T\Vert_{{{\mathfrak{S}}}^{p}}:=\left(  \sum_{j=1}^{\infty}s_{j}%
(T)^{p}\right)  ^{\frac{1}{p}}<\infty,
\]
where $\{s_{j}(T)\}_{j=1}^{\infty}$ is the non-increasing sequence of
$s$-numbers of $T$ determined by the Schmidt expansion (cf. \cite{GK}). The
Schatten-von Neumann classes belong to the larger family of two-sided
symmetrically normed ideals. Let $\alpha>0$ and let $\mathcal{M}^{\alpha}$ be
the dual ideal to the so-called \textit{Matsaev class}, i.e., the ideal of
bounded compact operators in a Hilbert space $\mathcal{H}$, provided with the
norm
\[
\Vert T\Vert_{\mathcal{M}^{\alpha}}:=\underset{n\in\mathbb{N}}{\sup}\frac
{\sum_{j=1}^{n}s_{j}(T)}{\log^{\alpha}(en)}\text{.}%
\]
Then, it is known that for each $p_{0}>1$ we have
\[
\Vert T\Vert_{\mathcal{M}^{\alpha}}\asymp\underset{1<p<p_{0}}{\sup
}(p-1)^{\alpha}\Vert T\Vert_{p}%
\]
(see \cite[5.2]{M94}). The ideal $\mathcal{M}^{\alpha}$ is a typical limiting
interpolation space with respect to the pair $({{\mathfrak{S}}}^{1}%
,{{\mathfrak{S}}}^{\infty})$, where by ${{\mathfrak{S}}}^{\infty}$ we denote
the class of all compact operators on $\mathcal{H}$ with the usual operator
norm. Theorem~\ref{theorPosl} allows to easily get the following general
extrapolation description of the limiting interpolation spaces with respect to
the pair $({{\mathfrak{S}}}^{1},{{\mathfrak{S}}}^{\infty})$.

\begin{theorem}
\label{Ideals} Let $\mathcal{X}$ be a symmetrically normed ideal of compact
operators on a Hilbert space $\mathcal{H}$, and let
\[
\Vert T\Vert_{\mathcal{X}}\asymp\Biggl\Vert\Bigr\{\sum_{j=1}^{n}%
s_{j}(T)\Bigl\}_{n}\Biggr\Vert_{F}\text{,}%
\]
where $F$ is a sequence Banach lattice such that the operator $S(\{a_{n}%
\}):=\{a_{n^{2}}\}$ is bounded on $F$.

Then
\[
\Vert T\Vert_{\mathcal{X}} \asymp\left\Vert {\{\Vert T\Vert_{p (n)}\}}%
_{n}\right\Vert _{F}\text{,}%
\]
where $p (n) =\frac{2\log(en)}{2\log(en) -1}$, $n\in\mathbb{N}$,
and $\Vert T\Vert_{p (n)}$ are the corresponding Schatten-von Neumann norms of
$T$.
\end{theorem}

\begin{proof}
Applying Theorem~\ref{theorPosl} to the equivalent expression of the ideal
norm of $T$ in $\mathcal{X}$ via the sequence of $s$-numbers of $T$, we have
\[
\Vert T\Vert_{\mathcal{X}}\asymp\left\Vert {\{\Vert\{s_{j}(T)\} \Vert
_{\ell^{p(n)}}\}}_{n}\right\Vert _{F}=\left\Vert {\{\Vert T\Vert_{{p(n)}}%
\}}_{n}\right\Vert _{F}\text{.}%
\]

\end{proof}

As an application of Theorem~\ref{Ideals} we present a result on the
boundedness, in some appropriate symmetrically normed ideals, of the mapping
which transfers the imaginary component of a Volterra operator into its real
component. Recall that a Volterra operator is a compact operator whose
spectrum coincides with the one-point set $\{0\}$. Denote by $T_{\mathcal{R}}$
and $T_{\mathcal{J}}$ the real and the imaginary components of an operator
$T$, respectively, i.e.,
\[
T_{\mathcal{R}}:=\frac{1}{2}(T+T^{\ast})\text{\quad}\text{and}\text{\quad
}T_{\mathcal{J}}:=\frac{1}{2i}(T-T^{\ast})\text{.}%
\]
By a well-known result of Matsaev (cf. \cite[Theorem~III.6.2]{GK1}), if $T$ is
a Volterra operator, then for each $1<p<\infty,$ we have that $T_{\mathcal{J}%
}\in{{\mathfrak{S}}}^{p}$ implies that $T_{\mathcal{R}}\in{{\mathfrak{S}}}%
^{p}$. In addition, by \cite[Theorem~III.6.3]{GK1},
\begin{equation}
\Vert T_{\mathcal{R}}\Vert_{p}\leq\max\left\{  \frac{p}{p-1},\;p\right\}
\cdot\Vert T_{\mathcal{J}}\Vert_{p},\text{\quad}1<p<\infty\text{.}
\label{MatcaevNer}%
\end{equation}
Moreover, if $T_{\mathcal{J}}\in{{\mathfrak{S}}}^{1}$ then $T_{\mathcal{R}}%
\in\mathcal{M}^{1}$ \cite[Theorem~III.2.1]{GK1}. Combining the latter
inequality with Theorem~\ref{Ideals}, we can obtain some information related
to the behavior of both components of a Volterra operator in symmetrically
normed ideals "close" to the ideal ${{\mathfrak{S}}}^{1}$.

\begin{theorem}
\label{Matcaev} Suppose that a symmetrically normed ideal $\mathcal{X}$
satisfies the conditions of Theorem~ \ref{Ideals} and $T$ is a Volterra
operator such that $T_{\mathcal{J}}\in\mathcal{X}$. Then $T_{\mathcal{R}}%
\in\mathcal{X}(\log^{-1})$, where the ideal $\mathcal{X}(\log^{-1})$ consists
of all compact operators $U$ on $\mathcal{H}$ that satisfy
\[
\Vert U\Vert_{\mathcal{X}(\log^{-1})}:=\Biggl\Vert\Bigl\{\frac{1}{\log
(en)}\sum_{j=1}^{n}s_{j}(U)\Bigr\}_{n}\Biggr\Vert _{F}<\infty\text{.}%
\]

\end{theorem}

\begin{proof}
First, we observe that the boundedness of the operator $S(\{a_{n}%
\})=\{a_{n^{2}}\}$ on $F$ implies that $S$ is also bounded on the Banach
lattice of all sequences $\{a_{n}\}_{n=1}^{\infty}$ such that
\[
\biggl\Vert\Bigl\{\frac{a_{n}}{\log(en)}\Bigr\}_{n}\biggr\Vert_{F}%
<\infty\text{.}%
\]
Indeed,%

\begin{align*}
\biggl\Vert\Bigl\{\frac{a_{n^{2}}}{\log(en)}\Bigr\}_{n}\biggr\Vert_{F}  &
=2\biggl\Vert\Bigl\{\frac{a_{n^{2}}}{\log((en)^{2})}\Bigr\}_{n}\biggr\Vert_{F}%
\leq2{\left\|  \left\{  \frac{a_{n^{2}}}{\log(en^{2})}\right\}  _{n} \right\|
}_{F}\\
&  =2{\left\|  S\left(  \left\{  \frac{a_{n}}{\log(en)}\right\}  _{n} \right)
\right\|  }_{F} \leq C{\left\|  \left\{  \frac{a_{n}}{\log(en)}\right\}  _{n}
\right\|  }_{F}.
\end{align*}

Therefore, applying Theorem~\ref{Ideals} to the ideals $\mathcal{X}$ and
$\mathcal{X}(\log^{-1})$, we obtain
\[
\Vert T\Vert_{\mathcal{X}}\asymp\left\Vert {\{\Vert T\Vert_{{p(n)}}\}}%
_{n}\right\Vert _{F}\text{,}%
\]
and
\[
\Vert T\Vert_{\mathcal{X}(\log^{-1})}\asymp\left\Vert \Bigl\{\frac{1}%
{\log(en)}\Vert T\Vert_{{p(n)}}\Bigr\}_{n}\right\Vert _{F}\text{,}%
\]
where, as above, $p(n)=\frac{2\log(en)}{2\log(en)-1}$, $n\in\mathbb{N}$.
Finally, from the hypothesis $T_{\mathcal{J}}\in\mathcal{X}$ and inequality
\eqref{MatcaevNer} it follows that
\begin{align*}
{\Vert T_{\mathcal{R}}\Vert}_{\mathcal{X}(\log^{-1})}  &  \leq C{\left\Vert
\Bigl\{\frac{1}{\log(en)}{\Vert T_{\mathcal{R}}\Vert}_{p(n)}\Bigr\}_{n}%
\right\Vert }_{F}\\
&  \leq C{\left\Vert \Bigl\{\frac{1}{\log(en)}\cdot\frac{p(n)}{p(n)-1}%
\cdot{\Vert T_{\mathcal{J}}\Vert}_{p(n)}\Bigr\}_{n}\right\Vert }_{F}\\
&  =2C{\left\Vert \bigl\{{\Vert T_{\mathcal{J}}\Vert}_{p(n)}\bigr\}_{n}%
\right\Vert }_{F}\leq C^{\prime}{\Vert T_{\mathcal{J}}\Vert}_{\mathcal{X}}.
\end{align*}


\end{proof}

\section{Grand Lebesgue spaces via extrapolation\label{sec:Grand}}

{Let $1<p<\infty$.} The \textit{Grand Lebesgue $L^{p)}$ space} introduced by
T.~Iwaniec and C.~Sbordone \cite{IwSb}, consists of all measurable functions
$f$ on $[0,1]$ such that
\begin{equation}
{\Vert f\Vert}_{L^{p)}}:=\sup_{0<\varepsilon<p-1}\varepsilon^{\frac
{1}{p-\varepsilon}}{\Vert f\Vert}_{L^{p-\varepsilon}}<\infty. \label{spre}%
\end{equation}
These spaces have found many applications in analysis, including the study of
maximal operators, PDEs, interpolation theory, etc (see \cite{FFG,JSS} and the
references therein). On the other hand, the expression (\ref{spre}) is
somewhat difficult to work with. In this context, an important result of
Fiorenza-Karadzhov \cite[Theorem 4.2]{FiorKar} gives a concrete description of
the Grand Lebesgue spaces $L^{p)}$. The proof in \cite[Theorem 4.2]{FiorKar}
is based on extrapolation methods. In this section we give a simpler proof
through the use of Theorem~\ref{StrExtrAbst} combined with elementary
inequalities for rearrangements of functions.

\begin{theorem}
\cite[Theorem 4.2]{FiorKar} \label{Fiorenza-Karadzhov} Let $1<p<\infty$. Then,
with universal constants of equivalence
\[
{\Vert f\Vert}_{L^{p)}}\asymp\sup_{0<t<1}(\log(e/t))^{-\frac{1}{p}}\left(
\int\limits_{t}^{1}f^{\ast}(s)^{p}\,ds\right)  ^{\frac{1}{p}}.
\]

\end{theorem}

\begin{proof}
Let $F$ be the Banach lattice on $[0,1]$ equipped with the norm
\[
{\Vert f\Vert}_{F}:=\sup_{0<s\leq1}\frac{|f(s)|}{s\log^{1/p}(e/s)}.
\]
Clearly, the operator $Tf(s)=f(s^{2})/s$ is bounded on $F$. Consequently, for
every ordered pair $(A_{0},A_{1}),$ and for each continuous function
$q:\,(0,1]\rightarrow\lbrack1,\infty],$ we have, by Theorem~\ref{StrExtrAbst}
(see also Remark~\ref{Zam1}),
\begin{equation}
\Vert a\Vert_{{\langle\vec{A}\rangle_{F}^{K}}}\asymp\left\Vert t\cdot\Vert
a\Vert_{\langle\vec{A}\rangle_{\theta(t),q(t)}^{K\blacktriangleleft}%
}\right\Vert _{F}, \label{EQ10}%
\end{equation}
where $\theta(t)=1-\frac{1}{2\log(e/t)}$. Let $\vec{A}$ be the pair $({L}%
^{1},{L}^{p})$ and let
\begin{equation}
\theta=1-\frac{\varepsilon}{(p-1)(p-\varepsilon)}. \label{EQ11}%
\end{equation}
By Holmstedt's formula (cf. \cite[Theorem~4.3]{Holm}) we get
\[
{\langle{L}^{1},{L}^{p}\rangle_{\theta,p-\varepsilon}^{K\blacktriangleleft}%
}={L}^{p-\varepsilon},
\]
where the equivalence constants are independent of $\varepsilon\in
(0,\varepsilon_{0})$, $\varepsilon_{0}=\frac{p(p-1)}{p+1}<p-1$. Consider the
interpolation space ${\langle L}^{1},{L}^{p}{\rangle}_{F}^{K}$. From
\eqref{EQ10} and \eqref{EQ11} it follows that, with constants that depend only
on $p$, we have
\begin{align}
{\Vert f\Vert}_{\langle L^{1},L^{p}\rangle_{F}^{K}}  &  \asymp\sup_{0<t\leq
1}\frac{\Vert f\Vert_{\langle{L}^{1},{L}^{p}\rangle_{\theta(t),p-\varepsilon
(t)}^{K\blacktriangleleft}}}{\log^{1/p}(e/t)}\asymp\sup_{0<t\leq1}%
(1-\theta(t))^{\frac{1}{p}}{\Vert f\Vert}_{L^{p-\varepsilon(t)}}\nonumber\\
&  \asymp\sup_{0<\varepsilon\leq\varepsilon_{0}}\varepsilon^{\frac
{1}{p-\varepsilon}}{\Vert f\Vert}_{L^{p-\varepsilon}}\asymp{\Vert f\Vert
}_{L^{p)}}. \label{grand Leb}%
\end{align}

On the other hand, using the quasi-concavity of the $K$-functional and
Holmstedt's formula \cite[Theorem~4.1]{Holm}, we have
\begin{align*}
{\Vert f\Vert}_{{\langle L}^{1},{L}^{p}{\rangle}_{F}^{K}}  &  \asymp
\sup_{0<t\leq1}\frac{K(t/2,f;{L}^{1},{L}^{p})}{t\log^{1/p}(e/t)}\asymp
\sup_{0<t\leq1/2}\frac{K(t,f;{L}^{1},{L}^{p})}{t\log^{1/p}(e/t)}\\
&  \asymp\sup_{0<t\leq1/2}\frac{\int\limits_{0}^{t^{\frac{p}{p-1}}}f^{\ast
}(s)\,ds}{t\log^{1/p}(e/t)}+\sup_{0<t\leq1/2}\frac{\Bigg(\int\limits_{t^{\frac
{p}{p-1}}}^{1}f^{\ast}(s)^{p}\,ds\Bigg)^{\frac{1}{p}}}{\log^{1/p}(e/t)}.
\end{align*}
We now show that the first term on the right-hand side can be estimated by the
second term (see also \cite[Lemma 3.1]{FFGKR} with another proof). For this
purpose we let%
\[
A:=\sup_{0<t\leq1/2}\frac{1}{t\log^{1/p}(e/t)}\int\limits_{0}^{t^{\frac
{p}{p-1}}}f^{\ast}(s)\,ds.
\]
Further, we choose $M>0$ so large that $M>(M+1)2^{1/p-1},$ and pick $t_{0}%
\in(0,1/2]$ such that
\[
\frac{1}{t_{0}\log^{1/p}(e/t_{0})}\int\limits_{0}^{t_{0}^{\frac{p}{p-1}}%
}f^{\ast}(s)\,ds>\frac{M+1}{2^{1-1/p}M}\cdot A.
\]
Then, we have
\begin{equation}
\int\limits_{0}^{t_{0}^{\frac{2p}{p-1}}}f^{\ast}(s)\,ds\leq M\int%
\limits_{t_{0}^{\frac{2p}{p-1}}}^{t_{0}^{\frac{p}{p-1}}}f^{\ast}(s)\,ds.
\label{EQ12}%
\end{equation}
Indeed, if (\ref{EQ12}) does not hold, then in particular
\[
\int\limits_{0}^{t_{0}^{\frac{2p}{p-1}}}f^{\ast}(s)\,ds\geq\frac{M}{M+1}%
\int\limits_{0}^{t_{0}^{\frac{p}{p-1}}}f^{\ast}(s)\,ds,
\]
and since $t_{0}\leq1/2$ we see that%
\[
\frac{1}{t_{0}^{2}\log^{1/p}(e/t_{0}^{2})}\int\limits_{0}^{t_{0}^{\frac
{2p}{p-1}}}f^{\ast}(s)\,ds\geq\frac{2^{1-1/p}M}{(M+1)t_{0}\log^{1/p}(e/t_{0}%
)}\int\limits_{0}^{t_{0}^{\frac{p}{p-1}}}f^{\ast}(s)\,ds>A,
\]
which is a contradiction.

From \eqref{EQ12} and H\"{o}lder's inequality we get
\begin{align*}
A  &  <\frac{C_{p}}{t_{0}\log^{1/p}(e/t_{0})}\int\limits_{t_{0}^{\frac
{2p}{p-1}}}^{t_{0}^{\frac{p}{p-1}}}f^{\ast}(s)\,ds\leq\frac{C_{p}}{\log
^{1/p}(e/t_{0})}\Big(\int\limits_{t_{0}^{\frac{2p}{p-1}}}^{1}f^{\ast}%
(s)^{p}\,ds\Big)^{1/p}\\
&  \leq C_{p}\sup_{0<t\leq1/2}\frac{1}{{\log^{1/p}(e/t)}}{\Big(\int%
\limits_{t^{\frac{p}{p-1}}}^{1}f^{\ast}(s)^{p}\,ds\Big)^{\frac{1}{p}}},
\end{align*}
where $C_{p}:=2^{1-1/p}M$. Thus,
\begin{align*}
{\Vert f\Vert}_{{\langle L}^{1},{L}^{p}{\rangle}_{F}^{K}}  &  \asymp
\sup_{0<t\leq1/2}\frac{1}{{\log^{1/p}(e/t)}}{\Big(\int\limits_{t^{\frac
{p}{p-1}}}^{1}f^{\ast}(s)^{p}\,ds\Big)^{\frac{1}{p}}}\\
&  \asymp\sup_{0<u\leq1}\frac{1}{{\log^{1/p}(e/u)}}{\Big(\int\limits_{u}%
^{1}f^{\ast}(s)^{p}\,ds\Big)^{\frac{1}{p}}}.
\end{align*}
Combining the last equivalence with ~\eqref{grand Leb}, we arrive at desired result.
\end{proof}

\begin{remark}
It is interesting to note that the approach developed in this paper allows to
get a streamlined proof of other results related to Grand and small Lebesgue
spaces. In particular, the interpolation description of the spaces
$L^{p),\alpha}$, $\alpha>0$, endowed with the norms \footnote{See \cite{GIS}
for the original definition of the spaces $L^{p),\alpha}$ and \cite{DFF} for
the proof of this equivalence. It can be proved also similarly as in the proof
of Theorem \ref{Fiorenza-Karadzhov}.}
\[
{\Vert x\Vert}_{L^{p),\alpha}}:=\sup_{0<\varepsilon<p-1}\varepsilon
^{\frac{\alpha}{p-\varepsilon}}{\Vert x\Vert}_{p-\varepsilon}\asymp
\sup_{0<t<1}\log^{-\alpha/p}(e/t)\left(  \int_{t}^{1}(x^{\ast}(s))^{p}%
\,ds\right)  ^{\frac{1}{p}},
\]
which was obtained in \cite[Theorem 1.1 and Theorem 3.1]{FFGKR}, is an
immediate consequence of Theorems \ref{StrExtrAbst} and \ref{TH2}. Indeed,
proceeding as in \eqref{grand Leb}, we obtain (see also \cite[Lemma
3.2]{FFGKR}) that
\[
L^{p),\alpha}=\langle L^{1},L^{p}\rangle_{F}^{K}%
\]
where the lattice $F:=L^{\infty}(t^{-1}\log^{-\alpha/p}(e/t))$ satisfies the
condition (i) of Theorem \ref{TH2}. Then, from Theorem \ref{TH2},(iii), it
follows that for arbitrary $r\in\lbrack1,p)$ it holds (a result first derived
in \cite[Theorem 3.1]{FFGKR})
\[
L^{p),\alpha}=\langle L^{r},L^{p}\rangle_{F}^{K}.
\]
Furthermore, using the embeddings $L^{1}\supset L^{r),\alpha}\supset L^{s}$,
$1<r<s<p$, we have (cf. \cite[Theorem 1.1]{FFGKR})
\[
L^{p),\alpha}=\langle L^{r),\alpha},L^{p}\rangle_{F}^{K}.
\]

\end{remark}

\section{Lions-Peetre-Pisier formula and vector valued Yano's
theorem\label{sec:vector}}

In this section we consider the problem of proving a vector valued version of
Yano's extrapolation theorem. Our development will be based on Pisier's
approach to the following well-known result due to Lions-Peetre (cf.
\cite{Cw}, \cite{PI} and the references therein) . Let $(\Omega,\mu)$ be a
measure space and let $\vec{A}=(A_{0},A_{1})$ be a Banach pair, then%
\begin{equation}
(L^{1}(\Omega,A_{0}),L^{\infty}(\Omega,A_{1}))_{\theta,p_{\theta}}%
^{K}=L^{p_{\theta}}(\Omega,(A_{0},A_{1})^{K}_{\theta,p_{\theta}}),\text{ with
}0<\theta<1,\frac{1}{p_{\theta}}=1-\theta. \label{pi1}%
\end{equation}

To simplify the discussion in this section we shall further assume that
$A_{0}\cap A_{1}$ is dense in $A_{0}$; then we can write (cf.
\cite[Proposition 5.1.15, page 303]{bs})%
\begin{equation}
K(t,g;\vec{A})=\int_{0}^{t}k(s,g;\vec{A})ds, \label{3}%
\end{equation}
where $k(s,g;\vec{A})$ is \textbf{decreasing}. As is well known \textbf{(}cf.
\cite[Example 7]{M-03} for the scalar case\footnote{Actually letting
$f=k(s,g;\vec{A}),$ it follows that $f^{\ast\ast}(t)=\frac{1}{t}\int_{0}%
^{t}k(u,g;\vec{A})du$ and one can apply the argument in \cite[Example 7]{M-03}
verbatim.}) from (\ref{3}) and reverse Hardy inequalities we readily
get\footnote{cf. (\ref{constant}) above for the definition of $c_{\theta
,p_{\theta}}.$} that for $\theta\in(0,1),\frac{1}{p_{\theta}}=1-\theta,$ and
all $g\in\vec{A}_{\theta,p_{\theta}}^{K\blacktriangleleft},$%
\begin{align}
\left\Vert f\right\Vert _{\vec{A}_{\theta,p_{\theta}}^{K\blacktriangleleft}}
&  =c_{\theta,p_{\theta}}\left\{  \int_{0}^{\infty}\left(  s^{-\theta
}K(s,f;\vec{A})\right)  ^{p_{\theta}}\frac{ds}{s}\right\}  ^{1/p_{\theta}%
}\nonumber\\
&  \asymp\left\{  \int_{0}^{\infty}\left(  s^{1-\theta}k(s,f;\vec{A})\right)
^{p_{\theta}}\frac{ds}{s}\right\}  ^{1/p_{\theta}}, \label{sete}%
\end{align}
where the universal constants of equivalence on the right-hand side are
independent of $\theta.$ As a consequence,%
\begin{equation}
\left\Vert f\right\Vert _{L^{p_{\theta}}(\Omega,\left(  A_{0},A_{1}\right)
_{\theta,p_{\theta}}^{K\blacktriangleleft})}\asymp\left\{  \int_{\Omega
}\left\{  \int_{0}^{\infty}\left(  s^{1-\theta}k(s,f(w);\vec{A})\right)
^{p_{\theta}}\frac{ds}{s}\right\}  d\mu(w)\right\}  ^{1/p_{\theta}}.
\label{lpv}%
\end{equation}
Pisier's method makes it possible to relate (\ref{lpv}) to the $K-$functional
for the pair $(L^{1}(\Omega,A_{0}),L^{\infty}(\Omega,A_{1})),$ and crucially
for our purposes, allows us to keep track of the constants in the intervening
inequalities. Pisier's formula for the $K-$functional for the pair
$(L^{1}(\Omega,A_{0}),L^{\infty}(\Omega,A_{1}))$ is given by%
\begin{equation}
K(t,f;L^{1}(\Omega,A_{0}),L^{\infty}(\Omega,A_{1}))=\sup_{\phi\geq0,\left\Vert
\phi\right\Vert _{L^{1}}\leq t}\int_{\Omega}K(\phi(w),f(w);\vec{A})d\mu(w).
\label{1}%
\end{equation}
Given $f\in L^{1}(\Omega,A_{0})+L^{\infty}(\Omega,A_{1}),$ we let $\Psi
_{f}:\Omega\times(0,\infty)\rightarrow R_{+},$ be defined by $\Psi
_{f}(w,s)=k(s,f(w);\vec{A}),$ $(w,s)\in\Omega\times(0,\infty).$ Pisier's
method is based on reinterpreting formula (\ref{1}) as follows

\begin{proposition}
(Pisier \cite{PI})
\end{proposition}

\begin{equation}
K(t,f;L^{1}(\Omega,A_{0}),L^{\infty}(\Omega,A_{1}))=K(t,\Psi_{f};L^{1}%
(\Omega\times(0,\infty)),L^{\infty}(\Omega\times(0,\infty)). \label{2}%
\end{equation}

\begin{proof}
It will be useful to give complete details of Pisier's proof of (\ref{2}).
First, by Fubini's theorem, we have
\[
(L^{1}(\Omega\times(0,\infty)),L^{\infty}(\Omega\times(0,\infty))=(L^{1}%
(\Omega,(L^{1}(0,\infty)),L^{\infty}(\Omega,(L^{\infty}(0,\infty))).
\]
Therefore we can use (\ref{1}), to write
\begin{align*}
K(t,\Psi_{f};  &  L^{1}(\Omega\times(0,\infty)),L^{\infty}(\Omega
\times(0,\infty))\\
&  =K(t,\Psi_{f};L^{1}(\Omega,(L^{1}(0,\infty)),L^{\infty}(\Omega,(L^{\infty
}(0,\infty))\\
&  =\sup_{\phi\geq0,\left\Vert \phi\right\Vert _{L^{1}}\leq t}\int_{\Omega
}K(\phi(w),\Psi_{f}(w,\cdot);L^{1}(0,\infty),L^{\infty}(0,\infty))d\mu(w).
\end{align*}
Now, by the known Peetre-Oklander formula for the $K-$functional of the pair
\newline$(L^{1}(0,\infty),L^{\infty}(0,\infty)),$ the fact that in the
representation of a $K-$functional given by (\ref{3}), $k(s,g;\vec{A})$ is
decreasing, we can continue with
\begin{align*}
&  =\sup_{\phi\geq0,\left\Vert \phi\right\Vert _{L^{1}}\leq t}\int_{\Omega
}\int_{0}^{\phi(w)}\Psi_{f}(w,.)^{\ast}(r)drd\mu(w)\\
&  =\sup_{\phi\geq0,\left\Vert \phi\right\Vert _{L^{1}}\leq t}\int_{\Omega
}\int_{0}^{\phi(w)}\Psi_{f}(w,s)dsd\mu(w)\\
&  =\sup_{\phi\geq0,\left\Vert \phi\right\Vert _{L^{1}}\leq t}\int_{\Omega
}K(\phi(w),f(w);\vec{A})d\mu(w)\text{ (by (\ref{3})}\\
&  =K(t,f;L^{1}(\Omega,A_{0}),L^{\infty}(\Omega,A_{1})).\text{ (by
(\ref{1})).}%
\end{align*}

\end{proof}

As a consequence, we obtain the Lions-Peetre-Pisier formula with constants:

\begin{corollary}
\label{coro:marcao}$(L^{1}(\Omega,A_{0}),L^{\infty}(\Omega,A_{1}%
))_{\theta,p_{\theta}}^{K\blacktriangleleft}=L^{p_{\theta}}(\Omega
,(A_{0},A_{1})_{\theta,p_{\theta}}^{K\blacktriangleleft}).$
\end{corollary}

\begin{proof}
Our only contribution here is that we keep track of the constants to be able
to introduce the symbol $\blacktriangleleft$ in the formula. Observe that by
the usual interpolation formulae for the pair $(L^{1},L^{\infty})$ (cf. also
(\ref{sete})) and Fubini's theorem, we have, with norm equivalence independent
of $\theta,$%
\begin{align*}
(L^{1}(\Omega\times(0,\infty)),L^{\infty}(\Omega\times(0,\infty)))_{\theta
,p_{\theta}}^{K\blacktriangleleft}  &  =L^{p_{\theta}}(\Omega\times
(0,\infty))\\
&  =L^{p_{\theta}}(\Omega,L^{p_{\theta}}(0,\infty)).
\end{align*}
We combine this fact with (\ref{2}) to find that
\begin{align}
\left\Vert f\right\Vert  &  _{(L^{1}(\Omega,A_{0}),L^{\infty}(\Omega
,A_{1}))_{\theta,p_{\theta}}^{K\blacktriangleleft}}\nonumber\\
&  =c_{\theta,p_{\theta}}\left\{  \int_{0}^{\infty}\left(  s^{-\theta}%
K(s,\Psi_{f};L^{1}(\Omega\times(0,\infty)),L^{\infty}(\Omega\times
(0,\infty))\right)  ^{p_{\theta}}\frac{ds}{s}\right\}  ^{1/p_{\theta}%
}\nonumber\\
&  \asymp\left\Vert \Psi_{f}\right\Vert _{L^{p_{\theta}}(\Omega,L^{p_{\theta}%
}(0,\infty))}\nonumber\\
&  =\left(  \int_{\Omega}\left\{  \int_{0}^{\infty}\left(  s^{1-\theta}%
\Psi_{f}(s,w\right)  ^{p_{\theta}}\frac{ds}{s}\right\}  d\mu(w)\right)
^{1/p_{\theta}}\text{ (since }(1-\theta)p_{\theta}=1\text{).} \label{aba}%
\end{align}
Using the definition of $\Psi_{f}$ \ we see that (\ref{lpv}) states that%
\begin{equation}
\left\Vert f\right\Vert _{L^{p_{\theta}}(\Omega,(A_{0},A_{1})_{\theta
,p_{\theta}}^{K\blacktriangleleft})}\asymp\left(  \int_{\Omega}\left\{
\int_{0}^{\infty}\left(  s^{1-\theta}\Psi_{f}(s,w\right)  ^{p_{\theta}}%
\frac{ds}{s}\right\}  d\mu(w)\right)  ^{1/p_{\theta}}. \label{baba}%
\end{equation}
The desired result now follows comparing (\ref{aba}) and (\ref{baba}).
\end{proof}

Obviously the same method can be used to find versions of (\ref{pi1}) that are
valid in the limiting cases.

\begin{corollary}
Let $F$ be a lattice on $[0,1].$ Then%
\[
\langle L^{1}(\Omega,A_{0}),L^{\infty}(\Omega,A_{1})\rangle^{K}_{F}%
=\{f:\Psi_{f}\in\langle L^{1}(\Omega\times(0,\infty)),L^{\infty}(\Omega
\times(0,\infty))\rangle^{K}_{F}\},
\]
isometrically, i.e.
\[
\left\Vert f\right\Vert _{\langle L^{1}(\Omega,A_{0}),L^{\infty}(\Omega
,A_{1})\rangle^{K}_{F}}=\left\Vert \Psi_{f}\right\Vert _{\langle L^{1}%
(\Omega\times(0,\infty)),L^{\infty}(\Omega\times(0,\infty))\rangle^{K}_{F}}.
\]

\end{corollary}

In the next set of results we use Example \ref{sumexample} from Introduction;
so it will be useful to provide the details here.

\begin{example}
\label{sumsum}For any measure space $(\Omega,\mu)$ we have
\begin{equation}
{\displaystyle\sum\limits_{\theta}}\frac{1}{\theta}(L^{1}(\Omega),L^{\infty
}(\Omega))_{\theta,q}^{K\blacktriangleleft}=\langle L^{1}(\Omega),L^{\infty
}(\Omega)\rangle_{0,1}^{K}=L(LogL)(\Omega)+L^{\infty}(\Omega).
\label{determinada}%
\end{equation}

\end{example}

\begin{proof}
Since $K(t,f;L^{1},L^{\infty})=tf^{\ast\ast}(t),$ it follows from (\ref{sum3})
that
\[%
{\displaystyle\sum\limits_{\theta}}
\frac{1}{\theta}(L^{1},L^{\infty})_{\theta,q}^{K\blacktriangleleft}=\{f:%
{\displaystyle\int\limits_{0}^{\text{ }1}}
f^{\ast\ast}(s)\,ds\}<\infty.
\]
On the other hand, the $K-$functional for the pair $(L(LogL)(\Omega
),L^{\infty}(\Omega))$ is given by (cf. \cite{BE})%
\[
K(t,f;L(\log L),L^{\infty})\asymp\left\Vert f^{\ast}\chi_{(0,\phi^{-1}%
(t))}\right\Vert _{L(LogL)},
\]
where $\phi^{-1}(t)$ is the inverse of the fundamental function of $L(LogL).$
Without loss of generality we can modify $\phi$ so that $\phi^{-1}(1)=1.$ It
follows that%
\begin{align*}
\left\Vert f\right\Vert _{L(\log L)+L^{\infty}}  &  =K(1,f;L(\log
L),L^{\infty})\\
&  \asymp\left\Vert f^{\ast}\chi_{(0,1)}\right\Vert _{L(LogL)}\\
&  \asymp\int_{0}^{1}f^{\ast}(s)(1+\log\frac{1}{s})\,ds\\
&  \asymp\int_{0}^{1}f^{\ast}(s)ds+\int_{0}^{1}f^{\ast\ast}(s)\,ds\\
&  \asymp\int_{0}^{1}f^{\ast\ast}(s)\,ds.
\end{align*}

\end{proof}

\begin{example}
Let $w_{\alpha}(t)=\left(  \log\frac{e}{t}\right)  ^{\alpha},$ $\alpha>0$, and
let $F_{\alpha}=L^{1}((0,1),w_{\alpha}(t)\frac{dt}{t}),$ then%
\[
\langle L^{1}(\Omega,A_{0}),L^{\infty}(\Omega,A_{1})\rangle^{K}_{F_{\alpha-1}%
}=\{f:\Psi_{f}\in L(LogL)^{\alpha}(\Omega\times(0,\infty))+L^{\infty}%
(\Omega\times(0,\infty))\},
\]
and%
\begin{align*}
\left\Vert f\right\Vert _{\langle L^{1}(\Omega,A_{0}),L^{\infty}(\Omega
,A_{1})\rangle^{K}_{F_{\alpha-1}}}  &  \asymp\left\Vert \Psi_{f}\right\Vert
_{L(LogL)^{\alpha}(\Omega\times(0,\infty))+L^{\infty}(\Omega\times(0,\infty
))}\\
&  =\left\Vert k(s,f(w);\vec{A})\right\Vert _{L(LogL)^{\alpha}(\Omega
\times(0,\infty))+L^{\infty}(\Omega\times(0,\infty))}.
\end{align*}

\end{example}

\begin{proof}
For any measure space $\Theta,$ and for all $\alpha>0,$ a slight
generalization of Example \ref{sumsum} yields that
\begin{align}
\left\Vert g\right\Vert _{L(\log L)^{\alpha}(\Theta)+L^{\infty}(\Theta)}  &
=K(1,g;L(\log L)^{\alpha}(\Theta),L^{\infty}(\Theta))\nonumber\\
&  \asymp\int_{0}^{1}K(s,g;L^{1}(\Theta),L^{\infty}(\Theta))\log^{\alpha
-1}(\frac{1}{s})\frac{ds}{s}\label{4}\\
&  =\left\Vert g\right\Vert _{\langle L^{1}(\Theta),L^{\infty}(\Theta
)\rangle^{K}_{F_{\alpha-1}}}.\nonumber
\end{align}
Let $f\in\langle L^{1}(\Omega,A_{0}),L^{\infty}(\Omega,A_{1})\rangle
^{K}_{F_{\alpha-1}}.$ Then the desired result obtains letting $\Theta
=\Omega\times(0,\infty),$ and $g=\Psi_{f}$.
\end{proof}

\begin{example}
\label{exa:marcao}Let $\vec{A}$ be an ordered pair. Suppose that $\mu
(\Omega)<\infty.$ Then \newline$(L^{1}(\Omega,A_{0}),L^{\infty}(\Omega
,A_{1}))$ is an ordered pair, and since $k(s,g;\vec{A})=0$ for $s>1,$ using
the notation of the previous example we can write%
\[
\langle L^{1}(\Omega,A_{0}),L^{\infty}(\Omega,A_{1})\rangle^{K}_{F_{\alpha-1}%
}=\{f:\Psi_{f}\in L(LogL)^{\alpha}(\Omega\times(0,1))+L^{\infty}(\Omega
\times(0,1))\}.
\]
Since in this case we have that
\[
L(LogL)^{\alpha}(\Omega\times(0,1))+L^{\infty}(\Omega\times
(0,1))=L(LogL)^{\alpha}(\Omega\times(0,1)),
\]
it follows that%
\[
\langle L^{1}(\Omega,A_{0}),L^{\infty}(\Omega,A_{1})\rangle^{K}_{F_{\alpha-1}%
}=\{f:\Psi_{f}\in L(LogL)^{\alpha}(\Omega\times(0,1))\}.
\]

\end{example}

Now, we show how the results of this section can be used to extend the
classical extrapolation theorems of Yano. A more detailed exposition will be
given elsewhere (cf. \cite{ALM}). To simplify the discussion we work with
finite measure spaces. From the point of view of extrapolation theory Yano's
theorem is simply the statement that (there is an analogous formula for the
$\Delta-$method, which will not be considered here)%
\[
\sum\limits_{p>1}\frac{L^{p}(\Omega)}{(p-1)^{\alpha}}=L(LogL)^{\alpha}%
(\Omega).
\]

\begin{theorem}
Let $\vec{A}$ be an ordered pair, and for $f\in L^{1}(\Omega,A_{0})+L^{\infty
}(\Omega,A_{1})$ let $\Psi_{f}(s,w)=k(s,f(w);\vec{A}),$ $(w,s)\in\Omega
\times(0,1).$ Then,%
\[
\sum_{\theta}\frac{L^{p_{\theta}}(\Omega,(A_{0},A_{1})_{\theta,p_{\theta}%
}^{K\blacktriangleleft})}{\theta^{\alpha}}=\{f:\Psi_{f}\in L(LogL)^{\alpha
}(\Omega\times(0,1))\}.
\]

\end{theorem}

\begin{proof}
By Corollary \ref{coro:marcao}%
\[
\sum_{\theta}\frac{L^{p_{\theta}}(\Omega,(A_{0},A_{1})_{\theta,p_{\theta}%
}^{K\blacktriangleleft})}{\theta^{\alpha}}=\sum_{\theta}\frac{(L^{1}%
(\Omega,A_{0}),L^{\infty}(\Omega,A_{1}))_{\theta,p_{\theta}}%
^{K\blacktriangleleft}}{\theta^{\alpha}}%
\]
Now by \cite{JM91}%
\[
\sum_{\theta}\frac{(L^{1}(\Omega,A_{0}),L^{\infty}(\Omega,A_{1}))_{\theta
,p_{\theta}}^{K\blacktriangleleft}}{\theta^{\alpha}}=\langle L^{1}%
(\Omega,A_{0}),L^{\infty}(\Omega,A_{1})\rangle_{F_{\alpha-1}}%
\]
and by Example \ref{exa:marcao} we can continue with%
\[
=\{f:\Psi_{f}\in L(LogL)^{\alpha}(\Omega\times(0,1))\},
\]
as we wished to show.
\end{proof}

\section{Further applications\label{sec:last}}

\subsection{Weights and $K/J$ equivalence\label{sec:equivalence}}

There is a simple mechanism underlying (\ref{equivJ}) that seems worthwhile to
discuss in detail. Let $w$ be a weight, $w:(0,\infty)\rightarrow
\lbrack0,\infty)$ (e.g., $w(s)=\chi_{(0,1)}(s)$), and let $\vec{A}$ be a
Gagliardo complete pair of Banach spaces. Define
\[
\vec{A}_{w}^{K}:=\{f:f\in A_{0}+A_{1}\text{ s.t. }\left\Vert f\right\Vert
_{\vec{A}_{w}^{K}}<\infty\},
\]
where%
\[
\left\Vert f\right\Vert _{\vec{A}_{w}^{K}}:=\int_{0}^{\infty}K(s,f;\vec
{A})w(s)\,\frac{ds}{s}.
\]
Likewise, we define%
\[
\vec{A}_{w}^{J}:=\{f=\int_{0}^{\infty}u(s)\,\frac{ds}{s},\text{ with
}u:(0,\infty)\rightarrow A_{0}\cap A_{1},\text{ }\left\Vert f\right\Vert
_{\vec{A}_{w}^{J}}<\infty\},
\]
where%
\[
\left\Vert f\right\Vert _{\vec{A}_{w}^{J}}:=\inf_{f=\int_{0}^{\infty
}u(s)\,\frac{ds}{s}}\{\int_{0}^{\infty}J(s,u(s);\vec{A})w(s)\frac{ds}{s}\}.
\]

\begin{lemma}
\label{lemamarkao}Let $w$ be a nonnegative locally summable function on
$(0,\infty)$ such that $\int_{0}^{\infty}w(s)\,ds=\infty$. For each Gagliardo
complete pair $\vec{A}$ such that $A_{0}\cap A_{1}$ is dense in $A_{0}$ it
holds%
\[
\vec{A}_{w}^{K}=\vec{A}_{\tilde{w}}^{J},
\]
where%
\[
\tilde{w}(t)=\int_{0}^{\infty}\min\{1,\frac{s}{t}\}w(s)\frac{ds}{s}.
\]

\end{lemma}

\begin{proof}
Let $f\in$ $\vec{A}_{\tilde{w}}^{J}.$ We can represent $f=\int_{0}^{\infty
}u(s)\frac{ds}{s}$ in such a way that%
\[
\left\Vert f\right\Vert _{\vec{A}_{\tilde{w}}^{J}}\asymp\int_{0}^{\infty
}J(s,u(s);\vec{A})\tilde{w}(s)\frac{ds}{s}.
\]
Using Minkowski's inequality, and the fact that (cf. \cite[Lemma 3.2.1
page~42]{BL}), $K(t,a;\vec{A})\leq\min\{1,\frac{t}{s}\}J(s,a;\vec{A}),$ we
get
\[
K(t,f;\vec{A})\leq\int_{0}^{\infty}\min\{1,\frac{t}{s}\}J(s,u(s);\vec{A}%
)\frac{ds}{s}.
\]
Integrating the last inequality and then using successively Fubini's theorem,
and the definition of $\tilde{w},$ we find,%
\begin{align*}
\int_{0}^{\infty}K(t,f;\vec{A})w(t)\frac{dt}{t}  &  \leq\int_{0}^{\infty}%
\int_{0}^{\infty}\min\{1,\frac{t}{s}\}\int_{0}^{\infty}J(s,u(s);\vec{A}%
)\frac{ds}{s}w(t)\frac{dt}{t}\\
&  =\int_{0}^{\infty}J(s,u(s);\vec{A})\int_{0}^{\infty}\min\{1,\frac{t}%
{s}\}w(t)\frac{dt}{t}\frac{ds}{s}\\
&  =\int_{0}^{\infty}J(s,u(s);\vec{A})\tilde{w}(s)\frac{ds}{s}\\
&  \asymp\left\Vert f\right\Vert _{\vec{A}_{\tilde{w}}^{J}}.
\end{align*}
On the other hand, suppose that $f\in$ $\vec{A}_{\tilde{w}}^{K}.$ Observe that
for all $t>0$ we have%
\begin{align*}
\int_{0}^{\infty}K(s,f;\vec{A})w(s)\,\frac{ds}{s}  &  \geq\int_{0}
^{t}K(s,f;\vec{A})w(s)\,\frac{ds}{s}\\
&  \geq\frac{K(t,f;\vec{A})}{t}\int_{0}^{t}w(s)\,ds. \label{dua2}%
\end{align*}
Hence, since $\int_{0}^{\infty}w(s)\,ds=\infty$ and $A_{0}\cap A_{1}$ is dense
in $A_{0},$ we deduce that%
\[
\lim_{t\to0}K(t,f;\vec{A})=\lim_{t\to\infty}\frac{K(t,f;\vec{A})}{t}=0.
\]
Consequently, by the strong form of the fundamental lemma (cf. \cite{CJM90}
and the references therein), we can find a representation $f=\int_{0}^{\infty
}u(s)\frac{ds}{s},$ such that%
\[
\int_{0}^{\infty}\min\{1,\frac{t}{s}\}J(s,u(s);\vec{A})\frac{ds}{s}\leq\gamma
K(t,f;\vec{A}),
\]
where $\gamma$ is an absolute constant. Multiplying the last inequality by
$w(t)$ and then integrating, by Fubini's theorem, we see that%
\[
\int_{0}^{\infty}J(s,u(s);\vec{A})\tilde{w}(s)\frac{ds}{s}\leq\gamma\left\Vert
f\right\Vert _{\vec{A}_{{w}}^{K}}.
\]
Thus,%
\[
\left\Vert f\right\Vert _{\vec{A}_{\tilde{w}}^{J}}\leq\gamma\left\Vert
f\right\Vert _{\vec{A}_{{w}}^{K}},
\]
as we wished to show.
\end{proof}

In particular, as have pointed out above, for ordered pairs $\vec{A}$ we may
restrict the weights to be defined on the interval $(0,1).$ Then $\tilde
{w}(t)=\int_{0}^{1}\min\{1,\frac{s}{t}\}w(s)\frac{ds}{s}$, $t\in(0,1)$, and we
have
\[
\langle\vec{A}\rangle_{w}^{K}=\langle\vec{A}\rangle_{\tilde{w}}^{J}.
\]
For example, if $w\equiv1,$ then $\tilde{w}(t)=1+\log\frac{1}{t},t\in(0,1).$
More generally, if $w_{\alpha}(t)=\left(  \log\frac{e}{t}\right)  ^{\alpha},$
$\alpha>-1$, then $\tilde{w}(t)\asymp\left(  \log\frac{e}{t}\right)
^{\alpha+1} $, and with norm equivalence we have%
\[
\langle\vec{A}\rangle_{w_{\alpha}}^{K}=\langle\vec{A}\rangle_{w_{\alpha+1}%
}^{J}.
\]
(see Remark~\ref{K/J formulas}).

\begin{example}
\label{exa:bre}The following extension of (\ref{prova1}) holds (cf.
\cite{JM91})%
\[
\left\Vert f\right\Vert _{\langle\vec{A}\rangle_{w_{\alpha}}^{K}}\leq
c\left\Vert f\right\Vert _{A_{0}}\log^{\alpha+1}\left(  e+\frac{\left\Vert
f\right\Vert _{A_{1}}}{\left\Vert f\right\Vert _{A_{0}}}\right)  ,f\in A_{1}.
\]
Inequalities of this type specialize to well-known inequalities in Analysis
(cf. \cite{brg}, \cite{brw}).
\end{example}

\subsection{Limits of interpolation spaces and extrapolation\label{sec:apllim}%
}

In this section we show how limit theorems of the form (\ref{lim1}),
(\ref{lim2}) can be combined with the strong form of the fundamental lemma to
prove extrapolation theorems. For more general results see \cite{ALM}.







Firstly, observe that from the strong form of the fundamental lemma precisely
as in the proof of embedding \eqref{fromprevious} above (cf. \cite{JM91}) it
follows
\begin{equation}
\left\Vert a\right\Vert _{\langle\vec{A}\rangle_{\theta,1}^{J}}\leq
C_{0}\theta(1-\theta)\left\Vert a\right\Vert _{\langle\vec{A}\rangle
_{\theta,1}^{K}}. \label{special ineq}%
\end{equation}

\begin{theorem}
\label{teoyano}(Abstract form of Yano's theorem). Let $\vec{A}$ be an ordered
Gagliardo complete pair. Suppose that $T$ is a bounded linear operator such
that%
\[
T:\,\langle\vec{A}\rangle_{\theta,1}^{J}\rightarrow\langle\vec{A}%
\rangle_{\theta,1}^{J},\text{ with }\left\Vert T\right\Vert _{\langle\vec
{A}\rangle_{\theta,1}^{J}\rightarrow\langle\vec{A}\rangle_{\theta,1}^{J}}%
\leq\frac{C}{\theta},\;\;0<\theta<1.
\]
Then,%
\[
T:\,\langle\vec{A}\rangle_{0,1}^{K}\rightarrow A_{0}.
\]

\end{theorem}

\begin{proof}
Let $f\in A_{1}.$ From the assumptions combined with

the inequality \eqref{special ineq}, we have that, for all $\theta\in(0,1),$
\[
\left\Vert Tf\right\Vert _{\langle\vec{A}\rangle_{\theta,1}^{J}}\leq
\frac{CC_{0}}{\theta}\cdot\theta(1-\theta)\left\Vert f\right\Vert
_{\langle\vec{A}\rangle_{\theta,1}^{K}}.
\]
Now let $\theta\rightarrow0,$ and compute the limits using

(\ref{lim2}) to find%
\[
\left\Vert Tf\right\Vert _{A_{0}}\leq\lim_{\theta\rightarrow0}\left\Vert
Tf\right\Vert _{\langle\vec{A}\rangle_{\theta,1}^{J}}\leq CC_{0}\lim
_{\theta\rightarrow0}\left\Vert f\right\Vert _{\langle\vec{A}\rangle
_{\theta,1}^{K}}\leq2CC_{0}\left\Vert f\right\Vert _{\langle\vec{A}%
\rangle_{0,1}^{K}}.
\]
Thus, we have obtained the desired inequality under the assumption that $f\in
A_{1},$ but this restriction can be eliminated on account of the fact that
$A_{1}$ is dense in $\langle\vec{A}\rangle_{0,1}^{K}$ (cf. Lemma
\ref{lemamarkao}).
\end{proof}

\subsection{Limiting Spaces with \textquotedblleft broken logarithms"}

In the literature of limiting spaces one finds the so called \textquotedblleft
spaces with broken logarithms\textquotedblright. In this section we wish to
point out how spaces of this type appear in extrapolation theory. We shall be
brief and refer to \cite{AL2017} and \cite{JM91} for details.

Let $\vec{A}$ be a Gagliardo complete Banach pair, and let $M(\theta)$ be a
weight defined on $(0,1).$ One can consider very general classes of weights
(cf. \cite{AL2017}) but here we shall only discuss the following example%
\[
M(\theta)=\theta^{-\alpha_{0}}(1-\theta)^{-\beta_{0}},\alpha_{0},\beta_{0}%
\geq0.
\]
It is then shown in \cite[Corollary 3.5, page 24, Example 3.15 page
31-32]{JM91} that if $T$ is a bounded linear operator, such that $T:\bar
{A}_{\theta,1}^{J}\rightarrow\bar{A}_{\theta,\infty}^{K},$ with norm
$\left\Vert T\right\Vert _{\bar{A}_{\theta,1}^{J}\rightarrow\bar{A}%
_{\theta,\infty}^{K}}\le C M(\theta),$ then%
\[
K(t,Tf;\vec{A})\le C^{\prime}\int_{0}^{\infty}\left[  \left(  \log^{+}\frac
{t}{s}\right)  ^{\alpha_{0}-1}+\left(  \log^{+}\frac{s}{t}\right)  ^{\beta
_{0}-1}\right]  K(s,f;\vec{A})\,\frac{ds}{s}.
\]
This is also directly connected with the strong form of the fundamental lemma
and with the equivalence \cite[Corollary 3.5, page 24]{JM91}%

\begin{align*}%
{\displaystyle\sum}
M(\theta)\vec{A}_{\theta,q}^{K\blacktriangleleft}  &  =%
{\displaystyle\sum}
M(\theta)\vec{A}_{\theta,q}^{J\blacktriangleleft}\\
&  =\{f:\int_{0}^{\infty}K(\frac{t}{r},f;\bar{A})\,d\mu(r)<\infty\},
\end{align*}
and where $\mu$ is the measure representing $\tau(t):=\inf_{0<\theta
<1}\{M(\theta)t^{\theta}\},$ that is,
\[
\tau(t)=\inf_{0<\theta<1}\{M(\theta)t^{\theta}\}=\int_{0}^{\infty}%
\min\{1,\frac{t}{r}\}\,d\mu(r),\;\;0<t<\infty.
\]

This discussion provides a motivation for the interest in the spaces with
\textquotedblleft broken powers" and \textquotedblleft broken logarithms"
defined in the literature.

\subsection{Limits, recovery of end points, and
extrapolation\label{sec:peelim}}

One implication of (\ref{lim1}), (\ref{lim2}) is that, for the real methods,
we can reverse the process of interpolation in the sense that if we know the
intermediate norms then we can recover the initial norms by taking suitable
limits. On the other hand, the $\Delta-$ and $\Sigma-$extrapolation functors
can be also used to recover the $K$- and $J$-functionals, from where we can,
once again, recover the initial norms. These methods can be used to complement
some classical inequalities in Analysis.

Let $\{F_{\theta}\}_{\theta\in(0,1)}$ be a family of interpolation functors,
which are exact of exponent $\theta$ (as examples of such families we mention
the normalized real interpolation methods $\{(\cdot,\cdot)_{\theta
,q}^{K\blacktriangleleft}\}_{\theta\in(0,1)},\{(\cdot,\cdot)_{\theta
,q}^{J\blacktriangleleft}\}_{\theta\in(0,1)},$ as well as the complex method
$\{[\cdot,\cdot]\}_{\theta\in(0,1)}$), then (cf. \cite[page 15]{JM91}) for all
Gagliardo complete pairs $\vec{A}$, $t>0,f\in A_{0}\cap A_{1},$ we have,
uniformly,%
\begin{equation}
\left\Vert f\right\Vert _{\Delta(t^{\theta}F_{\theta}(\vec{A}))}\asymp
J(t,f;\vec{A}), \label{pee1}%
\end{equation}
where by definition $\left\Vert {\cdot}\right\Vert _{t^{\theta}F_{\theta}%
(\vec{A})}=t^{\theta}\left\Vert {\cdot}\right\Vert _{F_{\theta}(\vec{A})}.$
Therefore,%
\[
\lim_{t\rightarrow0}\left\Vert f\right\Vert _{\Delta(t^{\theta}F_{\theta}%
(\vec{A}))}\asymp\left\Vert f\right\Vert _{A_{0}},\lim_{t\rightarrow\infty
}\frac{\left\Vert f\right\Vert _{\Delta(t^{\theta}F_{\theta}(\vec{A}))}}%
{t}\asymp\left\Vert f\right\Vert _{A_{1}}.
\]
Likewise, we can \textquotedblleft recover" the $K-$functional as follows (cf.
\cite[page 15]{JM91})%
\begin{equation}
\left\Vert f\right\Vert _{\sum\limits_{\theta}(t^{\theta}F_{\theta}(\vec{A}%
))}\asymp K(t,f;\vec{A}). \label{pee2}%
\end{equation}
As a consequence we deduce that%
\[
\lim_{t\rightarrow\infty}\left\Vert f\right\Vert _{\sum\limits_{\theta
}(t^{\theta}F_{\theta}(\vec{A}))}\asymp\left\Vert f\right\Vert _{A_{0}%
},\;\;\lim_{t\rightarrow0}\frac{\left\Vert f\right\Vert _{\sum\limits_{\theta
}(t^{\theta}F_{\theta}(\vec{A}))}}{t}\asymp\left\Vert f\right\Vert _{A_{1}}.
\]

From the previous discussion we see that if the pair $\vec{A}$ is ordered
then, letting $t=1$ in (\ref{pee1}) and (\ref{pee2}), we have%
\[
\left\Vert f\right\Vert _{\Delta(F_{\theta}(\vec{A}))}\asymp\left\Vert
f\right\Vert _{A_{1}},\left\Vert f\right\Vert _{\sum\limits_{\theta}%
(F_{\theta}(\vec{A}))}\asymp\left\Vert f\right\Vert _{A_{0}}.
\]

As an application we show a connection to the celebrated
Bourgain-Brezis-Mironescu-Maz'ya-Shaposhnikova Sobolev-Besov inequalities (cf.
\cite{bbm1}, \cite{leo}, \cite{PO}, and the references therein).

\begin{example}
For more background information concerning this example we refer to \cite{bs},
\cite{leo}. For succinctness for each $(s,r)\in(0,1)\times(1,\infty)$ we let
$\mathring{W}^{s,r}(\mathbb{R}^{n}):=(L^{r}(\mathbb{R}^{n}),\mathring{W}%
_{r}^{1}(\mathbb{R}^{n}))_{s,r}^{K\blacktriangleleft}$. Then, from
(\ref{pee2}) we see that for all $t>0,$
\[
\left\Vert f\right\Vert _{\sum\limits_{s}(t^{s}\mathring{W}^{s,r}%
(\mathbb{R}^{n}))}\asymp K(t,f;L^{r}(\mathbb{R}^{n}),\mathring{W}_{r}%
^{1}(\mathbb{R}^{n}))\asymp w_{r,f}(t),
\]
where $w_{r,f}(t)$ is the $r-$modulus of continuity of $f$ at $t$ (cf.
Introduction and \cite[Exercise 13 (b), page 431]{bs}). As another simple
application, let $p,q\in(1,\infty)$ be fixed and suppose that $T$ is a bounded
linear operator such that for all $s\in(0,1),$ $T:\mathring{W}^{s,p}%
(\mathbb{R}^{n})\rightarrow\mathring{W}^{s,q}(\mathbb{R}^{n}),$ with
$\left\Vert T\right\Vert _{\mathring{W}^{s,p}\rightarrow\mathring{W}^{s,q}%
}\leq C,$ independent of $s.$ Then, from (\ref{pee2}) (i.e., by extrapolation)
we have%
\begin{equation}
w_{q,Tf}(t)\leq Cw_{p,f}(t), \label{modulo}%
\end{equation}
where $C$ is an absolute constant.
\end{example}

\end{document}